\documentclass[11pt]{amsart}

\usepackage{latexsym}
\usepackage{amssymb}
\usepackage{amsmath}
\usepackage{color}
\usepackage{comment}

\newtheorem{theorem}{Theorem}[section]
\newtheorem{lemma}[theorem]{Lemma}
\newtheorem{proposition}[theorem]{Proposition}
\newtheorem{corollary}[theorem]{Corollary}
\newtheorem{definition}[theorem]{Definition}

\newtheorem{remark}[theorem]{Remark}

\newtheorem{question}[theorem]{Question}

\newcommand\id{\mathop{\rm id}}
\newcommand\tr{\mathop{\rm tr}}
\newcommand\Tr{\mathop{\rm Tr}}
\newcommand\dd{\mathop{\rm d}}
\newcommand\mm{\mathop{\rm m}}

\newcommand\nph{\varphi}
\newcommand\nep{\epsilon}

\newcommand\clas{\mathop{\rm cl}}

\newcommand{\cl}[1]{\mathcal{#1}}
\newcommand{\bb}[1]{\mathbb{#1}}

\begin{document}

\title{Synchronicity for quantum non-local games}

\author[M. Brannan]{Michael Brannan} 
\address{Department of Mathematics, Texas A\& M University, Mailstop 3368, College Station, TX 77843-3368 USA.
}
\email{mbrannan@tamu.edu}

\author[S.J. Harris]{Samuel J. Harris} 
\address{Department of Mathematics, Texas A\& M University, Mailstop 3368, College Station, TX 77843-3368 USA.
}
\email{sharris@tamu.edu}

\author[I. G. Todorov]{Ivan G. Todorov}
\address{
School of Mathematical Sciences, University of Delaware, 501 Ewing Hall,
Newark, DE 19716, USA, and 
Mathematical Sciences Research Centre,
Queen's University Belfast, Belfast BT7 1NN, United Kingdom}
\email{todorov@udel.edu}

\author[L. Turowska]{Lyudmila Turowska}
\address{Department of Mathematical Sciences, Chalmers University
of Technology and the University of Gothenburg, Gothenburg SE-412 96, Sweden}
\email{turowska@chalmers.se}

\date{12 June 2021}

\begin{abstract}
We introduce concurrent quantum non-local games, 
quantum output mirror games and concurrent classical-to-quantum non-local games,
as quantum versions of synchronous non-local games, 
and provide tracial characterisations of their perfect strategies belonging to  various correlation classes. 
We define *-algebras and C*-algebras of concurrent classical-to-quantum and concurrent quantum non-local games, 
and algebraic versions of the orthogonal rank of a graph. 
We show that quantum homomorphisms of quantum graphs can be viewed as entanglement assisted 
classical homomorphisms of the graphs, and give descriptions of the perfect quantum commuting and the perfect 
approximately quantum strategies for the quantum graph homomorphism game. 
We specialise the latter results to the case where the inputs of the game are based on a classical graph.
\end{abstract}

\maketitle

\tableofcontents


\section{Introduction}

Over the past decade, the theory of non-local games has undergone a flurry of development and  is now a fundamental branch of modern quantum information theory, with deep applications to many areas of mathematics, physics, and computer science, including operator algebras, non-commutative geometry, quantum non-locality, entanglement, and quantum complexity theory.  Mathematically, a (two-player) non-local game consists of a tuple $\mathcal G = (X,Y,A,B,\lambda)$, where $X,Y, A,B$ are finite sets, and $\lambda:X \times Y \times A \times B \to \{0,1\}$ is a function.  The game is played cooperatively by two spatially separated non-communicating players, Alice and Bob, against a referee. 
During each round of the game,  the referee samples a pair of ``questions'' $(x,y) \in X \times Y$, and sends question $x$ to Alice, and quesiton $y$ to Bob.  Alice is then required to supply an ``answer'' $a \in A$, and Bob -- 
an answer $b \in B$,
to the referee. Alice and Bob win the round of the game if and only if the rule function $\lambda$ 
evaluates to $1$ on this question-answer combination, that is, if the condition $\lambda(x,y,a,b) = 1$ is satisfied. 

The fact that the players Alice and Bob are not allowed to communicate during play makes it difficult to win each round of a non-local game with high probability.  On the other hand, it is precisely this nature of non-local games that makes them interesting as both theoretical and practical tools in quantum information.  The idea here is that, in certain scenarios, Alice and Bob can utilise the phenomenon of quantum entanglement to help correlate their answers in a much stronger way than what the resources of classical physics allow. 

A prototypical example of a non-local game is the {\it graph homomorphism game}:  Given a pair of finite simple graphs $G$ and $H$ with vertex sets $V(G)$, $V(H)$ and edge sets $E(G)$, $E(H)$, respectively, the $(G,H)$-homomorphism game is the non-local game $\mathcal G$ with $X = Y = V(G)$, $A=B = V(H)$ and $\lambda(x,y, a,b)=0$ if either (i) $x=y$  and $a \ne b$ or (ii) $(x,y) \in E(G)$ and $(a,b) \notin E(H)$.  Clearly the graph homomorphism game captures, in the operational language of non-local games, the notion of a graph homomorphism $G \to H$: Any winning strategy for this game would serve to convince an observer that there exists such a graph homomorphism $G \to H$.  

Graph homomorphism games form an interesting class of non-local games for several reasons.  
First, they give rise to quantum analogues of graph parameters, including 
quantum chromatic numbers and quantum independence numbers \cite{man-rob, psstw}. 
These parameters can be genuinely different than the corresponding classical versions, 
thus providing new manifestations of the fundamental Bell Theorem.
Second, they provide some of the simplest examples of {\it pseudo-telepathy games} -- ones which can be perfectly won only with the help of quantum entanglement as a resource \cite{man-rob, dpp, psstw}.  Third, and perhaps most importantly, graph homomorphism games belong to the particularly important class of {\it synchronous} non-local games introduced in 
\cite{psstw} (see also \cite{hmps}).
Recall that a non-local game $\mathcal G = (X,Y,A,B,\lambda)$ is called sychronous if $X = Y$, $A = B$, and $\lambda(x,x,a,b) = 0$ for all $x \in X$ and $a \ne b \in A$.  This means that in order for Alice and Bob to win a round of $\mathcal G$, they must ``sychronise'' their answers whenever they both receive the same question from the referee.  This seemingly innocuous constraint on a game $\mathcal G$ turns out to have very interesting quantum information theoretic and operator algebraic consequences.  For example, the problem of finding perfect quantum strategies for a synchronous game $\mathcal G$ amounts to finding tracial states on a certain {\it game $\ast$-algebra} $\mathcal A(\mathcal G)$ associated to $\mathcal G$ \cite{hmps}.  The algebras $\mathcal A(\mathcal G)$ play the role of a non-commutative analogue of the algebras of coordinate functions on spaces of perfect deterministic (classical) strategies for $\mathcal G$, and are therefore of significant interest from several perspectives in non-commutative geometry,  quantum groups \cite{soltan,bcehpsw}, and von Neumann algebra theory \cite{jnvwy}.  
It follows from the breakthrough work \cite{jnvwy} that there exists a synchronous non-local game 
$\mathcal G$ whose game $\ast$-algebra $\mathcal A(\mathcal G)$ admits a tracial state $\tau$ for which the generated von Neumann algebra $M = \pi_\tau(\mathcal A(\mathcal G))''$ fails to embed into an ultraproduct of the hyperfinite II$_1$-factor --  yielding a(n albeit non-constructive) counter-example to the Connes Embedding Problem in operator algebras and to the 
equivalent \cite{jnpp} strong Tsirelson Problem in quantum physics. 

The purpose of the present paper is to introduce and study generalisations of synchronous non-local games within the framework of {\it quantum non-local games} -- non-local games where the questions and answers are allowed to be quantum states, or possibly mixtures of classical and quantum states.  
In this paper, we use the language of quantum no-signalling (QNS) correlations and quantum non-local games recently introduced by two of the present authors \cite{tt-QNS}.  Classically, in the course of a non-local game 
$\mathcal G = (X,Y,A,B, \lambda)$, Alice and Bob's behaviour
is described by a family $p = (p(a,b|x,y))_{(a,b,x,y) \in A \times B \times X \times Y}$ of conditional 
probability distributions, which can, in a canonical way, be viewed as a noisy information channel 
$\cl N : X \times Y \to A \times B$ with well-defined marginal channels.
In the quantum setting, one replaces the classical state spaces $X, Y, A, B$ by their quantum analogues (i.e. the Hilbert spaces $\mathbb C^{|X|}, \mathbb C^{|Y|},$ etc.), and the classical channel $\cl N : X \times Y \to A \times B$ 
by a {\it quantum channel} $\Gamma: M_X \otimes M_Y \to M_A \otimes M_B$, where, for any finite set $Z$, 
we have let $M_Z = \cl B(\mathbb C^{|Z|})$ be the matrix algebra of linear maps on $\mathbb C^{|Z|}$.  
In this framework, the rule function $\lambda$ can be generalized by replacing it with a zero-preserving, join-preserving mapping $\varphi$ from the projection lattice on $\mathcal P_{XY}$ in $M_X \otimes M_Y$ to the projection lattice $\mathcal P_{AB}$ in $M_{A} \otimes M_{B}$.  A winning strategy for a quantum non-local game $\varphi:\mathcal P_{XY} \to \mathcal P_{AB}$ is then given by a QNS correlation $\Gamma$ satisfying the trace-orthogonality relation 
\[
\langle \Gamma(P), \varphi(P)_\perp \rangle = 0, \qquad P \in \mathcal P_{XY};
\]
the latter condition constrains the supports of the output states of $\Gamma$ according to the supports of
its input states (see Section \ref{ss_qnscor} for further motivation and details).

We note that non-local games with quantum inputs and/or outputs have been previously studied
in \cite{cjppg} and \cite{rv}. The strategies used in the latter papers are the elements from the quantum 
QNS correlation class. Since our main interest lies in the characterisation of the perfect strategies of a game
and their applications, we have adopted the present approach, where we only specify 
the rules of the non-local game, without fixing 
a probability distribution on the questions (or a quantum version thereof).

One of our main achievements in the present work is the introduction of quantum analogues of synchronous non-local games (called herein {\it concurrent quantum games}), as well as classical input-quantum output versions of 
the mirror games introduced in \cite{lmprsstw}.  Classically, synchronous games form a special class of mirror games, and both of these classes of games have the remarkable property that ``Alice's quantum behaviour completely determine Bob's quantum behaviour'' when considering perfect strategies for the games; moreover, 
such perfect strategies 
can always be described in terms of correlations coming from tracial states on a particular game algebra.    
We show that such a paradigm persists in the quantum case by associating  
*-algebras and C*-algebras to concurrent quantum and to concurrent classical-to-quantum games.
Our main results in this direction 
(cf. Theorem \ref{th_qomg}, Corollary \ref{c_cqt}, Theorem \ref{th_qtnoteq1}, Corollary \ref{c_trcqua}) 
provide an operational interpretation of the tracial QNS correlations introduced in \cite{tt-QNS} in terms of perfect strategies of concurrent and quantum mirror games, and their associated game algebras.

One of our long-term motivations for the present work is to develop tools that may eventually be useful for gaining a better understanding of the work \cite{jnvwy}, which, as mentioned above, implicitly constructs a game $\mathcal G$, whose game algebra is a witness to the failure of the Connes Embedding Problem.  At present, the game constructed in \cite{jnvwy} is not well understood, and involves very large input/output sets.  There is some hope that quantum non-local games may provide additional flexibility in the construction of game algebras with pathological operator algebraic properties.  A particularly interesting and tractable source of examples in this more general framework are the quantum graph homomorphism games.  Quantum graphs have achieved a lot of attention in recent years, as objects that arise in a variety of areas 
(e.g. zero-error quantum information theory, 
quantum error correction, quantum groups, quantum teleportation schemes, and subfactor theory) 
\cite{btw, bcehpsw, mrv1, stahlke, weaver}.  In Section \ref{s_qghg}, we study the quantum graph homomorphism game in detail, extending previous work of the authors \cite{bran-gan-har, tt-QNS} in the classical-quantum hybrid setting, and also making connections with the work of Stahlke \cite{stahlke} and the algebraic work of 
Musto-Reutter-Verdon \cite{mrv1} on quantum graph homomorphisms.

The paper is organised as follows.  
Section \ref{s_prel} introduces some necessary notation and background that will be used throughout the paper.  Section \ref{s_cq} recalls the notions related to QNS correlations and their 
various subclasses (quantum commuting, approximately quantum, quantum, local),
examines in detail the case of classical to quantum non-local games, introducing 
the aforementioned semi-quantised mirror games and concurrent games, and studies 
them as operational realisations of tracial QNS correlations.  
In Section \ref{s_congames}, we consider the fully quantum concurrent games, proving 
tracial characterisations of perfect strategies of these games.  
Finally, in Section \ref{s_qghg}, we focus on the quantum graph homomorphism game, 
and describe connections with the prior work of Stahlke \cite{stahlke}
on entanglement assisted quantum graph homomorphisms, 
as well as with our prior works \cite{bran-gan-har, tt-QNS}.     
We show that the perfect quantum strategies of the quantum graph homomorphism game can, in a rigorous sense, 
be thought of as entanglement assisted perfect classical strategies for this game, and 
extract characterisations of the corresponding quantum commuting and approximately quantum strategies 
in terms of natural inclusion relations relating the two quantum graphs.
Our results are further specialised in the case where the inputs are based on a classical graph, leading to separation results on the algebraic and C$^\ast$-algebraic versions of the orthogonal rank of a graph (cf. Propositions \ref{p_bound} and \ref{alg_orth}).

\subsection*{Acknowledgements} 
It is our pleasure to thank Marius Junge, Carlos Palazuelos and David P\'{e}rez-Garc\'{i}a
for fruitful discussions on the topic of this paper.
M.B. was partially supported by NSF grant DMS-2000331.  S.H. was partially supported by an NSERC Postdoctoral Fellowship.  I.T. was partially supported by 
the Simons Foundation (grant number 708084).

\subsection*{Note on related work} 
After the first draft of this paper was completed, we learnt from Piotr So\l tan that 
characterisations of concurrent correlations from the quantum commuting class, 
closely related to the ones described in Subsection \ref{ss_trdes}, 
were independently obtained by Bochniak-Kasprzak-So\l tan in the recently posted preprint \cite{bks}; 
more specifically, \cite[Theorem 6.6]{bks} generalises the first statement within Theorem \ref{th_qtnoteq1} in the present paper.


\section{Preliminary notions and results}\label{s_prel}

For a finite set $X$, let $M_X$ be the algebra of all complex matrices indexed by $X\times X$; 
we identify $M_X$ with the algebra of all linear transformations on the Hilbert space 
$\bb{C}^X := \oplus_{x\in X}\bb{C}$. 
If $T\in M_A$, we write $T^{\rm t}$ for the transpose matrix, and set $\bar T = (T^*)^{\rm t}$. 
We let $\cl D_X$ be the subalgebra of $M_X$ of all diagonal matrices, and $\Delta_X : M_X\to \cl D_X$
be the conditional expectation. 
We write $M_{XY} = M_X\otimes M_Y$, $\cl P_{XY}$ for the projection lattice of $M_{XY}$, 
and $\cl P_{XY}^{\clas}$ for the projection lattice of $\cl D_{XY}$.
We let $\xi\eta^*$ be the rank one operator given by $\xi\eta^*(\zeta) = \langle\zeta,\eta\rangle\xi$, 
and $(e_x)_{x\in X}$ be the canonical orthonormal basis of $\bb{C}^X$. 

For a Hilbert space $H$ and vectors $\xi,\eta\in H$, we write $\xi\perp \eta$ if 
$\langle \xi,\eta\rangle = 0$. 
Let $H^{\rm d}$ be the dual (Banach) space of $H$ and 
${\rm d} : H\to H^{\rm d}$ be the map, given by ${\rm d}(\xi) (\eta) = \langle \eta,\xi\rangle$; 
we write $\xi^{\dd} = \dd(\xi)$. 
Note that $(\lambda\xi)^{\dd} = \overline{\lambda}\xi^{\dd}$, $\lambda\in \bb{C}$, and that, 
if $T\in \cl L(H)$, then the dual operator $T^{\dd} : H^{\rm d}\to H^{\rm d}$ satisfies the relation
\begin{equation}\label{eq_Mxi}
T^{\rm d}\xi^{\rm d} = (T^*\xi)^{\rm d}, \ \ \ T\in \cl L(H).
\end{equation}

Let $\omega\in M_X$. 
Writing $f_\omega$ for the functional on $M_X$ given by
$f_\omega(\rho) = \Tr(\rho\omega^{\rm t})$,
we have that the map $\omega\to f_{\omega}$ is a complete order isomorphism from $M_X$ onto 
the dual operator system $M_X^{\rm d}$
(see e.g. \cite[Theorem 6.2]{ptt}). 
On the other hand, the map $\omega^{\dd}\mapsto \omega^{\rm t}$ is 
 a *-isomorphism  from 
$\cl L\left((\bb{C}^X)^{\dd}\right)$ onto $M_X$. 
The composition of these maps, $\omega^{\dd}\mapsto f_{\omega^{\rm t}}$, is thus 
a complete order isomorphism from 
$\cl L\left(\left(\bb{C}^X\right)^{\dd}\right)$ onto $M_X^{\rm d}$. In the sequel, we identify these two spaces; 
note that, via this identification,
\begin{equation}\label{eq_dtom}
\langle \rho,\omega^{\dd}\rangle = \langle\rho,\omega^{\rm t}\rangle = \Tr(\rho\omega), \ \ \ \rho,\omega\in M_X. 
\end{equation} 
If $P\in M_X$ is a projection, we write $P_{\perp}$ for the projection in $M_X^{\rm d}$ on the 
annihilator in $\left(\bb{C}^X\right)^{\dd}$ of the range of $P$.

Write $\epsilon_{x,x'} = e_x e_{x'}^*$ for the matrix unit in $M_X$, corresponding to the pair $(x,x')$ of indices. 
Set 
$$J_X = \frac{1}{|X|} \sum_{x,x'\in X} \epsilon_{x,x'}\otimes \epsilon_{x,x'};$$
if $\frak{m}_X = \frac{1}{\sqrt{|X|}} \sum_{x\in X} e_x\otimes e_x$ is the 
\emph{maximally entangled} unit vector in $\bb{C}^X\otimes \bb{C}^X$, then $J_X = \frak{m}_X\frak{m}_X^*$ is 
its corresponding rank one projection. 
Set also 
$$J_X^{\rm cl} = \sum_{x\in X} \epsilon_{x,x}\otimes \epsilon_{x,x},$$
and note that $\Delta_{XX}(J_X) =  \frac{1}{|X|} J_X^{\rm cl}$. 
Heuristically, $J_X^{\rm cl}$ is the (normalised) part of $J_X$ that can be seen by a classical observer. 

Recall \cite{tt-QNS} that a \emph{quantum non-local game} is a 
join-preserving map $\nph : \cl P_{XY}\to \cl P_{AB}$ with $\nph(0) = 0$, 
while a \emph{classical-to-quantum (cq) non-local game} is a 
join-preserving map $\nph : \cl P_{XY}^{\clas}\to \cl P_{AB}$ with $\nph(0) = 0$.
Similarly, a \emph{classical non-local game} is a 
join-preserving and zero-preserving map $\nph : \cl P_{XY}^{\clas}\to \cl P_{AB}^{\clas}$.

Recall also that a \emph{non-local game} on the quadruple $(X,Y,A,B)$ is a function 
$\lambda : X\times Y \times A \times B\to \{0,1\}$. 
In \cite{tt-QNS}, we associated to such $\lambda$ 
the classical non-local game
$\nph_{\lambda} : \cl P_{XY}^{\clas}\to \cl P_{AB}^{\clas}$ given by 
$$\nph_{\lambda}\left(\sum_{x,y\in\kappa} \epsilon_{x,x}\otimes \epsilon_{y,y}\right)
= \sum \left\{\epsilon_{a,a}\otimes \epsilon_{b,b} : \exists \ (x,y)\in \kappa \mbox{ s.t. }
\lambda(x,y,a,b) = 1\right\},$$
after recalling that projections in $\cl P^{\rm cl}_{XY}$ correspond to subsets $\kappa \subseteq X\times Y$.

A non-local game $(X,Y,A,B,\lambda)$ is called 
\begin{itemize}
\item
a \emph{mirror game} \cite{lmprsstw} if there exist functions $f : X\to Y$ and $g : Y\to X$ such that for every $x\in X$ 
(resp. $y\in Y$) the set 
$$\{(a,b) \in A\times B : \lambda(x,f(x),a,b) = 1\}$$ 
(resp. $$\{(a,b) \in A\times B : \lambda(g(y),y,a,b) = 1\})$$
is the graph of a bijection, and 
\item
a \emph{synchronous game} \cite{psstw} (see also \cite{hmps}) if
$X = Y$, $A = B$ and 
$$a,b\in A, a\neq b \ \Longrightarrow \ \lambda(x,x,a,b) = 0.$$
\end{itemize}
Mirror games include the subclass of unique games 
(that is, games for which the set 
$\{(a,b) \in A\times B : \lambda(x,y,a,b) = 1\}$ is the graph of a bijection for every $(x,y)\in X\times Y$ \cite{trevisan}); in particular, 
they form a class, strictly larger than that of synchronous games.

Set $B= A$ and recall the standard (linear) identification of matrices in $M_A$ with vectors in $\bb{C}^A\otimes \bb{C}^B$, which 
associates to the matrix unit $\nep_{a,b}$ the vector $e_a\otimes e_b$ (see e.g. \cite[Section 1.1.2]{watrous}).
Write 
$\tilde{\zeta}_T\in \bb{C}^A\otimes\bb{C}^B$ for the vector corresponding to $T\in M_A$ and set 
$\zeta_T = \frac{\tilde{\zeta}_T}{\|\tilde{\zeta}_T\|}$; we have that 
$\zeta_{I_A} = \frak{m}_A$.
We note the relations \cite[Section 1.1.1]{watrous}
\begin{equation}\label{eq_srt}
(R\otimes S)\tilde{\zeta}_T = \tilde{\zeta}_{RTS^{\rm t}}, \ \ \ R,S,T\in M_A.
\end{equation}

If $\alpha : A\to B$ is a bijection, let 
$P_{\alpha} = \sum_{a\in A}\epsilon_{a,a}\otimes \epsilon_{\alpha(a),\alpha(a)}$; 
clearly, $P_\alpha \in \cl P^{\rm cl}_{AB}$.

\begin{remark}\label{r_syn}
A non-local game $\lambda$ is 
\begin{itemize}

\item[(i)] 
synchronous if and only if 
$\nph_{\lambda}(J_X^{\rm cl}) \leq J_A^{\rm cl}$;

\item[(ii)]
mirror if and only if there exist 
functions $f : X\to Y$, $g : Y\to X$ and 
bijections $\alpha_x,\beta_y : A\to B$, $x\in X$, $y\in Y$, such that 
$$\nph_{\lambda}\left(\epsilon_{x,x} \otimes\epsilon_{f(x),f(x)}\right) = P_{\alpha_x}
\mbox{ and } 
\nph_{\lambda}\left(\epsilon_{g(y),g(y)}\otimes\epsilon_{y,y}\right) = P_{\beta_y^{-1}}, \ x\in X, y\in Y.$$
\end{itemize}
\end{remark}

\begin{proof}
(i) 
If $\lambda$ is synchronous then, clearly, 
$$\nph_{\lambda}\left(\epsilon_{x,x}\otimes \epsilon_{x,x}\right) \leq J_A^{\rm cl} \ \ \mbox{ for all } x\in X;$$
taking the span over all $x$, we get $\nph_{\lambda}(J_X^{\rm cl}) \leq J_A^{\rm cl}$. 
Conversely, the condition $\nph_{\lambda}(J_X^{\rm cl}) \leq J_A^{\rm cl}$ implies in particular 
$\nph_{\lambda}(\epsilon_{x,x}\otimes \epsilon_{x,x}) \leq J_A^{\rm cl}$, which is
equivalent to $\lambda(x,x,a,b) = 0$ whenever $a\neq b$. 
Claim (ii) is equally straightforward. 
\end{proof}

Remark \ref{r_syn} motivates the following versions of mirror and synchronous games, 
where the inputs are still classical, while the outputs are allowed to be quantum.
We assume that $|A| = |B|$ but continue to use different symbols to denote the sets $A$ and $B$ for clarity. 
If $\omega\in M_A$, let $L_{\omega} : M_{AB}\to M_B$ be the slice map, 
given by $L_{\omega}(S\otimes T) = \langle S,\omega\rangle T$ and write $\Tr_A = L_{I_A}$ for the partial trace; 
the slice map $L_{\rho} : M_{AB}\to M_A$, for $\rho\in M_B$, and the partial trace $\Tr_B$, are defined similarly. 
Call a rank one projection $P\in M_{AB}$ \emph{bijective} if 
\begin{equation}\label{eq_sep}
e,f\in \bb{C}^A, e\perp f \ \Longrightarrow \ L_{ee^*}(P) \perp L_{ff^*}(P)
\end{equation}
(note that the orthogonality is understood in terms of the trace in $M_B$).
Bijective projections can be thought of as quantum versions of bijections;
in fact, if $\alpha : A\to B$ is a bijection then $P = P_{\alpha}$ satisfies 
(\ref{eq_sep}) when $e$ and $f$ are taken to be elements of the standard basis.

\begin{lemma}\label{l_separ}
A rank one projection $P\in M_{AB}$ is bijective if and only if $P = \zeta_U\zeta_U^*$ for some unitary operator $U \in M_A$. 
\end{lemma}

\begin{proof}
Let $P = \zeta\zeta^*$ for some $\zeta\in \bb{C}^A$, and $U\in M_A$ be an operator with $\zeta = \zeta_U$. 
Let $e = \sum_{a\in A} \lambda_a e_a \in \bb{C}^A$ and write $\overline{e} = \sum_{a\in A} \overline{\lambda_a} e_a$. 
Set $r=\|\tilde\zeta_U\|$. We have 
\begin{eqnarray*}
\left \langle L_{ee^*}(\zeta\zeta^*),e_ae_b^*\right\rangle
& = & 
\left \langle \zeta\zeta^*, ee^* \otimes e_ae_b^*\right\rangle
= 
\Tr\left((\zeta\zeta^*) (ee^* \otimes e_ae_b^*)^{\rm t}\right)\\
& = & 
\Tr\left((\zeta\zeta^*) (ee^*)^{\rm t} \otimes (e_ae_b^*)^{\rm t}\right)
= 
\Tr\left((\zeta\zeta^*) (\overline{e}\overline{e}^*) \otimes (e_be_a^*)\right)\\
& = &  
\Tr\left((\zeta\zeta^*) (\overline{e}\otimes e_b)(\overline{e}\otimes e_a)^*\right)
= 
\langle \overline{e}\otimes e_b, \zeta\rangle
\langle \zeta, \overline{e}\otimes e_a\rangle\\
& = & 
r^{-2}\sum_{c\in A} \lambda_c \langle U e_a,e_c\rangle 
\sum_{d\in A} \overline{\lambda_d} \overline{\langle U e_b,e_d\rangle}\\
& = &  
r^{-2}\langle Ue_a,\overline{e}\rangle \overline{\langle Ue_b,\overline{e}\rangle}
= 
r^{-2}\langle e_a,U^*\overline{e}\rangle \langle U^*\overline{e},e_b\rangle\\
& = & 
r^{-2}\left\langle (U^*\overline{e})(U^*\overline{e})^* e_a, e_b\right \rangle
= 
r^{-2}\left\langle (U^*\overline{e})(U^*\overline{e})^*, e_b e_a^*\right \rangle\\
& = &  
r^{-2}\left\langle (\overline{U^*\overline{e}})(\overline{U^*\overline{e}})^*, e_a e_b^*\right \rangle
=
r^{-2}\left\langle (U^{\rm t}e)(U^{\rm t}e)^*, e_a e_b^*\right \rangle;
\end{eqnarray*}
thus, $L_{ee^*}(\zeta\zeta^*) = r^{-2}(U^{\rm t}e)(U^{\rm t}e)^*$.
It follows that $P$ is bijective if and only if $U^{\rm t}$ is a multiple of a unitary operator, 
that is, if and only if $\mu U$ is unitary for some $\mu\in\mathbb C$. 
Clearly, $P = \zeta_{\mu U}\zeta_{\mu U}^*$. 
\end{proof}

A projection $P\in M_A$ of rank $r$ will be called bijective if there exist 
partial isometries $U_i$, $i = 1,\dots,r$, such that 
$\sum_{i=1}^r U_iU_i^* = \sum_{i=1}^r U_i^*U_i = I$ and 
$P = \sum_{i=1}^r \zeta_{U_i}\zeta_{U_i}^*$.
Note that, if $\alpha:A\to B$ is a bijection and $P=P_\alpha$, then $P$ is bijective of rank $|A|$ with 
corresponding partial isometries $\epsilon_{\alpha(a),a}$, $a\in A$.

\begin{definition}\label{d_qmir}
Let $\nph : \cl P_{XY}^{\rm cl} \to \cl P_{AB}$ 
be a classical-to-quantum non-local game and 
$\psi : \cl P_{XY} \to \cl P_{AB}$ be a quantum non-local game. 
\begin{itemize}
\item[(i)]
$\nph$ is called a \emph{quantum output mirror game} if there exists functions $f : X\to Y$, $g : Y\to X$ such that 
the projections 
$\nph(\nep_{x,x}\otimes \nep_{f(x),f(x)})$ and 
$\nph(\nep_{g(y),g(y)}\otimes \nep_{y,y})$ are bijective, $x\in X$, $y\in Y$;

\item[(ii)]
$\nph$ is called \emph{concurrent} if $\nph(J_X^{\clas}) = J_A$;
\item[(iii)]
$\psi$ is called \emph{concurrent} if $\psi(J_X) = J_A$.
\end{itemize}
\end{definition}

In view of Remark \ref{r_syn}, we consider quantum output mirror games as a quantum version of mirror games, and concurrent games -- as quantum versions of synchronous games. 


\section{Classical-to-quantum games}\label{s_cq}

This section contains characterisations of the prefect strategies of quantum output mirror games and 
classical-to-quantum concurrent games, and their applications to quantum orthogonal ranks of graphs. 
We start with recalling 
the main classes of quantum no-signalling correlations introduced in \cite{tt-QNS} that will be used subsequently.


\subsection{Quantum no-signalling correlations}\label{ss_qnscor}

If $\cl A$ is a C*-algebra, we denote by $\cl A^{\rm op}$ its \emph{opposite} C*-algebra. As a set, $\cl A^{\rm op}$ can be identified with $\cl A$ and we write $\cl A^{\rm op} = \{z^{\rm op} : z\in \cl A\}$; the C*-algebra $\cl A^{\rm op}$ has the same norm, additive and involutive structure as $\cl A$, and its multiplication is given by letting $z_1^{\rm op}z_2^{\rm op} = (z_2 z_1)^{\rm op}$, $z_1, z_2\in \cl A$.

Let $\frak{V}_{X,A}$ be the ternary ring, generated by elements $v_{a,x}$, $x\in X$, $a\in A$, such that 
the matrix $V = (v_{a,x})_{a\in A, x\in X}$ satisfies the condition of an isometry, 
that is, 
$$\sum_{a\in A} v_{a'',x''} v_{a,x}^* v_{a,x'} = \delta_{x,x'} v_{a'',x''}, \ \ x,x',x''\in X, a''\in A.$$
Let $\frak{C}_{X,A}$ be the unital *-algebra, generated by the set 
$\{v_{a,x}^* v_{a',x'} : x,x'\in X, a,a'\in A\}$, 
and set $e_{x,x',a,a'} = v_{a,x}^* v_{a',x'}$ for brevity. 
Further, let $\cl V_{X,A}$ be the universal ternary ring of operators (TRO) of the isometry $V$, and let $\cl C_{X,A}$ be its right C*-algebra; thus, $\cl C_{X,A}$ is generated, as a C*-algebra, by $e_{x,x',a,a'}$, $x,x'\in X$, $a,a'\in A$ (see \cite{tt-QNS}). 
We write 
$$E = (e_{x,x',a,a'})_{x,x',a,a'} \ \mbox{ and } \ 
E^{\rm op} = (e_{x',x,a',a}^{\rm op})_{x,x',a,a'};$$
thus, 
$E\ \in M_{XA}\otimes \frak{C}_{X,A}$ and $E^{\rm op} \in M_{XA}\otimes \frak{C}_{X,A}^{\rm op}$.

A \emph{stochastic operator matrix} acting on a Hilbert space $H$ is a positive block operator matrix $\tilde{E} = (E_{x,x',a,a'})_{x,x',a,a'}\in M_{XA}(\cl B(H))$ such that $\Tr_A \tilde{E} = I$. 
Stochastic operator matrices $\tilde{E}$ acting on $H$ correspond to unital *-representations $\pi : \cl C_{X,A}\to \cl B(H)$ by via the assignment $\pi(e_{x,x',a,a'}) = E_{x,x',a,a'}$, $x,x'\in X$, $a,a'\in A$
\cite{tt-QNS}.

Let $X$, $Y$, $A$ and $B$ be finite sets. 
A \emph{quantum no-signalling (QNS) correlation} \cite{dw} is a quantum channel (that is, a completely positive trace preserving map) $\Gamma : M_{XY}\to M_{AB}$ such that 
\begin{equation}\label{eq_qns1}
\Tr\mbox{}_A\Gamma(\rho_X\otimes \rho_Y) = 0 \ \mbox{ whenever } \rho_X\in M_X \mbox{ and }  \Tr(\rho_X) = 0,
\end{equation}
and
\begin{equation}\label{eq_qns2}
\Tr\mbox{}_B\Gamma(\rho_X\otimes \rho_Y) = 0 \ \mbox{ whenever } \rho_Y\in M_Y \mbox{ and }\Tr(\rho_Y) = 0.
\end{equation}
A QNS correlation $\Gamma : M_{XY}\to M_{AB}$ is \emph{quantum commuting} if 
there exist a Hilbert space $H$, a unit vector $\xi\in H$ and
stochastic operator matrices $\tilde{E} = (E_{x,x',a,a'})_{x,x',a,a'}$ and $\tilde{F} = (F_{y,y',b,b'})_{y,y',b,b'}$ 
on $H$ such that 
$$E_{x,x',a,a'} F_{y,y',b,b'} = F_{y,y',b,b'}E_{x,x',a,a'}$$
for all $x,x'\in X$, $y,y'\in Y$, $a,a'\in A$, $b,b'\in B$,
and the Choi matrix of $\Gamma$ coincides with 
\begin{equation}\label{eq_EFp}
(E_{x,x',a,a'}F_{y,y',b,b'})_{x,x',a,a'}^{y,y',b,b'} \in M_{XYAB}(\cl B(H)). 
\end{equation}
\emph{Quantum} QNS correlations are defined as in (\ref{eq_EFp}), but using tensor products of stochastic operator matrices acting on finite dimensional Hilbert spaces (that is, ones having the form $E_{x,x',a,a'}\otimes F_{y,y',b,b'}$). 
\emph{Approximately quantum} QNS correlations are limits of quantum QNS correlations, while \emph{local} QNS correlations are defined as in (\ref{eq_EFp}) by requiring that the entries of $\tilde{E}$ (resp. $\tilde{F}$) pairwise commute.

We write $\cl Q_{\rm qc}$ (resp. $\cl Q_{\rm qa}$, $\cl Q_{\rm q}$, $\cl Q_{\rm loc}$) for the (convex) set of all quantum commuting (resp. approximately quantum, quantum, local) QNS correlations. 
It was shown in \cite{tt-QNS} that $\Gamma \in \cl Q_{\rm qc}$ precisely when there exists
a state $s : \cl C_{X,A}\otimes_{\max} \cl C_{Y,B}\to \bb{C}$ such that 
$\Gamma = \Gamma_s$, where $\Gamma_s$ is given by 
$$
\Gamma_s(\epsilon_{x,x'}\otimes \epsilon_{y,y'}) = \sum_{a,a'\in A} \sum_{b,b'\in B} s(e_{x,x',a,a'}\otimes e_{y,y',b,b'}) \epsilon_{a,a'}\otimes \epsilon_{b,b'},$$
where $x,x'\in X$, $y,y'\in Y$.
Similarly, $\Gamma\in \cl Q_{\rm qa}$ precisely when $\Gamma = \Gamma_s$ for some state $s$ of 
$\cl C_{X,A}\otimes_{\min} \cl C_{Y,B}$, and 
$\Gamma\in \cl Q_{\rm q}$ (resp. $\Gamma\in \cl Q_{\rm loc}$) if and only if 
$\Gamma = \Gamma_s$ for some state $s$ of $\cl C_{X,A}\otimes_{\min} \cl C_{Y,B}$ that factors through a finite dimensional (resp. abelian) representation of the latter C*-algebra. 
We point out that the elements of $\cl Q_{\rm loc}$ are precisely the quantum channels of the form $\Gamma = \sum_{i=1}^k \lambda_i \Phi_i \otimes \Psi_i$ as a convex combination (where $\Phi_i : M_X\to M_A$ and $\Psi_i : M_Y\to M_B$ are quantum channels, $i = 1,\dots,k$). 

Let $\frak{B}_{X,A}$ (resp. $\cl B_{X,A}$) 
be the algebraic (resp. the C*-algebraic) free product $M_{A}\ast_1\cdots \ast_1 M_{A}$, 
and $\frak{A}_{X,A}$ (resp. $\cl A_{X,A}$) be the algebraic (resp. the C*-algebraic) 
free product $\cl D_{A}\ast_1\cdots \ast_1 \cl D_{A}$,
both having $|X|$ terms and amalgamated over the units.
We denote by $e_{x,a,a'}$, $a,a'\in A$, the matrix units of the $x$-th copy of $M_A$ in $\frak{B}_{X,A}$, 
and by $e_{x,a}$, $a\in A$, the canonical basis of the $x$-th copy of $\cl D_A$ in $\frak{A}_{X,A}$. 
Set 
$E_{\rm cq} = (e_{x,a,a'})_{x,a,a'}\in \cl D_X\otimes M_A\otimes \frak{B}_{X,A}$ and
$E_{\rm cq}^{\rm op} = (e_{x,a',a}^{\rm op})_{x,a,a'}\in \cl D_X\otimes M_A\otimes \frak{B}_{X,A}^{\rm op}$;
similarly, let 
$E_{\rm cl} = (e_{x,a})_{x,a}\in \cl D_{XA}\otimes \frak{A}_{X,A}$ and
$E_{\rm cl}^{\rm op} = (e_{x,a}^{\rm op})_{x,a}\in \cl D_{XA}\otimes \frak{A}_{X,A}^{\rm op}$.

A \emph{classical-to-quantum no-signalling (CQNS)} correlation is a channel $\cl E : \cl D_{XY}\to M_{AB}$ such that (\ref{eq_qns1}) and (\ref{eq_qns2}) hold true for (traceless) elements $\rho_X\in \cl D_X$ and $\rho_Y\in \cl D_Y$. 
A \emph{semi-classical} stochastic operator matrix acting on a Hilbert space $H$ is a positive 
block operator matrix $\tilde{E} = (E_{x,a,a'})_{x,a,a'}\in \cl D_X\otimes M_A(\cl B(H))$ with $\Tr_A \tilde{E} = I$. 
A CQNS correlation $\cl E$ is \emph{quantum commuting} if its
Choi matrix is given as in (\ref{eq_EFp}) but employing semi-classical stochastic operator matrices; this is equivalent to the requirement that its
canonical extension to a QNS correlation $M_{XY}\to M_{AB}$ is quantum commuting, as well as to the existence of a state $s$ of $\cl B_{X,A}\otimes_{\max}\cl B_{X,A}$ such that 
$\cl E = \Gamma_s$, where $\cl E_s$ is the CQNS correlation given by 
$$
\Gamma_s(\epsilon_{x,x}\otimes \epsilon_{y,y}) = \sum_{a,a'\in A} \sum_{b,b'\in B} s(e_{x,a,a'}\otimes e_{y,b,b'}) \epsilon_{a,a'}\otimes \epsilon_{b,b'}.$$
Similarly, \emph{approximately quantum} (resp. \emph{quantum}, \emph{local}) CQNS correlations have the form $\cl E_s$, where $s$ is a state of 
$\cl B_{X,A}\otimes_{\min}\cl B_{X,A}$
(which in addition gives rise to a finite dimensional and abelain GNS representation, respectively). 
We denote by $\cl{CQ}_{\rm qc}$ (resp. $\cl{CQ}_{\rm qa}$, $\cl{CQ}_{\rm q}$, $\cl{CQ}_{\rm loc}$) the (convex) set of all quantum commuting (resp. approximately quantum, quantum, local) QNS correlations.

Let $\nph : \cl P_{XY}\to \cl P_{AB}$ be a quantum non-local game. A QNS correlation $\Gamma : M_{XY}\to M_{AB}$ is called a \emph{perfect strategy} for $\nph$ if 
$$\left\langle \Gamma(P),\nph(P)_{\perp}\right\rangle = 0, \ \ \ \ P\in \cl P_{XY}.$$
Perfect strategies for classical-to-quantum non-local games are defined analogously \cite{tt-QNS}.


\subsection{Quantum output mirror games}\label{ss_qomg}

We first describe the perfect strategies of quantum output mirror games 
that lie in the various correlation classes.  
In the sequel, we fix finite sets $X$, $Y$, $A$ and $B$, 
and for clarity denote the canonical generators of $\cl B_{Y,B}$ by $f_{y,b,b'}$, 
$y\in Y$, $b,b'\in B$.
We fix a quantum output mirror game 
$\varphi:\mathcal{P}_{XY}^{\rm cl} \to \mathcal{P}_{AB}$ 
and let $f:X \to Y$ and $g:Y \to X$ be as in Definition \ref{d_qmir}. We write
\[ \varphi\left(\epsilon_{x,x} \otimes \epsilon_{f(x),f(x)}\right)
= \sum_{i=1}^{r(x)} \zeta_{U_i^x} \zeta^*_{U_i^x}, \, x \in X,\]
where $U_i^x$, $i=1,\dots,r(x)$, $x \in X$, are partial isometries satisfying the relations
\begin{equation}\label{eq_unre}
\sum_{i=1}^{r(x)} (U_i^x)^*U_i^x=\sum_{i=1}^{r(x)} U_i^x (U_i^x)^*=I.
\end{equation}
Let $D_x := \sum_{i=1}^{r(x)} U_i^x$; the relations (\ref{eq_unre}) imply that $D_x$ is unitary.

\begin{lemma}\label{le_qomgcs}
Let $s$ be a state of $\mathcal{B}_{X,A} \otimes_{\max} \mathcal{B}_{Y,B}$ such that 
$\Gamma_s : \mathcal{D}_{XY} \to M_{AB}$ is a perfect quantum commuting CQNS strategy for $\varphi$. 
Let $\pi_1:\mathcal{B}_{X,A} \to \mathcal{B}(H)$ and $\pi_2:\mathcal{B}_{Y,B} \to \mathcal{B}(H)$ be *-representations with commuting ranges and $\xi \in H$ be a unit vector such that 
$$s(u_1\otimes u_2)=\langle\pi_1(u_1)\pi_2(u_2)\xi,\xi \rangle, \ \ \ u_1\in \cl B_{X,A}, u_2\in \cl B_{Y,B},$$ $E_x=(\pi_1(e_{x,a,a'}))_{a,a' \in A}$ and $F_y=(\pi_2(f_{y,b,b'}))_{b,b' \in B}$. Then
\[ 
(U_i^x\otimes I)^* E_x(e_a \otimes \xi) = F_{f(x)}^{\rm t} (U_i^x\otimes I)^* (e_a \otimes \xi), \ \  i=1,\dots,r(x), \, a \in A.\]
\end{lemma}

\begin{proof}
Set $P_{i,x} = U_i^x(U_i^x)^*$ and $Q_{i,x} = (U_i^x)^*U_i^x$; thus, 
$$\sum_{i=1}^{r(x)} P_{i,x} = \sum_{i=1}^{r(x)} Q_{i,x} = I,$$ 
that is, $(P_{i,x})_{i=1}^{r(x)}$ and $(Q_{i,x})_{i=1}^{r(x)}$ are PVM's (in $M_A$) for every $x\in X$.  
We have that, if $\Gamma := \Gamma_s$ then 
\begin{equation}\label{eqgamma}
\Gamma\hspace{-0.1cm}\left(\nep_{x,x}\otimes \nep_{f(x),f(x)}\right) 
\hspace{-0.1cm} = \hspace{-0.1cm}
\left(\sum_{i=1}^{r(x)}\zeta_{U_i^x}\zeta_{U_i^x}^*\right)
\hspace{-0.1cm} \Gamma\hspace{-0.1cm}\left(\nep_{x,x}\otimes \nep_{f(x),f(x)}\right)
\hspace{-0.1cm}\left(\sum_{i=1}^{r(x)} \zeta_{U_i^x}\zeta_{U_i^x}^*\right)
\hspace{-0.15cm}.
\end{equation}
Taking traces in (\ref{eqgamma}), we obtain \begin{eqnarray*}
1
& = & 
\Tr\left(\Gamma(\nep_{x,x}\otimes \nep_{f(x),f(x)})\right)\\
& = &
\sum_{i,j = 1}^{r(x)}\sum_{a,b,a',b'} s(e_{x,a,a'}\otimes f_{f(x),b,b'})
\Tr\left((\zeta_{U_j^x}\zeta_{U_j^x}^*)(\nep_{a,a'}\otimes \nep_{b,b'})(\zeta_{U_i^x}\zeta_{U_i^x}^*)\right)\\
& = &
\sum_{i,j=1}^{r(x)}\sum_{a,b,a',b'} s(e_{x,a,a'}\otimes f_{f(x),b,b'})
\langle\zeta_{U_i^x},e_{a'}\otimes e_{b'}\rangle\langle e_a\otimes e_b,\zeta_{U_j^x}\rangle\langle\zeta_{U_j^x},\zeta_{U_i^x}\rangle\\
& = &
\sum_{i=1}^{r(x)} \sum_{a,b,a',b'}s(e_{x,a,a'}\otimes f_{f(x),b,b'})\langle\zeta_{U_i^x},e_{a'}\otimes e_{b'}\rangle\langle e_a\otimes e_b,\zeta_{U_i^x}\rangle,
\end{eqnarray*}
where we have used the fact that  $\langle\zeta_{U_j^x},\zeta_{U_i^x}\rangle=0$ whenever $i\ne j$. 
Recall that $s(e_{x,a,a'}\otimes e_{y,b,b'}) = \left\langle E_{x,a,a'}F_{y,b,b'}\xi,\xi\right\rangle$ for all $x\in X, y\in Y, 
a,a'\in A, b,b'\in B$. In the sequel, we denote by $T_{a,b}$ the $(a,b)$-entry of a (possibly block operator) matrix $T$. 
Noting that 
$\zeta_{U_i^x} = \left(\sum_{a,b}(U_i^x)_{a,b}e_a\otimes e_b\right)/\|U_i^x\|_2$ 
and $\|U_i^x\|_2^2 = \Tr(U_i^x(U_i^x)^*)$ is the rank $r_i(x)$ of the projection $P_{i,x}$, we obtain
$$\left\langle\zeta_{U_i^x},e_{a'}\otimes e_{b'} \right\rangle = \frac{(U_i^x)_{a',b'}}{r_i(x)^{1/2}}
\ \mbox{ and } \ \left\langle e_a\otimes e_b,\zeta_{U_i^x}\right\rangle = \frac{\overline{(U_i^x)}_{a,b}}{r_i(x)^{1/2}}.$$
Setting $\tilde{U}_i^x = U_i^x\otimes I$, $i = 1,\dots,r(x)$, $x\in X$, we therefore have
\begin{eqnarray*}
1
& = &
\sum_{i=1}^{r(x)}\sum_{a,b,a',b'}
\left\langle E_{x,a,a'}F_{f(x),b,b'}\xi,\xi\right\rangle
\left\langle\zeta_{U_i^x},e_{a'}\otimes e_{b'}\right\rangle
\left\langle e_a\otimes e_b,\zeta_{U_i^x}\right\rangle\\
& = &
\sum_{i=1}^{r(x)}\frac{1}{r_i(x)}\sum_{a,b,a',b'}
\left\langle E_{x,a,a'}(U_i^x)_{a',b'}F_{f(x),b,b'}\overline{(U_i^x})_{a,b}\xi,\xi \right\rangle\\
&=&\sum_{i=1}^{r(x)}\frac{1}{r_i(x)}\sum_{a\in A}
\left\langle \left(E_x\tilde{U}_i^xF_{f(x)}^{\rm t}(\tilde{U}_i^x)^*\right)_{a,a}\xi,\xi\right\rangle\\
&=&\sum_{i=1}^{r(x)}\frac{1}{r_i(x)}\sum_{a\in A} 
\left\langle E_x\tilde{U}_i^xF_{f(x)}^{\rm t}(\tilde{U}_i^x)^*(e_a\otimes\xi),e_a\otimes\xi \right\rangle\\
& = &
\sum_{i=1}^{r(x)}\frac{1}{r_i(x)}\sum_{a\in A} 
\left\langle F_{f(x)}^{\rm t}(\tilde{U}_i^x)^*(e_a\otimes\xi),(\tilde{U}_i^x)^*E_x(e_a\otimes\xi)\right\rangle.
\end{eqnarray*} 
By the Cauchy-Schwartz inequality,
\begin{eqnarray}\label{prod}
1 
\hspace{-0.2cm} 
&\leq& 
\hspace{-0.2cm}
\sum_{i=1}^{r(x)} \frac{1}{r_i(x)}\sum_{a\in A} 
\left|\left\langle F_{f(x)}^{\rm t}(\tilde{U}_i^x)^*(e_a\otimes\xi),(\tilde{U}_i^x)^*E_x(e_a\otimes\xi)\right\rangle\right|\\
\hspace{-0.2cm} & \leq & \hspace{-0.2cm} 
\left(\sum_{i=1}^{r(x)} \sum_{a\in A} 
\frac{1}{r_i(x)}
\|F_{f(x)}^{\rm t}(\tilde{U}_i^x)^*(e_a\otimes\xi)\| \|(\tilde{U}_i^x)^*E_x(e_a\otimes\xi)\|\right)^2 \nonumber\\
\hspace{-0.2cm} & \leq & \hspace{-0.2cm} 
\left(\sum_{i=1}^{r(x)}\sum_{a\in A}\frac{1}{r_i(x)}\|(\tilde{U}_i^x)^*E_x(e_a\otimes\xi)\|^2\right)\nonumber\\
&& \hspace{2cm}  \times
\left(\sum_{i=1}^{r(x)}\sum_{a\in A}\frac{1}{r_i(x)}\| F_{f(x)}^{\rm t}(\tilde{U}_i^x)^*(e_a\otimes\xi)\|^2\right)\nonumber\\
\hspace{-0.2cm}  & = &\hspace{-0.2cm}  
\left(\sum_{a\in A}\left\langle\sum_{i=1}^{r(x)}\frac{1}{r_i(x)}E_x^*\tilde{U}_i^x(\tilde{U}_i^x)^*E_x(e_a\otimes\xi), e_a\otimes\xi\right\rangle\right)\nonumber\\
\hspace{-0.2cm} &&\hspace{-0.2cm} 
\times\left(\sum_{a\in A}\left\langle\sum_{i=1}^{r(x)}\frac{1}{r_i(x)}\tilde{U}_i^x(F_{f(x)}^{\rm t})^* F_{f(x)}^{\rm t}(\tilde{U}_i^x)^*(e_a\otimes\xi), e_a\otimes\xi\right\rangle\right).\nonumber
\end{eqnarray}
Since $(P_{i,x})_{i=1}^{r(x)}$ is a PVM, 
there exist a partition $(S_i)_{i=1}^{r(x)}$ of $A$ with $|S_i|=r_i(x)$ and 
a unitary $V_x$ in $M_A$  such that $V_x^*P_{i,x}V_x$ coincides with the projection 
$P_{S_i}$ onto $\text{span}\{e_a : a\in S_i\}$, $i = 1,\dots,r(x)$. 
Let $\tilde E_x = V_x^*E_xV_x$, and write $\tilde E_x=\sum_{a,b} \epsilon_{a,b} \otimes \tilde E_{x,a,b}$. As $V_x$ is unitary and $E_{x,a,b}E_{x,a',b'}=\delta_{b,a'} E_{x,a,b'}$, we also have $\tilde E_{x,a,b} \tilde E_{x,a',b'}=\delta_{b,a'} \tilde E_{x,a,b'}$. Thus, we have 
\begin{eqnarray*}
\tilde E_x ^*P_{S_i}\tilde E_x 
& = & 
\left(\sum_{a,b}\nep_{a,b}\otimes\tilde E_{x,a,b}\right)
\left(\sum_{c\in S_i}\nep_{c,c}\otimes 1\right)\left(\sum_{a',b'}\nep_{a',b'}\otimes \tilde E_{x,a',b'}\right)\\
& = & 
r_i(x)\tilde E_x.
\end{eqnarray*}
Let, similarly, 
$(R_i)_{i=1}^{r(x)}$ be a partition of $B$ with $|R_i| = r_i(x)$ and 
$W_x$ be a unitary such that $W_x^*Q_{x,i} W_x = P_{R_i}$, $i = 1,\dots,r(x)$.
Setting $\tilde F_{f(x)} = W_x^*F_{f(x)}W_x$, we have that 
$(\tilde F_{f(x)}^{\rm t})^*P_{R_i}\tilde F_{f(x)}^{\rm t} = r_i(x)\tilde F_{f(x)}^{\rm t}$. 
This implies that the last product in (\ref{prod}) is equal to 
\begin{eqnarray*}
&&\left(\sum_{a\in A} \langle  E_x(e_a\otimes\xi),e_a\otimes\xi\rangle\right)
\left(\sum_{a\in A} \langle F_{f(x)}^{\rm t}(e_a\otimes\xi),e_a\otimes\xi\rangle\right)\\
&&
= \left\langle\sum_{a\in A} E_{x,a,a}\xi,\xi\right\rangle
\left\langle \sum_{a\in A} F_{f(x),a,a}\xi,\xi\right\rangle = 1.
\end{eqnarray*}
Hence we have equalities in all chains of inequalities 
which implies that there exist scalars 
$\lambda_x$ such that 
$$F_{f(x)}^{\rm t}(\tilde{U}_i^x)^*(e_a\otimes\xi) 
= \lambda_x(\tilde{U}_i^x)^*E_x(e_a\otimes\xi), \ i = 1,\dots,r(x), \ a\in A.$$
Summing up over $i$, we obtain that 
$(D_x\otimes I)F_{f(x)}^{\rm t}(D_x^*\otimes I) (e_a\otimes\xi) = \lambda_xE_x(e_a\otimes\xi)$ for all $a\in A$. 
After applying $\Tr_A$, we conclude that $\lambda_x = 1$, which yields the desired result.
\end{proof}

\begin{theorem}\label{th_qomg}
Let $\nph : \cl P_{XY}^{\rm cl} \to \cl P_{AB}$ be a quantum output mirror game and 
$\Gamma : \cl D_{XY}\to M_{AB}$ be a perfect quantum commuting CQNS strategy for $\nph$. 
Then there exists a tracial state $\tau : \cl B_{X,A}\to \bb{C}$ and a 
*-homomorphism $\rho : \cl B_{Y,B}\to \cl B_{X,A}$ 
such that 
\begin{equation}\label{eq_Gam}
\Gamma(\epsilon_{x,x} \otimes \epsilon_{y,y}) = \left(\tau(e_{x,a,a'}\rho(f_{y,b',b}))\right)_{a,a',b,b'}, \ \ \ x,y\in X.\
\end{equation}
\end{theorem}

\begin{proof} 
We choose $f:X \to Y$ and $g:Y \to X$ as in Definition \ref{d_qmir}, and write
\[ \varphi(\epsilon_{x,x} \otimes \epsilon_{f(x),f(x)})=\sum_{i=1}^{r(x)} \zeta_{U_i^x} \zeta_{U_i^x}^*, \, x \in X,\]
for partial isometries $U_i^x$, $i=1,\dots,r(x)$, $x \in X$, such that
\[ \sum_{i=1}^{r(x)} (U_i^x)^*U_i^x=\sum_{i=1}^{r(x)}U_i^x(U_i^x)^*=I.\]
Keeping the notation from the proof of Lemma \ref{le_qomgcs}, we write $\Gamma(\epsilon_{x,x} \otimes \epsilon_{y,y})=\sum_{a,a' \in A} \sum_{b,b' \in B} \langle E_{x,a,a'}F_{y,b,b'}\xi,\xi \rangle$, where the assignments $e_{x,a,a'} \mapsto E_{x,a,a'} \in \mathcal{B}(H)$ and  $f_{y,b,b'} \mapsto F_{y,b,b'} \in \mathcal{B}(H)$ define  *-representations of $\mathcal{B}_{X,A}$ and $\mathcal{B}_{Y,B}$, respectively, with commuting ranges. 
Set $E_x=(E_{x,a,a'})_{a,a' \in A}$ and $F_y=(F_{y,b,b'})_{b,b' \in B}$.
By Lemma \ref{le_qomgcs},
\[F_{f(x)}^{\rm t} (\tilde U_i^x)^*(e_a \otimes \xi)=(\tilde U_i^x)^* E_x(e_a \otimes \xi), \, i=1,\dots,r(x), \, a \in A.\]
Let $Q = ((D_x\otimes I)(f_{f(x),a,b})^{\rm t}_{(a,b)}) (D_x^*\otimes I))_{a,b}$ and write 
$Q = (q_{x,a,b})_{b,a}$. Set
$$h_{x,a,b}=e_{x,a,b}\otimes 1-1\otimes q_{x,b,a}, \ \ \ x\in X, a,b\in A.$$
We have 
\begin{eqnarray*} 
h_{x,a,b}^*h_{x,a,b} 
& = & 
\left(e_{x,b,a}\otimes 1 - 1 \otimes q_{x,a,b}\right) \left(e_{x,a,b}\otimes 1 - 1 \otimes q_{x,b,a}\right)\\
& = & 
e_{x,b,b}\otimes 1 - e_{x,b,a} \otimes q_{x,b,a} - e_{x,a,b}\otimes q_{x,a,b} + 1 \otimes q_{x,a,a}.
\end{eqnarray*}
Let $s\in \cl B_{X,A}\otimes_{\max}\cl B_{Y,B}$ be such that $\Gamma=\Gamma_s$. 
As
\begin{eqnarray*}&&s(e_{x,b,a} \otimes q_{x,b,a})=\langle E_{x,b,a}((D_x\otimes I)F_{f(x)}^t (D_x^*\otimes I))_{a,b}\xi,\xi\rangle\\&=&\langle((D_x\otimes I)F_{f(x)}^t (D_x^*\otimes I))_{a,b}\xi,E_{x,a,b}\xi\rangle=\langle E_{x,a,b}\xi,E_{x,a,b}\xi\rangle=\langle E_{x,b,b}\xi,\xi\rangle,
\end{eqnarray*}
we get
\begin{equation}\label{eq_hab1}
s(h_{x,a,b}^*h_{x,a,b}) = 0, \ \ \ x\in X, a,b\in A.
\end{equation}

For $u,v \in \cl B_{X,A}\otimes_{\max} \cl B_{Y,B}$, write $u \sim v$ if $s(u - v) = 0$. 
Equations (\ref{eq_hab1}), combined with the Cauchy-Schwarz inequality, imply 
$$u h_{x,a,b} \sim 0 \mbox{ and } h_{x,a,b}^* u \sim 0, \ \ x\in X, a,b\in A, \ u\in \cl B_{X,A} \otimes_{\max}\cl B_{Y,B}.$$
Since $h_{x,a,b}^* = h_{x,b,a}$, we have 
\begin{equation}\label{eq_uhe}
u h_{x,a,b} \sim 0 \mbox{ and } h_{x,a,b} u \sim 0, \ \ x\in X, a,b\in A, \ u\in \cl B_{X,A} \otimes_{\max}\cl B_{Y,B}.
\end{equation}
In particular, 
\begin{equation}\label{eq_zeez1}
ze_{x,a,b}\otimes 1 \sim z\otimes q_{x,b,a} \sim e_{x,a,b} z\otimes 1, \ \ x\in X, a,b\in A, \ z\in \cl B_{X,A}.
\end{equation}
Similarly, let $V_i^y$, $i=1,\ldots d(y)$, be partial isometries such that 
$$\sum_{i=1}^{d(y)}V_i^y(V_i^y)^* = \sum_{i=1}^{d(y)}(V_i^y)^*V_i^y = I$$
and
$$\varphi(\nep_{g(y),g(y)}\otimes \nep_{y,y}) = \sum_{i=1}^{d(y)}\zeta_{V_i^y}\zeta_{V_i^y}^*.$$ 
Similarly to the proof of Lemma \ref{le_qomgcs}, letting $G_y = \sum_{i=1}^{d(y)} V_i^y$, we obtain that $F_{y,a,b}\xi = ((G_y\otimes I)(E_{g(y),a,b})_{a,b}^{\rm t}(G_y^*\otimes I))_{a,b}\xi$.

Set $(p_{y,a,b})_{b,a} = ((G_y\otimes I)(e_{g(y),a,b})_{a,b}^{\rm t} (G_y^*\otimes I))_{a,b}$ and 
note that $\{p_{y,a,b} : a,b\}$ is a matrix unit system, $y\in Y$. 
Letting
$g_{y,b,b'} = p_{y,b,b'}\otimes 1  - 1\otimes f_{y,b,b'}$, where $y\in Y$ and $b,b'\in B$, we
obtain, similarly,
\begin{equation}\label{eq_zeez21}
zp_{y,b,b'}\otimes 1 \sim z\otimes f_{y,b,b'} \sim p_{y,b,b'} z\otimes 1, \ \ y\in Y, b,b'\in B, \ z\in \cl B_{X,A}.
\end{equation}

Let $z$ and $w$ be (finite) words on the set $\cl E := \{e_{x,a,b} : x\in X, a,b\in A\}$. 
We show by induction on the length $|w|$ of $w$ that 
\begin{equation}\label{eq_zwwz1}
zw\otimes 1 \sim wz\otimes 1. 
\end{equation}
In the case $|w| = 1$, the claim reduces to (\ref{eq_zeez1}). Suppose (\ref{eq_zwwz1}) holds if $|w|\leq n-1$.
Let $|w| = n$ and write $w = w'e$, where $e\in \cl E$. Using (\ref{eq_zeez1}), we have 
$$zw \otimes 1 = zw'e \otimes 1 \sim ezw' \otimes 1 \sim w'ez \otimes 1 = wz \otimes 1.$$

Let $\tau : \cl B_{X,A}\to \bb{C}$ be given by 
$\tau(z) = s(z\otimes 1)$; it is clear that $\tau$ is a state on $\cl B_{X,A}$. 
From (\ref{eq_zwwz1}) and the fact that 
the set of all linear combinations of words on $\cl E$ is dense in $\cl A$, we conclude that $\tau$ is a trace on $\cl B_{X,A}$. 
Identity (\ref{eq_zeez21}) implies that 
$$s\left(e_{x,a,a'}\otimes f_{y,b,b'}\right) = \tau\left(e_{x,a,a'} p_{y,b,b'}\right), \ \ x\in X, y\in Y, a,a',b,b'\in A.$$
Equality (\ref{eq_Gam}) is now immediate if we let $\rho : \cl B_{Y,B}\to \cl B_{X,A}$ be the *-homomorphism defined by letting $\rho(f_{y,b',b})=p_{y,b,b'}$, $y\in Y$, $b,b'\in B$. 
\end{proof}

We will write $\Gamma = \Gamma_{\rho,\tau}$ if the CQNS correlation $\Gamma : \cl D_{XY}\to M_{AB}$ is given as in (\ref{eq_Gam}).
Keeping the notation from the proof of Theorem \ref{th_qomg}, 
let $\pi : \cl B_{X,A}\to \cl B_{Y,B}$ be the 
*-homomorphism given by $\pi(e_{x,a,a'}) = q_{x,a,a'}$. 
We will need the following lemma, which can be thought of as a 
dilation result for semi-classical stochastic operator matrices. 


\begin{lemma}\label{l_lift}
Let $X$ and $A$ be finite sets and 
$(E_{x,a,a'})_{x,a,a'}$, where $x\in X$ and $a,a'\in A$, 
be a semi-classical stochastic 
operator matrix acting on a finite dimensional Hilbert space $H$.
Then there exist matrix unit systems
$(\tilde{E}_{x,a,a'})_{a,a'}$, $x\in X$, on 
a finite dimensional Hilbert space $\tilde{H}$, and an isometry 
$V : H\to \tilde{H}$, such that 
$V^*\tilde{E}_{x,a,a'} V = E_{x,a,a'}$
for all $x\in X$ and all $a,a'\in A$. 
\end{lemma}

\begin{proof}
Write $X = [k]$ and use induction on $k$. If $k = 1$, the result is a direct consequence of the 
Stinespring Theorem. 
Resorting to the inductive assumption, suppose that 
$H_{k-1}$ is a finite dimensional Hilbert space, $V_{k-1} : H\to H_{k-1}$ is an isometry, 
and $(F_{x,a,a'})_{a,a'}$ is a matrix unit system on $H_{k-1}$, such that 
$$V_{k-1}^*F_{x,a,a'}V_{k-1} = E_{x,a,a'}, \ \ \ x\in [k-1], \ a,a'\in A.$$
Let $F'_{k,a,a'} = V_{k-1}E_{k,a,a'} V_{k-1}^*$, $a,a'\in A$. 
Note that 
$(F'_{k,a,a'})_{a,a'}\in (M_A\otimes$ $\cl B(H_{k-1}))^+$ and 
$\sum_{a\in A} F'_{k,a,a} = P_{k-1} := V_{k-1}V_{k-1}^*$. 
Fix $a_0\in A$ and define
$$
F_{k,a,a'} = 
\begin{cases}
F'_{k,a_0,a_0} + P_{k-1}^{\perp} & \text{if } a = a' = a_0\\
F'_{k,a,a'} & \text{otherwise.}
\end{cases}
$$
Note that $(F_{k,a,a'})_{a,a'}$ is a stochastic operator matrix acting on $H_{k-1}$. 
In addition, 
$$V_{k-1}^*F_{k,a_0,a_0}V_{k-1} = V_{k-1}^*(F'_{k,a_0,a_0} + P_{k-1}^{\perp})V_{k-1}
= E_{k,a_0,a_0},$$
and hence 
$$V_{k-1}^*F_{k,a,a'}V_{k-1} = E_{k,a,a'}, \ \ \ a,a'\in A.$$

By \cite[Theorem 3.1]{tt-QNS}, there exists a Hilbert space $K$ and operators $V_a : H_{k-1}\to K$ 
such that the column operator $V_k := (V_a)_{a\in A} : H_{k-1}\to K \otimes \bb{C}^A$ is an isometry, and 
$(F_{k,a,a'})_{a,a'} = V_a^*V_{a'}$, $a,a'\in A$. 
Let $\tilde{H} = K\otimes \bb{C}^A$ and $\tilde{E}_{k,a,a'} = I_K\otimes \epsilon_{a,a'}$, $a,a'\in A$. 
Then 
$V_k^*\tilde{E}_{k,a,a'}V_k = V_a^* V_{a'} = F_{k,a,a'}$ and hence, letting $V = V_k V_{k-1}$, we have that 
$V : H\to \tilde{H}$ is an isometry such that 
$V^*\tilde{E}_{k,a,a'}V = E_{k,a,a'}$, $a,a'\in A$.

Let $P_k = V_k V_k^*$ and 
$\tilde{F}_{x,a,a'} = V_k F_{x,a,a'} V_k^*$, $x\in [k-1]$, $a,a'\in A$. 
Then 
\begin{equation}\label{eq_xaabb}
\tilde{F}_{x,a,a'}\tilde{F}_{x,b,b'} = V_k F_{x,a,a'} V_k^*V_k F_{x,b,b'} V_k^* = 
\delta_{a',b} \tilde{F}_{x,a,b'}, \ \ a,a',b,b'\in A,
\end{equation}
and 
$$\sum_{a\in A} \tilde{F}_{x,a,a} = P_k.$$
Note that, if $x_0\in [k-1]$, $a_0\in A$ and 
$l = {\rm rank}(F_{x_0,a_0,a_0})$, then 
${\rm rank}(P_k) = l |A|$. 
It follows that 
$l = {\rm rank}(F_{x,a,a})$ for all $x\in [k-1]$ and all $a\in A$. 
Thus, $P_k^{\perp}(K\otimes \bb{C}^A) = K_0\otimes \bb{C}^A$ for some Hilbert space with $\dim K_0 = \dim K - l$. 

Let
$$\tilde{F}'_{x,a,a'} = I_{K_0} \otimes \epsilon_{a,a'}, \ \ \ x\in [k-1], a,a'\in A,$$
considered as an operator on $P_k^{\perp}(K\otimes \bb{C}^A)$, 
and 
$$\tilde{E}_{x,a,a'} := \tilde{F}_{x,a,a'} + \tilde{F}'_{x,a,a'}, \ \ \ x\in [k-1], a,a'\in A.$$
For $a,a',b,b'\in A$ and $x\in [k-1]$, using (\ref{eq_xaabb}) we have 
\begin{eqnarray*}
\tilde{E}_{x,a,a'}\tilde{E}_{x,b,b'} 
& = & 
(\tilde{F}_{x,a,a'} + \tilde{F}'_{x,a,a'})(\tilde{F}_{x,b,b'} + \tilde{F}'_{x,b,b'})\\
& = & \tilde{F}_{x,a,a'}\tilde{F}_{x,b,b'} + \tilde{F}'_{x,a,a'}\tilde{F}'_{x,b,b'}\\
& = & 
\delta_{a',b}\tilde{F}_{x,a,b'} + \delta_{a',b}\tilde{F}'_{x,a,b'} = \delta_{a',b}\tilde{E}_{x,a,b'}.
\end{eqnarray*}
In addition, for $x\in [k-1]$ and $a,a'\in A$ we have
\begin{eqnarray*}
V^*\tilde{E}_{x,a,a'} V 
& = & 
V^*\tilde{F}_{x,a,a'}V + V^*\tilde{F}'_{x,a,a'}V = V^*\tilde{F}_{x,a,a'}V\\
& = & V_{k-1}^*V_k^*(V_kF_{x,a,a'} V_k^*)V_kV_{k-1} = E_{x,a,a'}.
\end{eqnarray*}
\end{proof}

\noindent {\bf Remark. } In the notation of Lemma \ref{l_lift}, if $E_{x,a,a'} = \delta_{a,a'}E_{x,a,a}$ for all $x,a,a'$, 
the statement reduces to the simultaneous Naimark dilation of a finite family of POVM's exhibited 
in \cite[Theorem 9.8]{paulsen-lect}. We include the following consequence, which will be used later.


\begin{corollary}\label{c_repvsucp}
Let $X$, $Y$, $A$ and $B$ be finite sets.
A CQNS correlation $\Gamma : \cl D_{XY}\to M_{AB}$ is quantum if and only if there exist
finite dimensional Hilbert space $H_X$ and $H_Y$, 
*-representations $\pi_X : \cl B_{X,A}\to \cl B(H_X)$ and $\pi_Y : \cl B_{Y,B}\to \cl B(H_Y)$, 
and a unit vector $\xi\in H_A\otimes H_B$, such that 
$$\Gamma(\epsilon_{x,x}\otimes\epsilon_{y,y}) = \left(\langle (\pi_X(e_{x,a,a'})\otimes \pi_Y(f_{y,b,b'}))\xi,\xi\rangle
\right)_{a,a',b,b'}, \ \ x\in X, y\in Y.$$
\end{corollary}

\begin{proof}
Let 
$(E_{x,a,a'})_{x,a,a'}$ (resp. $(F_{y,b,b'})_{y,b,b'}$) be a semi-classical stochastic operator matrix
acting on finite dimensional Hilbert space $H_A$ (resp. $H_B$) and $\eta\in H_A\otimes H_B$ be a unit vector such that 
$$\Gamma(\epsilon_{x,x}\otimes\epsilon_{y,y}) = \left(\left\langle (E_{x,a,a'}\otimes F_{y,b,b'})\eta,\eta\right\rangle
\right)_{a,a',b,b'}, \ \ x\in X, y\in Y.$$
Let $(\tilde{E}_{x,a,a'})_{a,a'}$ and $V$ (resp. $(\tilde{F}_{x,a,a'})_{a,a'}$ and $W$) be the 
matrix unit systems acting on a finite dimensional Hilbert space $H_X$ (resp. $H_Y$) and 
the corresponding isometry, obtained via Lemma \ref{l_lift}. 
By the universal property of the C*-algebraic free product, there exists a *-representation 
$\pi_X : \cl B_{X,A} \to \cl B(H_X)$ (resp. $\pi_Y : \cl B_{Y,B} \to \cl B(H_Y)$) such that 
$\pi_X(e_{x,a,a'}) = \tilde{E}_{x,a,a'}$ (resp. $\pi_Y(f_{y,b,b'}) = \tilde{F}_{y,b,b'}$), 
$x\in X$, $a,a'\in A$ (resp. $y\in Y$, $b,b'\in B$). Letting $\xi = (V\otimes W)\eta$, 
we obtain the required representation of $\Gamma$. 
\end{proof}

\begin{theorem}\label{th_reststr}
Let $\varphi:\mathcal{P}_{XY}^{\rm cl} \to \mathcal{P}_{AB}$ be a quantum output mirror game,
$\tau$ be a tracial state on $\mathcal{B}_{X,A}$ and $\rho : \mathcal{B}_{Y,B} \to \mathcal{B}_{X,A}$ 
be a unital *-homomorphism such that $\Gamma = \Gamma_{\rho,\tau}$ is a 
perfect quantum commuting CQNS strategy for $\varphi$.
The following hold:
\begin{itemize}
\item[(i)] $\Gamma \in \mathcal{CQ}_{\rm qa}$ if and only if $\tau$ can be chosen to be amenable;

\item[(ii)] $\Gamma \in \mathcal{CQ}_{\rm q}$ if and only if $\tau$ can be chosen to factor through a finite-dimensional *-representation of $\mathcal{B}_{X,A}$.
    \end{itemize}
\end{theorem}


\begin{proof}
(i) 
Assume that $\Gamma \in \cl{CQ}_{\rm qa}$. 
By the Remark after \cite[Theorem 7.7]{tt-QNS}, $s$ can be chosen to be a state of $\cl B_{X,A}\otimes_{\rm min}\cl B_{Y,B}$. 
Let $\partial: \cl B_{X,A}\to\cl B_{X,A}^{\rm op}$ be the *-isomorphism given by 
$\partial(e_{x,a,a'}) = e_{x,a',a}^{\rm op}$, whose existence is guaranteed by \cite[Lemma 9.2]{tt-QNS}. 
Let $\phi : \cl B_{X,A}\otimes_{\min} \cl B_{X,A}^{\rm op} \to \bb{C}$ be the state defined by letting 
$$\phi = s\circ ({\rm id}\otimes \pi) \circ ({\rm id}\otimes \partial^{-1}).$$
Let $z\in \cl B_{X,A}$ and $w = e_{x_1,a_1,a_1'}\cdots e_{x_k,a_k,a_k'}$, for some $x_i\in X$, $a_i,a_i'\in A$, 
$i = 1,\dots,k$. 
Set $\bar{w}:= \partial^{-1}(w^{\rm op}) = e_{x_k,a_k',a_k}\cdots e_{x_1,a_1',a_1}$. 
Using (\ref{eq_uhe}), we have 
\begin{eqnarray*}
\phi (z\otimes w^{\rm op}) 
& = & 
s(z\otimes \pi(\bar{w})) = 
s(z\otimes q_{x_k,a_k',a_k}\cdots q_{x_1,a_1',a_1})\\
& = & 
\tau(z e_{x_1,a_1,a_1'}\cdots e_{x_k,a_k,a_k'}) = \tau(zw).
\end{eqnarray*}
By linearity and continuity, 
\begin{equation}\label{eq_zwop}
\phi(z\otimes w^{\rm op}) = \tau(zw), \ \ \ \ z, w\in \cl B_{X,A}.
\end{equation}
By \cite[Theorem 6.2.7]{bo}, $\tau$ is amenable.

Conversely, if $\tau$ is an amenable trace that implements $\Gamma$ then the functional 
$\phi : \cl B_{X,A}\otimes_{\max} \cl B_{X,A} \to \bb{C}$ defined via the identity (\ref{eq_zwop}) factors through 
$\cl B_{X,A}\otimes_{\min} \cl B_{X,A}$; by the Remark after \cite[Theorem 7.7]{tt-QNS}, 
$\Gamma\in \mathcal{CQ}_{\rm qa}$.

(ii) 
Let
$\Gamma:\cl D_{XY}\to M_{AB}$ be a perfect strategy in $\mathcal{CQ}_{\rm q}$.
By Corollary \ref{c_repvsucp}, 
there exist finite dimensional spaces $H$ and $K$, representations 
$\pi' : \cl B_{X,A}\to \cl B(H)$ and $\rho' : \cl B_{Y,B}\to\cl B(K)$, and a unit vector $\xi\in H\otimes K$ such that 
$\Gamma = \Gamma_s$, where 
$s : \cl B_{X,A}\otimes_{\min}\cl B_{Y,B}$ is a state such that 
\begin{equation}\label{gamma}
s\left(e_{x,a,a'} \otimes f_{y,b,b'}\right) = \left\langle(\pi'(e_{x,a,a'})\otimes\rho'(f_{y,b,b'}))\xi,\xi\right\rangle,
\end{equation}
for all $x\in X, y\in Y, a,a'\in A, b,b'\in B$.
The proof of Theorem \ref{th_qomg} shows 
that the left marginal of $s$ is a trace on 
$\cl B_{X,A}$ that factors through the finite dimensional space $H\otimes K$ and satisfies (\ref{eq_Gam}).
The converse direction follows from \cite[Proposition 9.15]{tt-QNS}. 
\end{proof}

\noindent {\bf Remark.} 
In case the bijective projections $\nph(\epsilon_{x,x} \otimes \epsilon_{f(x),f(x)})$ and 
$\nph(\epsilon_{g(y),g(y)} \otimes \epsilon_{y,y})$ from Definition \ref{d_qmir} have full rank, the corresponding quantum output mirror games reduces to a classical one and has possesses non-trivial local perfect strategies. However, if $|A| > 1$ and at least one of those projections has rank smaller than $|A|$, a local perfect strategy does not exist, since local CQNS correlations preserve separability of states.

\medskip

The following is a partial converse of Theorem \ref{th_qomg}.

\begin{proposition}
Let $\tau:\cl B_{X,A}\to\mathbb C$ be a 
tracial state and let $\rho: \cl B_{Y,B}\to\cl B_{X,A}$ be a *-homomorphism for which
there exist bijections $f:X\to Y$, $g:Y\to X$ and unitary operators 
$U_x$, $V_y:\mathbb C^B\to \mathbb C^A$, $x\in X$, $y\in Y$, such that 
$(\rho(f_{y,b,b'}))_{b,b'}=(V_y^*\otimes I)(e_{g(y),a,a'})_{a,a'}(V_y\otimes I)$ and  
$(\rho(f_{f(x),b,b'}))_{b,b'}=(U_x^*\otimes I)(e_{x,a,a'})_{a,a'}(U_x\otimes I)$. 
Then $\Gamma_{\rho,\tau}$ is a perfect strategy for the game $\varphi$ given by 
$$
\nph(\nep_{x,x}\otimes \nep_{y,y}) = 
\begin{cases}
\zeta_{\bar U_x}\zeta_{\bar U_x}^* & \text{if } y = f(x),\\
\zeta_{\bar V_y}\zeta_{\bar V_y}^* & \text{if } x = g(y),\\
I_{AA} & \text{otherwise} .
\end{cases}
$$
\end{proposition}

\begin{proof}
We have that 
$$\tau\left(e_{x,a,a'}\rho(f_{f(x),b',b})\right)_{a,a',b,b'}
= (I\otimes U_x^*)\left(\tau(e_{x,a,a'}e_{x,b',b})\right)_{a,a',b,b'}(I\otimes U_x).$$
As $\tau$ is a trace and $\{e_{x,a,a'}\}_{a,a'\in A}$ is a family of matrix units, 
$\tau(e_{x,a,a'}e_{x,b',b})=\delta_{a',b'}\delta_{a,b}\lambda_x$, where $\lambda_x=\tau(e_{x,a,a})=\tau(e_{x,b,b})$ for all $a,a',b,b'\in A$ and $x\in X$. Hence
\begin{eqnarray*}
&&\langle\Gamma(\nep_{x,x}\otimes\nep_{f(x),f(x)})(e_{a'}\otimes e_{b'}),e_a\otimes e_b\rangle\\&&=\langle (\tau(e_{x,a,a'}e_{x,b',b}))_{a,a',b,b'}(I\otimes U_x)(e_{a'}\otimes e_{b'}),( I\otimes U_x)(e_a\otimes e_b)\rangle\\&&=u^x_{a',b'}\overline{u^x_{a,b}}\lambda_x,
\end{eqnarray*}
where $U_x=(u^x_{a,b})_{a,b}$. 
On the other hand,
$$
\langle\zeta_{\bar U_x}\zeta_{\bar U_x}^*e_{a'}\otimes e_{b'},e_a\otimes e_b\rangle=\langle\zeta_{\bar U_x},e_a\otimes e_b\rangle\langle e_{a'}\otimes e_{b'},\zeta_{\bar U_x}\rangle=u^x_{a',b'}\overline{u^x_{a,b}}
$$
showing that $\Gamma(\nep_{x,x}\otimes\nep_{f(x),f(x)})=\lambda_x\zeta_{\bar U_x}\zeta_{\bar U_x}^*$ and hence for $P=\nep_{x,x}\otimes\nep_{f(x),f(x)}$, we obtain 
\begin{equation}\label{game_con}
\langle\Gamma(P),\varphi(P)_\perp\rangle=0.
\end{equation}
Similar arguments give (\ref{game_con}) for $P=\nep_{g(y),g(y)}\otimes\nep_{y,y}$. 
\end{proof}

The classical-to-quantum \emph{concurrency game} is the game $\nph : \cl P_{XX}^{\rm cl}\to \cl P_{AA}$ 
defined as follows:
$$\nph(\nep_{x,x}\otimes \nep_{y,y}) = 
\begin{cases}
J_A & \text{if } x = y\\
I_{AA} & \text{if } x \neq y.
\end{cases}
$$
A CQNS correlation $\Gamma$ will be called concurrent if $\Gamma$ is a perfect strategy for the 
concurrency game.

\begin{corollary}\label{c_cqt}
Let 
$\Gamma : \cl D_{XX}\to M_{AA}$ be a quantum commuting CQNS correlation. 
The following are equivalent: 
\begin{itemize}
\item[(i)] $\Gamma$ is concurrent;
\item[(ii)]  there exists a tracial state $\tau : \cl B_{X,A}\to \bb{C}$ such that 
\begin{equation}\label{eq_taunph}
\Gamma(\epsilon_{x,x} \otimes \epsilon_{y,y}) = \left(\tau(e_{x,a,a'}e_{y,b',b})\right)_{a,a',b,b'}, \ \ \ x,y\in X.
\end{equation}
\end{itemize}
Moreover, 
\begin{itemize}
\item[(i')] 
$\Gamma \in \cl{CQ}_{\rm qa}$ if and only if the trace $\tau$ can be chosen to be amenable;

\item[(ii')] 
$\Gamma \in \cl{CQ}_{\rm q}$ if and only if $\tau$ can be chosen to factor through a finite dimensional *-representation of 
$\cl B_{X,A}$.
\end{itemize}
\end{corollary}
\begin{proof}
(i)$\Rightarrow$(ii) 
The concurrency game is a 
quantum output mirror  game with $B = A$, $f$ and $g$ the identity maps, and 
$\varphi (\epsilon_{x,x}\otimes\epsilon_{x,x}) = J_A = \zeta_{I_A}\zeta_{I_A}^*$ for every $x\in X$. 
In this case the *-homomorphism $\rho:\cl B_{X,A}\to\cl B_{X,A}$ from the proof of Theorem \ref{th_qomg} is given by 
$\rho(f_{y,b,b'})=e_{y,b,b'}$. The statement now follows from Theorem \ref{th_qomg}.

(ii)$\Rightarrow$(i) 
Fix $x\in X$ and note that, by the uniqueness of the trace on $M_A$, the restriction of $\tau$ 
to any of the free product terms in the definition of $\cl B_{X,A}$ coincides with the normalised trace ${\rm tr}$;
thus, 
\begin{equation}
\tau(e_{x,a,a'}) = \frac{1}{|A|} \delta_{a,a'}, \ \ \ a,a'\in A.
\end{equation}
It follows that
\begin{eqnarray*}
\Gamma(\nep_{x,x}\otimes \nep_{x,x}) 
& = & 
\sum_{a,a',b,b'\in A} \tau(e_{x,a,a'}e_{x,b',b}) \nep_{a,a'}\otimes \nep_{b,b'}\\
& = & 
\sum_{a,b,b'\in A} \tau(e_{x,a,b'}e_{x,b',b}) \nep_{a,b'}\otimes \nep_{b,b'}\\
& = & 
\sum_{a,b,b'\in A} \tau(e_{x,a,b}) \nep_{a,b'}\otimes \nep_{b,b'}
=
\frac{1}{|A|}  \sum_{a,b\in A} \nep_{a,b}\otimes \nep_{a,b} =J_A.
\end{eqnarray*}

Statements (i') and (ii') are immediate from statements (i) and (ii) in Theorem \ref{th_reststr}.
\end{proof}

\begin{remark}\label{r_fact}
{\rm 
Factorisable quantum channels were introduced in \cite{delaroche} and have been subsequently 
studied by a number of authors (see \cite{mr-2019} and the references therein). 
It was shown in \cite[Proposition 3.1]{mr-2019} that a quantum channel $\Phi : M_A\to M_A$ is 
factorisable if and only if its Choi matrix has the form 
$\left(\tau(p_{a,a'}q_{b',b})\right)_{a,a',b,b'}$, for some matrix unit systems $(p_{a,a'})_{a,a'}$ and 
$(q_{b,b'})_{b,b'}$ in $M_A\ast_1 M_A$. 
It follows that the factorisable quantum 
channels on $M_A$ can be identified with the perfect quantum commuting CQNS strategies for concurrent games with two inputs.   
Note, in addition, that 
the perfect quantum commuting strategies of quantum output mirror games with a single input  form a subclass of the factorisable quantum 
channels. } 
\end{remark}

\begin{remark}\label{r_tracial}
{\rm 
Let $\cl A$ be a unital C*-algebra, equipped with a tracial state $\tau$. 
Recall \cite{tt-QNS} that a \emph{semi-stochastic} 
$\cl A$-matrix over $(X,A)$ is a positive matrix 
$(q_{x,a,a'})_{x,a,a'}\in \cl D_X\otimes M_A\otimes \cl A$ such that $\sum_{a\in A} g_{x,a,a} = 1$ for all $x\in X$. 
A CQNS correlation $\Gamma : \cl D_{XX}\to M_{AA}$ is called \emph{tracial} \cite{tt-QNS} if
there exists a semi-stochastic $\cl A$-matrix $(g_{x,a,a'})_{x,a,a'}$ such that 
$$\Gamma(\epsilon_{x,x}\otimes\epsilon_{y,y}) = \sum_{a,a',b,b'} 
\tau(g_{x,a,a'}g_{y,b',b}) \epsilon_{a,a'}\otimes\epsilon_{b,b'}, \ \ x,y\in X.$$
It follows from Corollary \ref{c_cqt} that every concurrent quantum commuting CQNS correlation is tracial.
}
\end{remark}


\subsection{Algebras of classical-to-quantum games}\label{ss_acq}

Let $P\in \cl P^{\rm cl}_{XX}$ and $Q\in \cl P_{AA}$. We define a linear map 
$$\beta_{P,Q} : \cl D_{XX}\otimes M_{AA}\otimes \frak{B}_{X,A}\otimes \frak{B}_{X,A}^{\rm op}\to \frak{B}_{X,A}$$
by letting 
$$\beta_{P,Q}\left(\omega\otimes u \otimes v^{\rm op}\right) = \Tr(\omega (P\otimes Q)) uv, \ \ \omega\in \cl D_{XX}\otimes M_{AA}, 
u,v\in \frak{B}_{X,A}.$$
When, in addition, $Q\in \cl P^{\rm cl}_{AA}$, define a corresponding map
$$\alpha_{P,Q} : \cl D_{XX}\otimes \cl D_{AA}\otimes \frak{A}_{X,A}\otimes \frak{A}_{X,A}^{\rm op}\to \frak{A}_{X,A}.$$
Both $\beta_{P,Q}$ and $\alpha_{P,Q}$ will be considered as maps on the ampliations of the 
algebraic tensor products $\cl B_{X,A}\otimes \cl B_{X,A}^{\rm op}$ and $\cl A_{X,A}\otimes \cl A_{X,A}^{\rm op}$, 
with values in $\cl B_{X,A}$ and $\cl A_{X,A}$, respectively. 

We use the notation $\langle S\rangle$ to refer to the *-ideal generated by a subset $S$ of a *-algebra. 
If $\nph : \cl P^{\rm cl}_{XX}\to \cl P_{AA}$ is a classical-to-quantum game, set
$$\frak{I}(\nph) = \left\langle \left\{\beta_{P,\nph(P)^{\perp}}(E_{\rm cq}\otimes E_{\rm cq}^{\rm op}) : 
P\in \cl P^{\rm cl}_{XX}\right\}\right\rangle \subseteq \frak{B}_{X,A},$$
and let $\cl I(\nph)$ be the closure of $\frak{I}(\nph)$ in $\cl B_{X,A}$. 
Set $\frak{B}(\nph) = \frak{B}_{X,A}/\frak{I}(\nph)$ and $\cl B(\nph) = \cl B_{X,A}/\cl I(\nph)$.
Define $\frak{A}(\nph)$ and $\cl A(\nph)$ similarly, using the ideal
$$\left\langle \left\{\alpha_{P,\nph(P)^{\perp}}(E_{\rm cl}\otimes E_{\rm cl}^{\rm op}) : P\in \cl P^{\rm cl}_{XX}\right\}\right\rangle$$
of $\frak{A}_{X,A}$.

Given a synchronous non-local game $\lambda : X\times X\times A\times A\to \{0,1\}$, its *-algebra 
$\cl A(\lambda)$ was defined in \cite{hmps} as the unital *-algebra with generators selfadjoint idempotents $e'_{x,a}$, 
where $x\in X$, $a\in A$, subject to the relations 
$$\sum_{a\in A} e'_{x,a} = 1 \mbox{ for all } x\in X, \mbox{ and } e'_{y,b}e'_{z,c} = 0 \mbox{ if } \lambda(y,z,b,c) = 0.$$

\begin{proposition}\label{p_consis}
Let $\lambda : X\times X\times A\times A\to \{0,1\}$ be a synchronous non-local game. 
Then $\frak{A}(\nph_{\lambda})$ (resp. $\cl A(\nph_{\lambda})$) coincides with the *-algebra (resp. C*-algebra)
of the game $\lambda$. 
\end{proposition}

\begin{proof}
Let $\frak{A}(\lambda)$ be the *-algebra of the game $\lambda$ as defined in \cite{hmps}, and note that 
$\frak{A}(\lambda) = \frak{A}_{X,A}/ \frak{I}(\lambda)$, where 
$$\frak{I}(\lambda) = \langle e_{x,a} e_{y,b} : \lambda(x,y,a,b) = 0\rangle.$$
We show that 
\begin{equation}\label{eqInph}
\frak{I}(\lambda) = \frak{I}(\nph_{\lambda}).
\end{equation}
Note that 
$\varphi_\lambda(\epsilon_{x,x}\otimes \epsilon_{y,y})^\perp=\sum_{(a,b):\lambda(x,y,a,b)=0 }\epsilon_{a,a}\otimes \epsilon_{b,b}$; thus, 
$$\alpha_{\epsilon_{x,x}\otimes \epsilon_{y,y}, \nph_{\lambda}(\epsilon_{x,x}\otimes \epsilon_{y,y})^{\perp}}
(E_{\rm cl}\otimes E_{\rm cl}^{\rm op})=\sum_{(a,b):\lambda(x, y,a,b)=0}e_{x,a}e_{y,b}.$$ 
Multipying  from the left by $e_{x,a}$ and by $e_{y,b}$ from the right, 
we conclude that $e_{x,a}e_{y,b} \in \frak{I}(\nph_{\lambda})$ whenever $\lambda(x,y,a,b) = 0$; 
thus, $\frak{I}(\lambda) \subseteq \frak{I}(\nph_{\lambda})$. 

Let $P = \sum_k \nep_{x_k,x_k}\otimes \nep_{y_k,y_k}$ as a finite sum. Then 
$$\varphi_\lambda(P)^\perp=\sum_{(a,b):\lambda(x_k,y_k,a,b) = 0, \forall k} \nep_{a,a}\otimes \nep_{b,b}$$ 
and hence 
$$\alpha_{P,\varphi_{\lambda}(P)^\perp}(E_{\rm cl}\otimes E_{\rm cl}^{\rm op})
= \sum_{(a,b):\lambda(x_k,y_k,a,b)=0, \forall k}\sum_k e_{x_k,a}e_{y_k,b}.$$ 
This shows that 
$\frak{I}(\varphi_\lambda)\subseteq \frak{I}(\lambda)$, 
establishing (\ref{eqInph}).
\end{proof}

\begin{remark}\label{r_Ecq}
We have $\beta_{J^{\rm cl}_X,J_A^{\perp}}(E_{\rm cq}\otimes E_{\rm cq}^{\rm op}) = 0$.
\end{remark}

\begin{proof}
The claim follows from the fact that
\begin{eqnarray*}
\beta_{J^{\rm cl}_X,J_A}\left(E_{\rm cq}\otimes E_{\rm cq}^{\rm op}\right)
& = & 
\frac{1}{|A|}\sum_{x\in X}\sum_{a,b\in A} e_{x,a,b} e_{x,b,a}
=
\sum_{x\in X}\sum_{a\in A} e_{x,a,a}\\
& = & 
\sum_{x\in X}\sum_{a,b\in A} e_{x,a,a} e_{x,b,b}
=
\beta_{J^{\rm cl}_X,I_{AA}}\left(E_{\rm cq}\otimes E_{\rm cq}^{\rm op}\right).
\end{eqnarray*}
\end{proof}

\begin{corollary}\label{c_trcstog}
Let $\nph : \cl P_{XX}^{\rm cl} \to \cl P_{AA}$ be a classical-to-quantum concurrent game. 
The following are equivalent for a CQNS correlation $\Gamma : \cl D_{XX}\to M_{AA}$:
\begin{itemize}
\item[(i)] $\Gamma$ is a perfect quantum commuting (resp. quantum) strategy for $\nph$;
\item[(ii)] there exists trace $\tau$
(resp. a trace $\tau$ that factors through a finite dimensional *-representation)
on $\cl B_{X,A}$ such that (\ref{eq_taunph}) holds and 
$$\tau \left(\beta_{P,\nph(P)^\perp} \left(E_{\rm cq}\otimes E_{\rm cq}^{\rm op}\right)\right) = 0,$$
for all $P\in \cl P_{XX}^{\rm cl}$.
\end{itemize}
\end{corollary}

\begin{proof}
We only prove the statement in the case of quantum commuting strategies. 
Let $\tau$ be the trace of $\cl B_{X,A}$ that implements $\Gamma$, arising from Theorem \ref{th_qomg}.
For any $P\in\cl P^{\rm cl}_{XX}$, taking into account the duality relations 
(\ref{eq_dtom}), we have
\begin{eqnarray*}
0 & = & \left\langle \Gamma(P),\nph(P)_\perp\right\rangle
= \sum_{x,y\in X} \Tr\left((\epsilon_{x,x}\otimes\epsilon_{y,y}) P\right) 
\left\langle \Gamma\left(\epsilon_{x,x}\otimes\epsilon_{y,y}\right),\varphi(P)_\perp\right\rangle\\
& = &
\sum_{x,y\in X} \Tr\left((\epsilon_{x,x}\otimes\epsilon_{y,y}) P\right) 
\sum_{a,a',b,b'} \tau\left(e_{x,a,a'}e_{y,b',b}\right)\left\langle \epsilon_{a,a'}\otimes\epsilon_{b,b'},\varphi(P)_\perp\right\rangle\\
& = & 
\sum_{x,y\in X} \Tr\left((\epsilon_{x,x}\otimes\epsilon_{y,y}) P\right) 
\sum_{a,a',b,b'} \tau\left(e_{x,a,a'}e_{y,b',b}\right)
\Tr((\epsilon_{a,a'}\otimes\epsilon_{b,b'})\varphi(P)^{\perp})\\
& = &
\tau \left(\beta_{P,\nph(P)^\perp} \left(E_{\rm cq}\otimes E_{\rm cq}^{\rm op}\right)\right).
\end{eqnarray*}
\end{proof}

\begin{remark}\label{r_amenap}
\rm
Clearly, any trace $\tau$ on $\cl B(\varphi)$ gives rise to a perfect quantum commuting strategy for $\varphi$.  
If, in particular, $\tau$ is amenable on  $\cl B(\nph)$, by 
 \cite[Proposition 6.3.5]{bo}, the induced trace $\tilde{\tau}$ on $\cl B_{X,A}$ is amenable, and hence the CQNS correlation defined via (\ref{eq_taunph}) is approximately quantum. 
We do not know if any perfect quantum commuting strategy for a non-local game $\nph$
arises from a trace of $\cl B(\nph)$ in general.  Similarly we are not aware  if any the approximately quantum perfect strategies of a classical-to-quantum non-local game $\nph$
all arise from amenable traces of $\cl B(\nph)$.
\end{remark}


\section{Concurrent quantum games} \label{s_congames}

In this section, we define the *-algebra and the C*-algebra of a quantum concurrent game and provide a characterisation of the prefect strategies for this type of games.


\subsection{Tracial descriptions}\label{ss_trdes}

Let $\tau : \cl C_{X,A}\to \bb{C}$ be a tracial state; then the linear map $\Gamma_{\tau} : M_{XX}\to M_{AA}$, given by $$\Gamma_{\tau} (\epsilon_{x,x'} \otimes \epsilon_{y,y'}) 
= \sum_{a,a',b,b'}\tau(e_{x,x'a,a'}e_{y',y,b',b}) \epsilon_{a,a'} \otimes \epsilon_{b,b'},$$
is a QNS correlations; the QNS correlations arising in this way were called \emph{tracial} in \cite{tt-QNS}.
The classes of \emph{quantum tracial} 
(resp. \emph{locally tracial}) QNS correlations are 
defined by requiring that $\tau$ factors through a finite dimensional (resp. abelian) *-representation.

\begin{theorem}\label{th_qtnoteq1}
Let $X$ and $A$ be finite sets, 
$\nph : \cl P_{XX} \to \cl P_{AA}$ be a concurrent game and 
$\Gamma : M_{XX}\to M_{AA}$ be a perfect quantum commuting QNS strategy for $\nph$. 
Then there exists a tracial state $\tilde{\tau} : \cl C_{X,A}\to \bb{C}$ such that $\Gamma = \Gamma_{\tilde{\tau}}$. 
Moreover, 
\begin{itemize}
\item[(i)] 
if $\Gamma \in \cl{Q}_{\rm qa}$ then $\tilde{\tau}$ can be chosen to be amenable;

\item[(ii)] 
if $\Gamma \in \cl{Q}_{\rm q}$ then $\tilde{\tau}$ can be chosen to factor through 
a finite dimensional *-representation of $\cl C_{X,A}$; 

\item[(iii)] 
if $\Gamma \in \cl{Q}_{\rm loc}$ then $\tilde{\tau}$ can be chosen to factor through an abelian *-representation of 
$\cl C_{X,A}$.
\end{itemize}
\end{theorem}

\begin{proof}
Let $\Gamma\in \cl{Q}_{\rm qc}$ be a perfect strategy for $\nph$. 
By \cite[Theorem 6.3]{tt-QNS}, there exists a state 
$s : \cl C_{X,A}\otimes_{\max} \cl C_{X,A}\to \bb{C}$ such that 
\begin{equation}\label{eq_strep2}
\Gamma(\nep_{x,x'}\otimes \nep_{y,y'}) 
= \sum_{a,a',b,b'} s(e_{x,x',a,a'}\otimes f_{y,y',b,b'}) \nep_{a,a'}\otimes \nep_{b,b'},
\end{equation} 
for all $x,x',y,y'\in X$ and all $a,a',b,b'\in A$
(for clarity, we use $f_{y,y',b,b'}$ to denote the canonical generators of the second 
copy of $\cl C_{X,A}$). 
It follows that 
$$\frac{1}{|X|}\sum_{x,y}\sum_{a,a',b,b'} s(e_{x,y,a,a'}\otimes f_{x,y,b,b'}) \nep_{a,a'}\otimes \nep_{b,b'} = J_A,$$
and hence 
\begin{equation}\label{eq_qco}
\sum_{x,y} s(e_{x,y,a,b}\otimes f_{x,y,a,b}) = \frac{|X|}{|A|}, \ \ \ a,b\in A.
\end{equation}
Let $V=(v_{a,x})_{a,x}$ be the isometry such that $e_{x,x',a,a'}=v_{a,x}^*v_{a',x'}$. 
Then $$VV^* = \left(\sum_{x\in X} v_{a,x} v_{b,x}^*\right)_{a,b}$$
is a projection, and hence
$$\sum_{x\in X} v_{a,x} v_{a,x}^* \leq 1, \ \ \ a\in A.$$
It follows that 
\begin{eqnarray}
\sum_{x\in X} e_{y,x,b,a}e_{x,y,a,b} 
& = & 
\sum_{x\in X} v_{b,y}^* v_{a,x} v_{a,x}^* v_{b,y}\\
& \leq & 
v_{b,y}^*v_{b,y}=e_{y,y,b,b} \nonumber
\end{eqnarray} 
for all $y\in X$ and all $a,b\in A$.
Thus, 
\begin{eqnarray}\label{eq_sumabe}
\sum_{x,y \in X} \sum_{a,b\in A} e_{y,x,b,a}e_{x,y,a,b} 
& \leq & 
\sum_{y \in X} \sum_{a,b\in A} e_{y,y,b,b}\nonumber\\
& = & 
|A| \sum_{y\in X} \sum_{b\in A} e_{y,y,b,b}
= 
|X||A| 1. 
\end{eqnarray}
Similarly, 
\begin{equation}\label{eq_sumabewarn}
\sum_{y\in X} f_{x,y,a,b}f_{y,x,b,a} \leq f_{x,x,a,a}
\end{equation} and 
\begin{equation}\label{eq_sumabf}
\sum_{x,y \in X} \sum_{a,b\in A}  f_{x,y,a,b}f_{y,x,b,a} \leq |X||A|1.
\end{equation}

Let 
$$h_{x,y,a,b} = e_{x,y,a,b}\otimes 1 - 1 \otimes f_{y,x,b,a}, \ \ \ x,y\in X, a,b\in A.$$
Equation (\ref{eq_qco}) and inequalities  (\ref{eq_sumabe}) and (\ref{eq_sumabf}) imply
\begin{eqnarray*} 
& & 
\sum_{x,y,a,b} s(h_{x,y,a,b}^*h_{x,y,a,b})\\
& = & 
\sum_{x,y,a,b} s\left(\left(e_{y,x,b,a}\otimes 1 - 1 \otimes f_{x,y,a,b}\right) \left(e_{x,y,a,b}\otimes 1 - 1 \otimes f_{y,x,b,a}\right)\right)\\
& = & 
\sum_{x,y,a,b} s\left(e_{y,x,b,a}e_{x,y,a,b}\otimes 1  + 1 \otimes f_{x,y,a,b}f_{y,x,b,a}\right)\\
& - & \sum_{x,y,a,b} s\left(e_{y,x,b,a} \otimes f_{y,x,b,a} + e_{x,y,a,b} \otimes f_{x,y,a,b}\right)\\
& \leq & 
2|X||A| 1 - 2|X||A| 1 = 0.
\end{eqnarray*}
It follows that 
\begin{equation}\label{eq_hxyab}
s(h_{x,y,a,b}^*h_{x,y,a,b}) = 0, \ \ \ x,y \in X, a,b\in A.
\end{equation}
As in the proof of Theorem \ref{th_qomg}, write $u\sim v$ if $s(u - v) = 0$ and note that, by 
(\ref{eq_hxyab}),
$$u h_{x,y,a,b} \sim 0 \mbox{ and } h_{x,y,a,b} u \sim 0, \ \ x,y\in X, a,b\in A, \ u\in \cl C_{X,A} \otimes_{\max}\cl C_{X,A}.$$
In particular, 
\begin{equation}\label{eq_zeez4}
ze_{x,y,a,b}\otimes 1 \sim z\otimes f_{y,x,b,a} \sim e_{x,y,a,b} z\otimes 1, \ \ x,y\in X, a,b\in A, \ z\in \cl C_{X,A}.
\end{equation}
Using (\ref{eq_zeez4}) and induction, as in the proof of Theorem \ref{th_qomg}, we conclude that 
the map $\tilde{\tau} : \cl C_{X,A}\to \bb{C}$, given by 
$\tilde{\tau}(z) = s(z\otimes 1)$, is a trace on $\cl C_{X,A}$.
Identity (\ref{eq_zeez4}) implies 
$$s\left(e_{x,x',a,a'}\otimes f_{y,y',b,b'}\right) = \tilde \tau\left(e_{x,x',a,a'} e_{y',y,b',b}\right), \ x,x',y,y' \in X, a,a',b,b'\in A.$$

Statements (i)-(ii) are proved similarly to Corollary \ref{c_cqt}.
To see (iii), let $\Gamma$ be a perfect strategy of class $\cl Q_{\rm loc}$. We have $\Gamma=\sum_{i=1}^n\lambda_i\Phi_i\otimes\Psi_i$ as a convex linear combination of quantum channels $\Phi_i,\Psi_i: M_X\to M_A$, 
$i = 1,\dots,n$.
Let $\left(\lambda^{(j)}_{x,x',a,a'}\right)_{a,a'} = \Phi_j\left(\nep_{x,x'}\right)$, $x,x'\in X$,
$\left(\mu^{(j)}_{y,y',b,b'}\right)_{b,b'} = \Psi_j\left(\nep_{y,y'}\right)$, $y,y'\in X$, 
$\pi_j, \rho_j : \cl C_{X,A}\to\mathbb C$  be the 
*-representations given by
$\pi_j(e_{x,x',a,a'}) = \lambda^{(j)}_{x,x',a,a'}$ and $\rho_j(f_{y,y',b,b'}) = \mu^{(j)}_{y,y',b,b'}$, and 
$\pi',\rho' : \cl C_{X,A}\to \cl B(\mathbb C^n)$ 
be the *-representations given by 
$$\pi'(u) = \sum_{j=1}^n \pi_j(u) \epsilon_{j,j}, 
\ \ \ \rho'(v) = \sum_{j=1}^n\rho_j(v) \epsilon_{j,j}.$$ 
The images of $\pi'$ and $\rho'$ are abelian. 
Set $\xi = \sum_{i=1}^n\sqrt{\lambda_i}e_i\otimes e_i\in \mathbb C^n\otimes\mathbb C^n$; then 
$$\Gamma\left(\nep_{x,x'}\otimes \nep_{y,y'}\right) = \left(\left\langle(\pi'(e_{x,x',a,a'})\otimes\rho'(f_{y,y',b,b'}))\xi,\xi\right\rangle\right)_{a,a',b,b'}$$ and the corresponding state $s$ is given by 
$$
s(e_{x,x',a,a'}\otimes f_{y,y',b',b})=\left\langle(\pi'(e_{x,x',a,a'})\otimes\rho'(f_{y,y',b,b'}))\xi,\xi\right\rangle. 
$$
It follows that the left marginal of $s$ is a trace on $\cl C_{X,A}$ that factors through the abelian representation $\pi'$ of $\cl C_{X,A}$. 
\end{proof}

We now assume that $X=A$; we will see that in this case, we can obtain more precise conditions than the ones in Theorem \ref{th_qtnoteq1} that are also sufficient. 
Let $\cl B_X$ be the universal C*-algebra (usually referred to as the Brown algebra), 
generated by the elements $u_{a,x}$, $x,a\in X$ such that the matrix $(u_{a,x})_{a,x\in A}$ is unitary.  Consider the C*-subalgebra $\cl C_X$ of $\cl B_X$ generated by $p_{x,x',a,a'}=u_{a,x}^*u_{a',x'}$, $x$, $x'$, $a$, $a'\in X$. 
Write $\cl J$ for the closed ideal of $\cl C_{X,A}$, generated by the elements $$\sum_{x\in X} e_{y,x,b,a}e_{x,y,a,b}-e_{y,y,b,b}, \ \ \ y, a, b\in X.$$

Let $\cl V_{X,A}$ be the universal TRO of an isometry $(v_{a,x})_{a,x}$, 
as defined in \cite[Section 5]{tt-QNS}. In the sequel, we will consider products
$v_{a_1,x_1}^{\varepsilon_1}v_{a_2,x_2}^{\varepsilon_2}\cdots v_{a_k,x_k}^{\varepsilon_k}$, where 
$\varepsilon_i$ is either the empty symbol or $\ast$, and $\varepsilon_i\neq \varepsilon_{i+1}$ for all $i$, as elements of either $\cl V_{X,A}$, $\cl V_{X,A}^*$, $\cl C_{X,A}$ or the left C*-algebra corresponding to the TRO $\cl V_{X,A}$.

\begin{lemma}\label{quotient}
The map  $\pi: e_{x,x',a,a'}\mapsto p_{x,x',a,a'}$, $x,x',a,a'\in X$ extends to  a surjective
*-homomorphism $\pi:\cl C_{X,A}\to \cl C_X$ with $\ker\pi= \cl J$.
\end{lemma}

\begin{proof}
Since $U=(u_{a,x})$ is unitary and hence an isometry, we have that 
$E=(p_{x,x',a,a'})_{x,x',a,a'}$ is a stochastic operator matrix; 
thus, there exists a *-homomorphism $\pi : \cl C_{X,A}\to \cl C_X$ such that 
$\pi(e_{x,x',a,a'}) = p_{x,x',a,a'}$. 
We have
$$\pi\left(\sum_{x\in X}e_{y,x,b,a}e_{x,y,a,b}-e_{y,y,b,b}\right)
= 
\sum_{x\in X} u_{b,y}^* u_{a,x} u_{a,x}^* u_{b,y}- u_{b,y}^*u_{b,y} = 0,$$
showing that $\cl J\subseteq \ker\pi$. 

For the reverse inclusion, let $\theta : \cl C_{X,A}\to \cl B(K)$ be a unital *-representation. 
By \cite[Lemma 5.1]{tt-QNS}, there exists a block operator matrix $V = (V_{a,x})_{a,x}$ that is an isometry,
such that $\theta(e_{x,x',a,a'}) = V^*_{a,x}V_{a',x'}$, $x,x',a,a'\in X$; we write $\theta = \theta_V$.
Note that $\theta$ annihilates $\cl J$ if and only if  
$$V_{b,y}^*V_{b,y} - \sum_{x\in X}
V_{b,y}^*V_{a,x} V_{a,x}^*V_{b,y} = 0, \ \ \ a,b, y\in X.$$  
Letting $D_a = 1 - \sum_{x\in X} V_{a,x} V_{a,x}^*$, we have that $D_a$ is positive, $D_a^{1/2}V_{b,y} = 0$ and hence 
\begin{equation}\label{da}
D_aV_{b,y} = \left(1 - \sum_{x\in X} V_{a,x} V_{a,x}^*\right) V_{b,y} = 0.
\end{equation}
Since
$(V_{b,y}\otimes I)^*(I-VV^*)(V_{b,y}\otimes I) \in M_X(\cl \theta(\cl C_{X,A}))^+$ and has zeros on its main diagonal, 
it is the zero operator. 
In particular, 
$$(I-VV^*)^{1/2}(V_{b,y}\otimes I) = (I - VV^*) (V_{b,y}\otimes I) = 0, \ \ y,b \in A,$$ 
implying that 
\begin{equation}\label{eq_3}
\sum_{x\in X} V_{a,x}V_{a',x}^*V_{b,y} = 0 \ \mbox{ whenever } a\neq a'.
\end{equation}

The block operator matrix
$$U := \left(\begin{array}{cc}V&I-VV^*\\0&V^*\end{array}\right)$$
is unitary; let 
$$U_{a,x}
=
\left(\begin{array}{cc}
V_{a,x} & \delta_{a,x}I - \sum_{b\in X} V_{a,b}V_{x,b}^*\\
0 & V_{x,a}^*
\end{array}\right).$$
Since $(U_{a,x})_{a,x}$ is unitary, it gives rise to a *-representation $\rho_U$ of $\cl B_X$ on 
the Hilbert space $K_1\oplus K_2$, where $K_1 = K_2 = K$. 
Using (\ref{da}) and (\ref{eq_3}), we have 
\begin{eqnarray*}
\rho_U (p_{x,x',a,a'})
& = &
U_{a,x}^*U_{a',x'} 
= \left(\begin{array}{cc}
V_{a,x}^* & 0\\ 
\delta_{a,x}I - \sum_{b\in X} V_{x,b}V_{a,b}^* & V_{x,a}
\end{array}\right)\\
& & 
\times\left(\begin{array}{cc}
V_{a',x'} & \delta_{a',x'}I - \sum_{b\in X} V_{a',b}V_{x',b}^*\\
0 & V_{x',a'}^*
\end{array}\right)\\
& & 
= \left(\begin{array}{cc}
V_{a,x}^*V_{a',x'} & * \\ 
(\delta_{a,x}I - \sum_{b\in X} V_{x,b} V_{a,b}^*) V_{a',x'} & *
\end{array}\right)\\
& = & 
\left(\begin{array}{cc}
V_{a,x}^*V_{a',x'} & *\\
0 & *
\end{array}\right).
\end{eqnarray*} 
It follows that $K_1$ is an invariant subspace for $\rho_U|_{\cl C_X}$, and 
$\theta_V(e_{x,x',a,a'}) = \rho_U(p_{x,x',a,a'})|_{K_1}$ for all $x,x',a,a'\in X$.
This yields $\theta_V(T) = \rho_U(\pi(T))|_{K_1}$, $T\in \cl C_{X,A}$. 
Thus, for a fixed $T\in \cl C_{X,A}$, we have 
\begin{eqnarray*}
\|T+ \cl J\|
& = &
\sup \{\|\theta_V(T)\|: V = (V_{a,x}) \mbox{ isometry with } \theta_V(\cl J)=\{0\}\}\\
& \leq & 
\sup\{\|\rho_U(\pi(T))\| : U = (U_{a,x}) \mbox{ unitary}\}=\|\pi(T)\|.
\end{eqnarray*}
Therefore $T\in \ker\pi$ implies $T\in \cl J$.
\end{proof}

Note that, according to Lemma \ref{quotient}, we have $\cl C_{X,A}/\cl J \cong \cl C_X$.

\begin{theorem}\label{th_qt}
Let $X$ be a finite set and
$\nph : \cl P_{XX} \to \cl P_{XX}$ be a concurrent game. A quantum commuting QNS correlation 
$\Gamma : M_{XX}\to M_{XX}$ is a perfect strategy for $\nph$ if and only if there exists 
a tracial state $\tau : \cl C_{X}\to \bb{C}$ such that 
$$\Gamma(\epsilon_{x,x'} \otimes \epsilon_{y,y'}) = \left(\tau(p_{x,x'a,a'}p_{y',y,b',b})\right)_{a,a',b,b'}, \ \ \ x,x',y,y'\in X.$$
Moreover, 
\begin{itemize}
\item[(i)] 
$\Gamma \in \cl{Q}_{\rm q}$ if and only if $\tau$ can be chosen to factor through a finite dimensional *-representation of 
$\cl C_{X}$; 

\item[(ii)] 
$\Gamma \in \cl{Q}_{\rm loc}$ if and only if $\tau$ can be chosen to factor through an abelian *-representation of 
$\cl C_{X}$.
\end{itemize}
\end{theorem}

\begin{proof}
For clarity, we set $A = X$. 
Let $\Gamma\in \cl{Q}_{\rm qc}$ be a perfect strategy for $\nph$. 
Keeping the notation from the proofs of Theorem \ref{th_qtnoteq1} and Lemma \ref{quotient}, we see that 
$$s\left(\sum_{x\in X}e_{y,x,b,a}e_{x,y,a,b}\otimes 1 \right) = s\left(e_{y,y,b,b}\otimes 1\right)$$ 
(for otherwise we would have $\sum_{x,y,a,b} s(h_{x,y,a,b}^*h_{x,y,a,b}) < 0$).
It follows that the trace 
$\tilde\tau$ on $\cl C_{X,A}$
annihilates the elements 
$$d_{y,a,b} := e_{y,y,b,b}-\sum_{x\in X} e_{y,x,b,a}e_{x,y,a,b}, \ \ \ y\in X, a,b\in A.$$ 
As $d_{y,a,b}\geq 0$,  
we have that $\tilde\tau(d_{y,a,b}^{1/2}u) = 0$ for every $u\in \cl C_{X,A}$ and so $\tilde\tau(\cl J) = \{0\}$.
By Lemma \ref{quotient}, $\tau(\pi(u)):=\tilde\tau(u)$ is a well-defined trace on $\cl C_X$. 
Identity (\ref{eq_zeez4}) implies that 
$$s\left(e_{x,x',a,a'}\otimes f_{y,y',b,b'}\right) = 
\tau\left(p_{x,x',a,a'} p_{y',y,b',b}\right), \  x,x',y,y' \in X, a,a',b,b'\in A.$$

\medskip

Conversely,
let $\tau$ be a trace on $\cl C_X$ and 
$\Gamma : M_{XX}\to M_{AA}$ be the QNS correlation, given by 
$$\Gamma(\epsilon_{x,x'}\otimes \epsilon_{y,y'})
= 
\sum_{a,a',b,b'\in X} 
\tau(p_{x,x',a,a'}p_{y',y,b',b}) 
\epsilon_{a,a'}\otimes \epsilon_{b,b'}.$$
Write 
$$w_{a,b} = \sum_{x\in X} u_{a,x}^* u_{b,x}, \ \ \ a,b\in X.$$
We have that 
\begin{eqnarray*}
\Gamma(J_X)
& = & 
\frac{1}{|X|}\sum_{x,y\in X}\sum_{a,a',b,b'\in X} 
\tau(p_{x,y,a,a'}p_{y,x,b',b}) \epsilon_{a,a'}\otimes \epsilon_{b,b'}\\
& = & 
\frac{1}{|X|}\sum_{x,y\in X}\sum_{a,a',b,b'\in X} 
\tau(u_{a,x}^* u_{a',y} u_{b',y}^* u_{b,x}) \epsilon_{a,a'}\otimes \epsilon_{b,b'}\\
& = & 
\frac{1}{|X|}\sum_{x\in X}\sum_{a,a',b,b'\in X} 
\delta_{a',b'}\tau(u_{a,x}^*u_{b,x}) \epsilon_{a,a'}\otimes \epsilon_{b,b'}\\
& = & 
\frac{1}{|X|} \sum_{x\in X}\sum_{a,b,c\in X} \tau(u_{a,x}^*u_{b,x}) \epsilon_{a,c}\otimes \epsilon_{b,c}\\
& = & 
\frac{1}{|X|} \sum_{a,b,c\in X} \tau(w_{a,b}) \epsilon_{a,c}\otimes \epsilon_{b,c}\\
& = & 
\frac{1}{|X|} 
\left(\sum_{a,b\in X} \tau(w_{a,b}) e_a\otimes e_b\right)  \left(\sum_{c\in X} e_c\otimes e_c \right)^*\\
\end{eqnarray*}
Since $\Gamma$ is a quantum channel, $\Gamma(J_X)$ is a positive operator and hence
$$\sum_{a,b\in X} \tau(w_{a,b}) e_a\otimes e_b = \sum_{c\in X} e_c\otimes e_c,$$
implying 
$$\tau(w_{a,b}) = \delta_{a,b}, \ \ \ a,b\in X,$$
and $\Gamma(J_X)=J_A$.

\medskip

(i)-(ii) If $\cl A$ is a unital C*-algebra, equipped with a trace $\tau_{\cl A}$, 
and $\rho: \cl C_{X,A}\to\cl A$ is a *-homomorphism such that 
$\tilde \tau = \tau_{\cl A} \circ \rho$,
then $\tau_{\cl A}(\rho(\cl J)) = 0$. 
Let $\tilde\rho : \cl C_{X,A}/\cl J\to \cl A/\rho(\cl J)$ be given by 
$\tilde\rho (u+\cl J) = \rho(u) + \rho(\cl J)$. 
Then $\tilde\rho$ is a *-homomorphism and the map 
$\tau_{\cl A/\rho(J)} : \cl A/\rho(\cl J) \to \bb{C}$, given by 
$\tau_{\cl A/\rho(J)}(a+\rho(\cl J)):=\tau_{\cl A}(a)$, is a well-defined trace on 
$\cl A/\rho(\cl J)$. 
We have  $\tau_{\cl A/\rho(\cl J)}(\tilde\rho(\pi(u))) = \tau_{\cl A}(\rho(u))$, $u\in\cl C_{X,A}$. 
Clearly, if $\cl A$ is finite-dimensional (resp. abelian), so is $\cl A/\rho(\cl J)$.  
The statements now follow after an inspection of the proof of Theorem \ref{th_qtnoteq1}. 
\end{proof}

We do not know if the approximately quantum perfect strategies 
for concurrent games admit a characterisation via amenable traces of $\cl C_X$ under the 
conditions of Theorem \ref{th_qt}.


\subsection{Algebras of quantum games}\label{ss_aq}

Similarly to concurrent classical-to-quantum games, concurrent quantum games give rise
to *- and C*-algebras which we now describe. 
For $P\in \cl P_{XX}$ and $Q\in \cl P_{AA}$, define a linear map 
$$\gamma_{P,Q} : M_{XX}\otimes M_{AA}\otimes \frak{C}_{X,A}\otimes \frak{C}_{X,A}^{\rm op}\to \frak{C}_{X,A}$$
by letting 
$$\gamma_{P,Q}\left(\omega\otimes u \otimes v^{\rm op}\right) = \Tr(\omega (P\otimes Q)) uv, \ \ \omega\in M_{XX}\otimes M_{AA}, 
u,v\in \frak{C}_{X,A}.$$
For a quantum game $\nph : \cl P_{XX}\to \cl P_{AA}$, let 
$$\frak{I}(\nph) = \left\langle \gamma_{P,\nph(P)^{\perp}}\left(E\otimes E^{\rm op}\right) : P\in \cl P_{XX}\right\rangle$$ 
be the *-ideal in $\frak{C}_{X,A}$ generated by $\gamma_{P,\nph(Q)^{\perp}}(E\otimes E^{\rm op})$, 
$P\in \cl P_{XX}$,
and $\cl I(\nph)$ be the closed ideal in $\cl C_{X,A}$ generated by the same set. 
Write $\frak{C}(\nph) = \frak{C}_{X,A}/\frak{I}(\nph)$ (resp. $\cl C(\nph) = \cl C_{X,A}/\cl I(\nph)$) for the 
quotient *-algebra (resp. quotient C*-algebra). 
Similarly, we define an ideal $I(\varphi)$ in $\cl C_X$ and its quotient, where we write $E$ for $(p_{x,x',a,a'})_{x,x',a,a'}\in M_{XX}(\cl C_X)$.

Similarly to Corollary \ref{c_trcstog}, we obtain the following:

\begin{corollary}\label{c_trcqua}
Let $X$ be a finite set and $\nph : \cl P_{XX} \to \cl P_{XX}$ be a concurrent quantum game.
The following are equivalent for a QNS correlation $\Gamma : M_{XX}\to M_{XX}$:
\begin{itemize}
\item[(i)] $\Gamma$ is a perfect quantum commuting (resp. quantum/local) strategy for $\nph$;
\item[(ii)] there exists a trace $\tau$ 
(resp. a trace $\tau$ that factors through a finite dimensional/abelian *-representation)
of $\cl C_X$ such that  
$$\Gamma(\epsilon_{x,x'} \otimes \epsilon_{y,y'}) = \left(\tau(e_{x,x',a,a'}e_{y,y',b',b})\right)_{a,a',b,b'}, \ \ \ x,x',y,y'\in X,$$
and
$$\tau(\gamma_{P,\varphi(P)^\perp}(E\otimes E^{\rm op}))=0.$$
\end{itemize}
\end{corollary}


\section{The quantum graph homomorphism game}\label{s_qghg}

In this section, we revisit the quantum graph homomorphism game as introduced in 
\cite{tt-QNS}, and provide characterisations of its perfect QNS strategies of various classes.

\subsection{Characterisations of the existence of perfect strategies}

Let $Z$ be a finite set, $H = \bb{C}^Z$, and recall that 
$H^{\rm d}$ stands for the dual (Banach) space of $H$. 
Let $\theta : H\otimes H\to \cl L(H^{\dd},H)$ be the linear map given by 
$$\theta(\xi\otimes\eta)(\zeta^{\dd}) = \langle \xi,\zeta\rangle\eta, \ \ \ \zeta\in H.$$
We have
\begin{equation}\label{eq_MNtheta}
\theta((S\otimes T)\zeta) = T\theta(\zeta)S^{\rm d}, \ \ \ \zeta\in H\otimes H, \ S,T\in \cl L(H).
\end{equation}
We denote by $\mm : H\otimes H\to \bb{C}$ the map, given by 
$$\mm(\zeta) = \left\langle \zeta, \sum_{z\in Z} e_z\otimes e_z\right\rangle, \ \ \ \zeta\in H\otimes H.$$
Let also 
$\frak{f} : H\otimes H \to H\otimes H$ be the flip
operator, given by $\frak{f}(\xi\otimes\eta) = \eta\otimes\xi$.

\begin{definition}\label{d_ss}
A linear subspace $\cl U\subseteq H\otimes H$ is called \emph{skew} if $\mm(\cl U) = \{0\}$ and 
\emph{symmetric} if $\frak{f}(\cl U) = \cl U$. 
\end{definition}

If $\cl U$ is a symmetric skew subspace of $H\otimes H$ and 
$\cl S_{\cl U} = \theta(\cl U)$ then 
the subspace $\cl S_{\cl U}$ of $\cl L(H^{\dd},H)$ has the following properties:
\smallskip
\begin{itemize}
\item $T\in \cl S_{\cl U}$ $\Longrightarrow$ $\dd^{-1}\circ T^* \circ \dd^{-1} \in \cl S_{\cl U}$, and 
\item $T\in \cl S_{\cl U}$ $\Longrightarrow$  $\sum_{z\in Z} \langle (T\circ \dd) (e_z),e_z\rangle = 0$.
\end{itemize}
\smallskip
We call a subspace of $\cl L(H^{\dd},H)$ satisfying these properties a \emph{twisted operator anti-system}, because of its resemblance to \emph{operator anti-systems} (that is, selfadjoint subspaces of $M_X$ each of whose elements has trace zero \cite{btw}).
Given a twisted operator anti-system $\cl S\subseteq \cl L(H^{\dd},H)$, one has that  
the subspace $\cl U_{\cl S} = \theta^{-1}(\cl S)$ of $H\otimes H$ is symmetric and skew. 

Given a graph $G$, let 
$$\cl U_G = {\rm span}\{e_x\otimes e_y : x\sim y\};$$
then $\cl U_G$ is a symmetric skew subspace of $\bb{C}^X\otimes \bb{C}^X$. 
We thus consider symmetric skew subspaces of $\bb{C}^X\otimes \bb{C}^X$ as a 
non-commutative version of graphs. 
We note that a couple of other non-commutative incarnations of graphs were considered in the literature, namely, 
operator subsystems in $M_X$ in \cite{dsw} -- 
after noting that the subspace 
$$\cl S_G := {\rm span}\{\epsilon_{x,x'} : x\simeq x'\}$$
of $M_X$ is an operator system, 
and operator anti-systems in \cite{stahlke} -- after noting that 
the subspace 
$$\cl S_G^0 := {\rm span}\{\epsilon_{x,x'} : x\sim x'\}$$ 
of $M_X$ is an operator anti-system.
Our use of symmetric skew subspaces, instead of some of these concepts,
is dictated by the nature of the definition of QNS correlations, adopted in \cite{dw}.

We write $P_{\cl U}$ for the orthogonal projection from $\bb{C}^X\otimes \bb{C}^X$ onto $\cl U$. 
Let $\cl U_{\perp}\subseteq \left(\bb{C}^X\otimes \bb{C}^X\right)^{\dd}$ be the annihilator of $\cl U$ and write
$P_{\cl U_\perp}\in \cl L\left((\bb{C}^X\otimes \bb{C}^X)^{\dd}\right) $ for the orthogonal projection onto $\cl U_{\perp}$. 
Observe that $\zeta^{\dd}\in \cl U_{\perp}$ if and only if $\zeta$ 
belongs to the orthogonal complement $\cl U^{\perp}$ of $\cl U$ in $\bb{C}^X\otimes \bb{C}^X$.
In addition, 
$$P_{\cl U_\perp} = (P_{\cl U}^\perp)^{\dd}.$$


Let $\cl U\subseteq \bb{C}^X$ and $\cl V\subseteq \bb{C}^A$ be symmetric skew spaces. 
The \emph{quantum graph homomorphism game} $\cl U\to \cl V$ is the quantum non-local game 
$\nph_{\cl U\to\cl V} : \cl P_{XX}\to \cl P_{AA}$ determined by
$$
\nph_{\cl U\to\cl V}(P) = 
\begin{cases} 
0 & \mbox{if } P = 0\\
P_{\cl V} & \mbox{if } P \leq \cl P_{\cl U}\\
I_{AA} &\mbox{otherwise } 
\end{cases}$$

\begin{definition}\label{d_hss}
Let $X$ and $A$ be finite sets and $\cl U\subseteq \bb{C}^X\otimes \bb{C}^X$,  
$\cl V\subseteq \bb{C}^A\otimes \bb{C}^A$
be symmetric skew subspaces. 
We say that $\cl U$ is \emph{quantum commuting homomorphic} 
(resp. \emph{quantum homomorphic}, \emph{locally homomorphic}) 
to $\cl V$, and write 
$\cl U\stackrel{\rm qc}{\to}\cl V$ (resp. $\cl U\stackrel{\rm q}{\to}\cl V$, $\cl U\stackrel{\rm loc}{\to}\cl V$),
if $\nph_{\cl U\to \cl V}$ has a perfect quantum commuting (resp. quantum, local) tracial strategy. 
\end{definition}

Given operator anti-systems $\cl S\subseteq M_X$ and $\cl T\subseteq M_A$,
Stahlke \cite{stahlke} defines a 
\emph{non-commutative graph homomorphism} from $\cl S$ to $\cl T$ to be a 
quantum channel $\Phi : M_X\to M_A$ whose family $\{M_i\}_{i=1}^m$ of Kraus operators satisfies the conditions
$$M_i \cl S M_j^*\subseteq \cl T, \ \ i,j = 1,\dots,m;$$
if such $\Phi$ exists, one writes $\cl S\to\cl T$. 
We recall the suitable version of this notion for twisted operator anti-systems, described in \cite{tt-QNS}.

\begin{definition}\label{d_ssss}
Let $X$ and $A$ be finite sets,
and $\cl S\subseteq \cl L\left((\bb{C}^{X})^{\dd},\bb{C}^X\right)$ and $\cl T\subseteq \cl L\left((\bb{C}^{A})^{\dd},\bb{C}^A\right)$ be 
twisted operator anti-systems. 
A \emph{homomorphism} from $\cl S$ into $\cl T$ is a quantum channel 
$$\Phi : M_X\to M_A, \ \ \Phi(T) = \sum_{i=1}^m M_i T M_i^*,$$
such that 
$$\overline{M}_j \cl S M_i^{\rm d}\subseteq \cl T, \ \ \ i,j = 1,\dots,m.$$
\end{definition}

If $\cl S$ and $\cl T$ are twisted operator anti-systems, we write $\cl S\to\cl T$ as in \cite{stahlke} 
to denote the existence of a homomorphism from $\cl S$ to $\cl T$.  
Further, if $G$ and $H$ are graphs, we write $G\to H$ if there exists a homomorphism from $G$ to $H$. 
The following was shown in \cite{tt-QNS}.

\begin{proposition}\label{p_losame}
Let $X$ and $A$ be finite sets, $\cl U\subseteq \bb{C}^X\otimes \bb{C}^X$,
$\cl V\subseteq \bb{C}^A\otimes \bb{C}^A$ be symmetric skew spaces, and $G$, $H$ be graphs.
The following hold: 
\begin{itemize}
\item[(i)] $\cl U\stackrel{\rm loc}{\to}\cl V$ if and only if $\cl S_{\cl U}\to \cl S_{\cl V}$;

\item[(ii)] $G\to H$ if and only if $\cl U_G \stackrel{{\rm loc}}{\to} \cl U_H$. 
\end{itemize}
\end{proposition}

Let $U_X:(\bb{C}^X)^{\dd}\to \bb{C}^X$ be  the unitary operator given on the 
standard basis by $U_Xe_x^{\dd} = e_x$, $x\in X$, and define $U_A:(\bb{C}^A)^{\dd}\to \bb{C}^A$ similarly. Then
$\cl S\subseteq \cl L\left((\bb{C}^{X})^{\dd},\bb{C}^X\right)$ is a twisted operator anti-system if and only if 
the space $\cl S U_X^{-1}$ of $M_X$  has the following properties:
\begin{itemize}
\item $T \in \cl S U_X^{-1} \Rightarrow T^{\rm t}\in \cl S U_X^{-1}$;
\item $T \in \cl S U_X^{-1} \Rightarrow \Tr (T) = 0$.
\end{itemize}
Indeed, the first property is a direct consequence of the fact that
\begin{eqnarray*}
{\dd}^{-1}\circ (TU_X)^*\circ{\dd}^{-1}U_X^{-1}e_x
& = & 
{\dd}^{-1}\circ U_X^*(T^*e_x)=\sum_{y\in X}\overline{\langle T^*e_x,e_y\rangle}e_y\\
& =& \sum_{y\in X} t_{x,y}e_y = T^{\rm t}e_x,
\end{eqnarray*}
while the second one follows directly from the definition of a twisted operator anti-system.

Recall from Section \ref{s_prel} that $\frak{m}_Z$ denotes the (normalised) maximally entangled vector in $\bb{C}^Z\otimes\bb{C}^Z$.
For a symmetric skew space $\cl U\subseteq \bb{C}^X$, 
set
$$\cl U\otimes\frak{m}_Z = \left\{\xi\otimes\frak{m}_Z: \xi \in\cl U\right\};$$
after applying the shuffle map, we view 
$\cl U\otimes\frak{m}_Z$ as a symmetric skew subspace of 
$\left(\bb{C}^X\otimes\bb{C}^Z\right)\otimes \left(\bb{C}^X\otimes\bb{C}^Z\right)$.

\begin{theorem}\label{th_entgr}
Let $X$ and $A$ be  finite sets and $\cl U\subseteq \bb{C}^X\otimes \bb{C}^X$, 
$\cl V\subseteq \bb{C}^A\otimes \bb{C}^A$ be symmetric skew spaces. The following are 
equivalent:
\begin{itemize}
\item[(i)] 
$\cl U\stackrel{\rm q}\to\cl V$;

\item[(ii)] $\cl U\otimes\frak{m}_Z \stackrel{{\rm loc}}{\to} \cl V$ for some finite set $Z$.
\end{itemize}
\end{theorem}

\begin{proof}
(i)$\Rightarrow$(ii) 
Let $\Gamma : M_{XX}\to M_{AA}$ be a 
tracial quantum QNS correlation such that 
$$\langle\Gamma(P_{\cl U}), P_{{\cl V}_\perp}\rangle=0,$$
that is, such that 
$$\langle\Gamma(\xi\xi^*)\eta,\eta\rangle= \langle\Gamma(\xi\xi^*), (\eta\eta^*)^{\dd}\rangle = 0, 
\ \ \ \xi\in\cl U, \eta\in \cl V^\perp.$$
By definition of  tracial quantum QNS correlation, 
there exists a finite dimensional $C^*$-algebra $\cl A$, a tracial state $\tau_{\cl A}$ on $\cl A$  and a 
*-homomorphism $\pi: \cl C_{X,A}\to \cl A$ 
such that 
$$\Gamma (\epsilon_{x,x'}\otimes \epsilon_{y,y'}) = (\tau_{\cl A}(\pi(e_{x,x',a,a'}e_{y',y,b',b})))_{a,a',b,b'}.$$ 
Writing $\xi = \sum_{x,y\in X} \alpha_{x,y}e_x\otimes e_y$  and 
$\eta = \sum_{a,b\in A}\eta_{a,b}e_a\otimes e_b$, we have
$$\Gamma(\xi\xi^*) = 
\sum_{a,a',b,b'}\sum_{x,x',y,y'} \tau_{\cl A}(\pi(e_{x,x',a,a'}e_{y',y,b',b}))\alpha_{x,y}\overline{\alpha_{x',y'}}
\epsilon_{a,a'}\otimes \epsilon_{b,b'}.$$
Let 
$$Y_{\bar\xi} := \sum_{x',y' \in X}\overline{\alpha_{x',y'}} \epsilon_{x',y'}, \  
Y_\eta = \sum_{a', b' \in A} \eta_{a',b'} \epsilon_{a',b'}$$
and $E = (\pi(e_{x,x',a,a'}))_{x,x',a,a'}$; then 
$E$ is a stochastic $\cl A$-matrix. 
Observe that 
\begin{equation}\label{eq_UX}
\theta(\xi)U_X^{-1} 
= \sum_{x,y\in X} \alpha_{x,y}\theta(e_x\otimes e_y)U_X^{-1}
= \sum_{x,y \in X}\alpha_{x,y} \epsilon_{y,x} =Y_{\bar\xi}^*.
\end{equation}
We have 
\begin{eqnarray*}
0 
& = & 
\langle\Gamma(\xi\xi^*)\eta,\eta\rangle\\
& = & 
\sum_{a,a',b,b'}\sum_{x,x',y,y'} 
\tau_{\cl A}(\pi(e_{x,x',a,a})\pi(e_{y',y,b,b}))\alpha_{x,y}\overline{\alpha_{x',y'}}\eta_{a',b'}\overline{\eta_{a,b}}\\
&= &
\left(\Tr\otimes\tau_{\cl A}\right) \left(\left(\sum_{x',y',a',b'}\pi(e_{x,x',a,a'})\overline{\alpha_{x',y'}}\eta_{a',b'}\pi(e_{y',y,b',b})\right)_{x,y,a,b}\right.\\
&&\left.\times\left(\sum_{x,y,a,b}\overline{\alpha_{x,y}}
\eta_{a,b} \epsilon_{x,y}\otimes \epsilon_{a,b}\otimes 1_{\cl A}\right)^*\right)\\
& = & 
\left(\Tr\otimes\tau_{\cl A}\right) 
\left(E(Y_{\bar\xi}\otimes Y_{\eta}\otimes 1_{\cl A})E(Y_{\bar\xi}^*\otimes Y_\eta^*\otimes1_{\cl A})\right).
\end{eqnarray*}

After passing to a quotient, we may assume
that $\cl A$ is faithfully represented on a Hilbert space $H$ and $\tau_{\cl A}$ is faithful. As $E$ is positive, we have 
$$E^{1/2}(Y_{\bar \xi}\otimes Y_\eta\otimes 1_{\cl A})E(Y_{\bar\xi}^*\otimes Y_\eta^*\otimes 1_{\cl A})E^{1/2}=0.$$ 
It follows that 
$E^{1/2}(Y_{\bar\xi}\otimes Y_\eta\otimes 1_{\cl A})E^{1/2}=0$ and 
hence $E(Y_{\bar\xi}^*\otimes Y_\eta^*\otimes1_{\cl A})E=0$.  

Define a linear map $\psi : M_A\to M_X\otimes\cl A$ by letting 
$$\psi\left(\epsilon_{a,b}\right) = E_{a,b} := (\pi(e_{x,x',a,b}))_{x,x'};$$ 
by Choi's Theorem, $\psi$ is a unital completely positive map.
Let $\psi(\omega) = \sum_{i=1}^m M_i\omega M_i^*$ be a Kraus representation
(here $M_i : \bb{C}^A\to \bb{C}^X\otimes H$, $i = 1,\dots, m$), and set
$$X_{a,b,i,j} = \sum_{a',b'\in A}\overline{\eta_{b',a'}} 
\epsilon_{a,a'} M_j^*(Y_{\bar\xi}^*\otimes 1_{\cl A})M_i \epsilon_{b',b}, \ \ \ a,b\in A, i,j\in [m].$$  
Let
$\sigma^{1,2}:\mathbb C^A\otimes \mathbb C^X\otimes H\to \mathbb C^X\otimes\mathbb C^A\otimes H$ 
be the flip operator defined on the elementary tensors by 
$\sigma^{1,2}(\xi_1\otimes\xi_2\otimes \xi_3)=\xi_2\otimes\xi_1\otimes\xi_3$, 
and write $M_i^{1,3}:\mathbb C^A\otimes\mathbb C^A\to \mathbb C^X\otimes \mathbb C^A\otimes H$ 
for the operator $\sigma^{1,2}(1\otimes M_i)$.

We have
\begin{eqnarray*}
&&
\hspace{-0.7cm}   \Tr(X_{a,b,i,j}X_{a,b,i,j}^*)\\
& = &
\hspace{-0.75cm}  \sum_{a',b',a'',b''}\overline{\eta_{b',a'}}\eta_{b'',a''}
\Tr(\epsilon_{a,a'} M_j^*(Y_{\bar\xi}^*\otimes 1_{\cl A})M_i \epsilon_{b',b} \epsilon_{b,b''}M_i^*
(Y_{\bar\xi}\otimes 1_{\cl A})M_j \epsilon_{a'',a}).
\end{eqnarray*}
Letting $\epsilon_a = (\epsilon_{a,a'})_{a'\in A}$, considered as a row operator over 
$M_A$, we have
\begin{eqnarray*}
& &
\sum_{i,j=1}^m\Tr(X_{a,b,i,j}X_{a,b,i,j}^*)\\
& = & 
\hspace{-0.25cm}
\sum_{j,a',b',a'',b''}\overline{\eta_{b',a'}}\eta_{b'',a''}
\Tr\left(\epsilon_{a,a'} M_j^*(Y_{\bar\xi}^*\otimes 1_{\cl A})
E_{b',b''}(Y_{\bar\xi} \otimes 1_{\cl A})M_j \epsilon_{a'',a}\right)\\
& = &
\sum_{j=1}^m\Tr\left(\epsilon_a(M_j^{1,3})^*(Y_{\bar\xi}^*\otimes Y_\eta^*\otimes 1_{\cl A})E( Y_{\bar\xi}\otimes Y_\eta\otimes 1_{\cl A}) M_j^{1,3}\epsilon_a^*\right).
\end{eqnarray*}

Write $R_{a,j}=E^{1/2}( Y_{\bar\xi}\otimes Y_\eta\otimes 1_{\cl A})M_j^{1,3}\epsilon_a^*$. 
Then
\begin{eqnarray*}
&&\sum_{j=1}^m
(\Tr\otimes\tau_{\cl A})(R_{a,j}R_{a,j}^*)\\
&=&
\hspace{-0.25cm}
\sum_{j=1}^m \Tr\left(E^{1/2}(Y_{\bar\xi}\otimes Y_\eta\otimes 1_{\cl A}) M_j^{1,3}\epsilon_a^*\epsilon_a (M_j^{1,3})^* (Y_{\bar\xi}^*\otimes Y_\eta^*\otimes 1_{\cl A})E^{1/2}\right)\\
&=&
\hspace{-0.25cm}
(\Tr\otimes\tau_{\cl A})
\left(E^{1/2}(Y_{\bar\xi}\otimes Y_\eta\otimes 1_{\cl A})E(Y_{\bar\xi}^*\otimes Y_\eta^*\otimes 1_{\cl A})E^{1/2}\right)
= 0,
\end{eqnarray*}
giving $R_{a,j}=0$, as we assume that the trace is faithful and therefore $\sum_{i,j=1}^m\Tr(X_{a,b,i,j}X_{a,b,i,j}^*)=0$  implying 
\begin{equation}\label{eq_Xabij}
X_{a,b,i,j}=0, \ \ \ a,b\in A, i,j = 1,\dots,m.
\end{equation}
We may assume that  $\cl A$ is faithfully represented on $\bb{C}^Z$ and so that 
 $M_i$ is an operator from $\bb{C}^A$ into $\bb{C}^X\otimes\bb{C}^Z$.  
Let $R_j=\overline{M_j^*}$, $j = 1,\dots,m$. 
For $a\in A$ we have 
\begin{eqnarray*}
(U_X\otimes U_Z)R_i^{\dd}U_A^{-1}(e_a) 
& = & (U_X\otimes U_Z)R_i^{\dd}(e_a^{\rm d})
= (U_X\otimes U_Z)(R_i^*e_a)^{\rm d}\\
& = & (U_X\otimes U_Z)(\overline{M}_i e_a)^{\rm d} = M_i e_a.
\end{eqnarray*}
Taking into account (\ref{eq_UX}), we obtain
 \begin{eqnarray}\label{eq_RjRi}
\overline{R_j}\theta(\xi\otimes\frak{m}_Z)R_i^{\dd} 
& = & 
\overline{R_j}(\theta(\xi)U_X^{-1}\otimes\theta(\frak{m}_Z) U_Z^{-1})((U_X\otimes U_Z)R_i^{\dd}U_A^{-1}U_A\nonumber\\
&= & 
M_j^*(Y_{\bar\xi}^*\otimes 1)((U_X\otimes U_Z)R_i^{\dd}U_A^{-1}U_A\nonumber\\
& = & 
M_j^*(Y_{\bar\xi}^*\otimes 1)M_iU_A.
\end{eqnarray}
Since $\psi$ is unital, 
$$\sum_{j=1}^m R_j^*R_j=\sum_{j=1}^m \overline{M_jM_j^*} = I$$ 
and hence the map $\omega\mapsto \sum_{j=1}^m R_j\omega R_j^*$ from $M_{XZ}$ into $M_A$ is a quantum channel.
We claim that 
\begin{equation}\label{eq_Rj}
\overline{R_j}\theta(\xi\otimes\frak{m}_Z)R_i^{\dd}\subseteq \cl S_{\cl V}.
\end{equation}
Indeed, fix $\eta\in \cl V^\perp$. Since $\theta(\eta)U_A^{-1}=Y_{\bar\eta}^*$, taking 
(\ref{eq_Xabij}) and (\ref{eq_RjRi}) into account, we have 
\begin{eqnarray*}
\left\langle \overline{R_j}\theta(\xi\otimes\frak{m}_Z)R_i^{\dd}, \theta(\eta)\right\rangle
& = & 
\Tr\left(M_j^*(Y_{\bar\xi}^*\otimes 1)\right)M_iU_A\theta(\eta)^*)\\
& = & 
\Tr\left(M_j^*(Y_{\bar\xi}^*\otimes 1)M_iY_{\bar\eta}\right)\\
& = & 
\sum_{a',b' \in A}\overline{\eta_{b',a'}} \left\langle M_j^*(Y_{\bar\xi}^*\otimes 1)M_ie_{b'},e_{a'}\right\rangle\\
& = & 
\left\langle X_{a,b,i,j}e_b,e_a\right\rangle = 0;
\end{eqnarray*}
(\ref{eq_Rj}) now follows.

(ii)$\Rightarrow$(i) 
By Proposition \ref{p_losame}, 
\begin{equation}\label{eq_Sloc}
\cl S_{\cl U\otimes\frak{m}_Z}\to \cl S_{\cl V}. 
\end{equation}
Let $(R_i)_{i=1}^m \subseteq \cl L(\bb{C}^X\otimes\bb{C}^Z,\bb{C}^A)$ be a family of 
Kraus operators of the quantum channel
implementing (\ref{eq_Sloc}). 
Keeping the notation from the previous paragraphs and 
reversing the arguments therein, we see that, if $M_i=\overline{R_i^*}$, 
$\xi\in \cl U$ and $\eta\in \cl V^\perp$, then 
\begin{equation}\label{eq1}
\sum_{a',b' \in A} \overline{\eta_{b',a'}}
\left\langle M_j^*(Y_{\bar\xi}^*\otimes 1_Z)M_ie_{b'},e_{a'} \right\rangle = 0. 
\end{equation}  
Thus, $X_{a,b,i,j}=0$ for all $a,b \in A$ and all $i,j = 1,\dots,m$. 

Letting $\psi:M_A\to M_{X}\otimes M_Z$ be the unital completely positive map given by  
$\omega\mapsto\sum_iM_i\omega M_i^*$ and setting 
$E_{a,b} = \psi\left(\epsilon_{a,b}\right)$, we see that 
$E = (E_{a,b})_{a,b}$ is a stochastic operator matrix acting on $\bb{C}^Z$.
By \cite[Theorem 5.2]{tt-QNS},  
there exists a *-representation $\pi: \cl C_{X,A}\to \cl B(\bb{C}^Z)$ such that $(\pi(e_{x,x',a,a'}))_{x,x',a,a'}=E$. 
Let $\Gamma: M_{XX}\to M_{AA}$ be the linear map given by 
$$\Gamma\left(\epsilon_{x,x'}\otimes \epsilon_{y,y'}\right) = \left(\Tr(\pi( e_{x,x',a,a'}e_{y',y,b',b}))\right)_{a,a',b,b'};$$ 
thus, $\Gamma$ is a tracial quantum QNS correlation and, 
by (\ref{eq1}) and the previous paragraphs, 
$$\langle\Gamma(\xi\xi^*), (\eta\eta^*)^{\dd}\rangle=\Tr(E(Y_{\bar\xi}\otimes Y_\eta\otimes 1_{\cl A})E(Y_{\bar\xi}^*\otimes Y_\eta^*\otimes 1_{ Z}))=0.$$
It follows that 
$\langle \Gamma(\xi\xi^*), P_{{\cl V}_\perp}\rangle =0$ for every $\xi\in \cl U$, 
giving $\langle \Gamma(P_{\cl U}), P_{{\cl V}_\perp}\rangle =0$.
\end{proof}

\begin{remark}\label{r_extraction}
\rm 
It was shown as part of the proof of Theorem \ref{th_entgr} that, for symmetric skew spaces $\cl U\subseteq \bb{C}^X\otimes\bb{C}^X$ and $\cl V \subseteq \bb{C}^A\otimes\bb{C}^A$, we have that $\cl U\stackrel{\rm q}\to\cl V$ if and only if there exist a finite-dimensional algebra $\cl A$, a unital completely positive map
$\psi : M_A\to M_X\otimes\cl A$ with Kraus representation 
$\psi(T) = \sum_{i=1}^m M_i T M_i^*$, such that 
\begin{equation}\label{eq_Mjtheta}
M_j^*(\theta(\cl U)U_X^{-1}\otimes 1_{\cl A})M_i\subseteq \theta(\cl V)U_A^{-1}, \ \ i,j = 1,\ldots, m.
\end{equation}
The same arguments allow us to conclude the equivalence (i)$\Leftrightarrow$(ii) in the following statement. 
\end{remark}

\begin{theorem}\label{th_qcsss}
Let $X$ and $A$ be finite sets and $\cl U\subseteq \mathbb C^X\otimes\mathbb C^X$, $\cl V\subseteq \mathbb C^A\otimes\mathbb C^A$ be symmetric skew spaces. The following are equivalent:

\begin{itemize}
\item[(i)] 
 $\cl U\stackrel{\rm qc}\to\cl V$;
 
\item[(ii)] 
 there exists a unital completely positive map  
 $\psi: M_A\to M_X\otimes\cl C_{X,A}$ with Kraus representation 
 $\psi(T) = \sum_{i=1}^m M_i T M_i^*$, for which inclusions (\ref{eq_Mjtheta}) hold;

\item[(iii)] there exists a von Neumann algebra $\cl N$ with a faithful normal tracial state $\tau$ 
and a 
unital completely positive map 
$\psi: M_A\to M_X\otimes\cl N$ with Kraus representation 
$\psi(T) = \sum_{i=1}^m M_i T M_i^*$, for which inclusions (\ref{eq_Mjtheta}) hold. 
\end{itemize}
\end{theorem}

\begin{proof}
The equivalence (i)$\Leftrightarrow$(ii) was pointed out in Remark \ref{r_extraction}.
The implication (iii)$\Rightarrow$(i)  is similar to that of (ii)$\Rightarrow$(i) of Theorem \ref{th_entgr}. 
For (ii)$\Rightarrow$(iii), we take 
$\cl N = \pi_\tau(\cl C_{X,A})''$, where $\pi_\tau$  is the GNS representation of $\tau$; if $\xi$ is the cyclic vector of $\pi_\tau$ then $\langle(\cdot)\xi,\xi\rangle $ is a faithful normal trace on $\cl N$. 
\end{proof}

Let $\cl S\subseteq M_X$ and $\cl T\subseteq M_A$ be operator anti-systems. 
Stalhke writes \cite{stahlke} $\cl S\stackrel{*}\to \cl T$ if there exists a finite set $B$ and a state 
$\Lambda\in M_B^+$
such that $\cl S\otimes\Lambda\rightarrow \cl T$; in this case he 
says that there exists an \emph{entanglement assisted homomorphism} from $\cl S$ to $\cl T$.

\begin{corollary}
Let $G$, $H$ be graphs. Then $$\cl U_G\stackrel{\rm q}\to\cl U_H \ \Longrightarrow \ \cl S_G^0\stackrel{*}\to \cl S_H^0.$$
\end{corollary}
\begin{proof}
First observe that $\cl S_G^0 = \theta(\cl U_G)U_X^{-1}$. 
The statement now follows from Remark \ref{r_extraction}. 
\end{proof}

In the next corollary, we partially improve \cite[Proposition 10.5]{tt-QNS} by providing a lower bound on the 
relaxed orthogonal rank $\xi_{\rm q}(G)$. 

\begin{corollary}\label{xiq}
If $G$ is a graph then $\xi_{\rm q}(G)\geq \sqrt{\theta(\overline{G})}$.
\end{corollary}

\begin{proof}
We observe first that
$$\xi_{\rm q}(G) = {\rm min}\{|A|: \cl U_G\stackrel{\rm q}\to  \langle{\frak m}_A\rangle^\perp\}.$$
Moreover,  $\theta(\langle{\frak m}_A\rangle^\perp )U_A^{-1}=({\mathbb C }I_A)^\perp$, and hence 
$\cl U_G\stackrel{\rm q}\to  \langle{\frak m}_A\rangle^\perp$ implies 
$\cl S_G^0\stackrel{*}\to (\mathbb CI_A)^\perp$. It follows from \cite[Corollary 20]{stahlke} 
that $\xi_{\rm q}(G)\geq\sqrt{\theta(\overline{G})}$. 
\end{proof}


\subsection{Quantum colourings of graphs}\label{ss_qcrev}

Let $G$ be a (finite) simple graph with vertex set $X$. 
For $x,y\in X$, we write $x\sim y$ when $\{x,y\}$ is an edge of $G$, and $x\simeq y$ when $x\sim y$ or $x = y$. 
The \emph{classical-to-quantum colouring game} 
$\nph_G^A : \cl P^{\rm cl}_{XX}\to \cl P_{AA}$ is determined by the requirements
$$\nph_G^A(\epsilon_{x,x}\otimes \epsilon_{y,y}) = 
\begin{cases} 
J_A &\mbox{if } x = y \\ 
J_A^{\perp} & \mbox{if } x\sim y,\\
I_{AA} &\mbox{otherwise}.
\end{cases}
$$
In this subsection, we apply the previous results to 
give a description of perfect quantum commuting and perfect quantum 
strategies for the classical-to-quantum colouring game in terms of quantum channels whose 
Kraus operators respect certain containment relations. 
These relations define a ``pushforward" of the graph $G$ into  
$M_A$ or, in the terminology of Weaver \cite{weaver}, 
into the quantum graph $(\mathcal{S},\mathcal{M})$ with $\mathcal{S}=\mathcal{M}=M_A$. 
Namely, for a von Neumann algebra $\cl N$, equipped with faithful tracial state 
$\tau$, and a unital completely positive map 
$\Psi : M_A\to {\cl D}_X\otimes\cl N$ with Kraus representation 
$\Psi(T) = \sum_{i=1}^m M_iT M_i^*$, we consider the inclusion relations
\begin{equation}\label{eq_intoI}
M_i^*(\cl D_X\otimes 1_{\cl N})M_j\subseteq {\mathbb C}I_A, , \ \ \ i,j\in [m],
\end{equation}
and 
\begin{equation}\label{eq_intoperp}
M_i^*(\cl S_G^0 \otimes 1_{\cl N})M_j\perp \mathbb C I_A, \ \ \ i,j\in [m].
\end{equation}

\begin{definition}\label{d_quaco}
Let $X$ and $A$ be finite sets and $G$ be graph with vertex set $X$. A pair $(\cl N,\Psi)$, where 
$\cl N$ is a von Neumann algebra and 
$\Psi : M_A\to {\cl D}_X\otimes\cl N$ is a unital completely positive map with Kraus representation 
$\Psi(T) = \sum_{i=1}^m M_iT M_i^*$, is called a \emph{quantum colouring} of $G$ if 
conditions (\ref{eq_intoI}) and (\ref{eq_intoperp}) are satisfied. 
\end{definition}

Let $\cl R^{\frak{u}}$ denote an ultrapower of the hyperfinite II$_1$-factor $\cl R$ 
by a free ultrafilter $\frak{u}$ on $\mathbb N$ and $\tr_{\cl R^{\frak{u}}}$ be its trace.

\begin{theorem}\label{th_nphGA}
Let $G$ be a graph with vertex set $X$. 
\begin{itemize}
\item[(1)] The following are equivalent:
\begin{itemize}
\item[(i)] the classical-to-quantum colouring game $\varphi_G^A$ has a perfect quantum commuting 
strategy; 
    \item[(ii)] there exists a quantum colouring $(\cl N,\Psi)$ of $G$, with $\cl N$ possessing a faithful tracial state.
\end{itemize}

\item [(2)]The following are equivalent:
\begin{itemize}
\item[(i)] $\varphi_G^A$ has a perfect approximately quantum strategy; 
    \item[(ii)] 
    there exist a quantum colouring of the form $(\cl R^{\frak{u}},\Psi)$.
\end{itemize}

\item [(3)]The following are equivalent:
\begin{itemize}
\item[(i)] $\varphi_G^A$ has a perfect quantum strategy; 
    \item[(ii)] 
    there exists a quantum colouring $(\cl N,\Psi)$ of $G$, where $\cl N$ is finite dimensional.
\end{itemize}
\end{itemize}
\end{theorem}

\begin{proof} 
(1) \ 
(i)$\Rightarrow$(ii)
Let $\cl E : \cl D_{XX}\to M_{AA}$ be a CQNS correlation, which is a perfect quantum commuting 
strategy for $\varphi_G^A$. Let
$\tilde\tau $ be a trace on $\cl B_{X,A}$ associated with $\cl E$ via Corollary \ref{c_trcstog}, and 
$\cl N := \pi_{\tilde\tau}(\cl B_{X,A})''$, where $\pi_{\tilde\tau}$ is the GNS representation 
corresponding to $\tilde\tau$. 
If $\xi$ is the cyclic vector of $\pi_{\tilde\tau}$, then 
$\tau(T):=\langle T\xi,\xi\rangle$ is a faithful trace on $\cl N$.  
Let $\Gamma : M_{XX}\to M_{AA}$ be the canonical lift of $\cl E$ to a QNS correlation:
\begin{eqnarray*}
\Gamma(\epsilon_{x,x'}\otimes\epsilon_{y,y'})&=&(\delta_{x,x'}\delta_{y,y'}\tilde\tau(e_{x,a,b}e_{y,b',a'}))_{a,a',b,b'}\\&=&(\delta_{x,x'}\delta_{y,y'}\tau(\pi_{\tilde\tau}(e_{x,a,b}e_{y,b',a'}))_{a,a',b,b'}.
\end{eqnarray*} As $(\delta_{x,x'}\pi_{\tilde\tau}(e_{x,a,a'})_{x,x',a,a'})$ is a stochastic operator matrix, 
there exists a *-representation 
$\pi : \cl C_{X,A}\to \cl N$ such that $\pi(e_{x,x',a,a'}) = \delta_{x,x'}\pi_{\tilde \tau}(e_{x,a,a'})$, 
$x,x'\in X$, $a,a'\in A$. 
Therefore
$\Gamma$ is a tracial QNS correlation with 
$$\left\langle\Gamma(P_{\cl U_G)}), P_{\cl V_\perp}\right\rangle =0, \ \ 
\mbox{ where } \cl V=\langle{\frak m}_A\rangle^\perp.$$
As $\theta(\cl V^\perp)\cl U_A^{-1}=\mathbb C I_A$ and $\theta(\cl U_G)U_X^{-1} = \cl S_G^0$, 
Theorem \ref{th_qcsss} shows that the 
unital completely positive map $\Psi: M_A\to \cl D_X\otimes\cl N$, given by 
$\Psi(\epsilon_{a,a'}) = (\pi(e_{x,x',a,a'}))_{x,x'}$, has a 
Kraus representation $\Psi(T) = \sum_{i=1}^m M_iTM_i^*$ satisfying (\ref{eq_intoperp}).  
As $\langle\Gamma(\epsilon_{x,x}\otimes\epsilon_{x,x}), (\eta\eta^*)^{\rm d}\rangle =0$
whenever $\eta\in \cl V$, similar arguments show that (\ref{eq_intoI}) is satisfied. 

(ii)$\Rightarrow$(i) 
Let $E_{a,b} = \Psi(\epsilon_{a,b})$, $a,b\in A$. 
Then $E := (E_{a,b})_{a,b}$ is a semi-classical stochastic operator matrix; thus, 
there exists a *-representation $\pi:\cl C_{X,A}\to \cl N$ such that 
$(\pi(e_{x,x',a,a'}))_{x,x',a,a'}=E$. 
Let $\Gamma: M_{XX}\to M_{AA}$ be the QNS correlation given by
$$\Gamma(\epsilon_{x,x'}\otimes\epsilon_{y,y'}) = (\tau(\pi(e_{x,x',a,a'}e_{y',y,b',b})))_{a,a',b,b'}, \ \ \ 
x,x',y,y'\in X.$$
As $\pi(e_{x,x',a,a'})=0$ whenever $x\ne x'$, we have that 
$\Gamma = \Gamma\circ \Delta_{XX}$. 
By Theorem \ref{th_qcsss}, $\langle\Gamma(P_{\cl U_G}),P_{\cl V_\perp}\rangle=0$. 
It hence suffices to show that the CQNS correlation $\Gamma|_{\cl D_{XX}}$ is concurrent.

Let $\eta = \sum_{a,b\in A} \eta_{a,b}e_a\otimes e_b$ be orthogonal to ${\frak m}_A$; 
thus, $\sum_{a\in A} \eta_{a,a} = 0$.
Let $Y_\eta = \sum_{a', b' \in A} \eta_{a',b'} \epsilon_{a',b'}$.  
We have to show that $\langle\Gamma(\epsilon_{x,x}\otimes\epsilon_{x,x}),(\eta\eta^*)^{\rm d}\rangle=0$. 
As $M_j^*(\epsilon_{x,x}\otimes 1_{\cl N})M_i\in\mathbb C I$, 
there exists $\lambda_x\in\mathbb C$ such that 
\begin{equation}\label{eq_d}
\epsilon_{a,a'}M_j^*(\epsilon_{x,x}\otimes 1_{\cl N})M_i\epsilon_{b',b}=\delta_{a',b'}\lambda_x \epsilon_{a,b},
\end{equation}
for all $a,a',b,b'\in A$.
As in the proof of Theorem \ref{th_entgr}, let 
$\epsilon_a = (\epsilon_{a,a'})_{a'}$, considered as a row operator over $M_A$, 
and  $M_i^{1,3} : \mathbb C^A\otimes\mathbb C^A\to \mathbb C^X\otimes \mathbb C^A\otimes H$ be 
the operator $\sigma^{1,2}(1\otimes M_i)$, where $\sigma^{1,2}:\mathbb C^A\otimes \mathbb C^X\otimes H\to \mathbb C^X\otimes\mathbb C^A\otimes H$ is the flip operator defined on the elementary tensors by $\sigma^{1,2}(\xi_1\otimes\xi_2\otimes \xi_3)=\xi_2\otimes\xi_1\otimes\xi_3$.
Fix $a,b\in A$. We have
\begin{eqnarray*}
&&
\hspace{-0.8cm}
\sum_{i,j=1}^m
\sum_{a',b',a'',b''\in A}
\hspace{-0.5cm}
\overline{\eta_{b',a'}}\eta_{b'',a''}\epsilon_{a,a'}M_j^*
(\epsilon_{x,x}\hspace{-0.1cm}\otimes\hspace{-0.1cm} 1_{\cl N})M_i\epsilon_{b',b}\epsilon_{b,b''}M_i^*
(\epsilon_{x,x}\hspace{-0.1cm}\otimes \hspace{-0.1cm} 1_{\cl N})M_j\epsilon_{a'',a}\\
&&
\hspace{-0.8cm}
=\sum_{j=1}^m\sum_{a',b', a'',b''\in A}
\hspace{-0.5cm}
\overline{\eta_{b',a'}}\eta_{b'',a''}\epsilon_{a,a'}M_j^*
(\epsilon_{x,x}\hspace{-0.1cm}\otimes\hspace{-0.1cm} 1_{\cl N})E_{b',b''}
(\epsilon_{x,x}\hspace{-0.1cm}\otimes\hspace{-0.1cm} 1_{\cl N})M_j\epsilon_{a'',a}\\
&&
\hspace{-0.8cm}
=\sum_{j=1}^m \epsilon_a(M_j^{1,3})^*
(\epsilon_{x,x}\hspace{-0.1cm}\otimes\hspace{-0.1cm} Y_\eta^* \hspace{-0.1cm}\otimes\hspace{-0.1cm} 1_{\cl N})E
(\epsilon_{x,x}\hspace{-0.1cm}\otimes \hspace{-0.1cm} Y_\eta \hspace{-0.1cm} \otimes \hspace{-0.1cm} 1_{\cl N})M_j^{1,3}\epsilon_a^*.
\end{eqnarray*}

On the other hand, by (\ref{eq_d}),
\begin{eqnarray*}
&&
\hspace{-0.5cm}
\sum_{a',b',a'',b''\in A}
\hspace{-0.5cm}
\overline{\eta_{b',a'}}\eta_{b'',a''}\epsilon_{a,a'}M_j^*(\epsilon_{x,x}\otimes 1_{\cl N})M_i\epsilon_{b',b}\epsilon_{b,b''}M_i^*(\epsilon_{x,x}\otimes 1_{\cl N})M_j\epsilon_{a'',a}\\&&
=\sum_{a',a''\in A}
\hspace{-0.2cm}
\overline{\eta_{a',a'}}\eta_{a'',a''}\epsilon_{a,a'}M_j^*(\epsilon_{x,x}\otimes 1_{\cl N})M_i\epsilon_{a',a''}M_i^*(\epsilon_{x,x}\otimes 1_{\cl N})M_j\epsilon_{a'',a}\\
&&
= \left(\sum_{a'\in A}\overline{\eta_{a',a'}}\epsilon_{a,a'}M_j^*(\epsilon_{x,x}\otimes 1_{\cl N})M_i\epsilon_{a',a}\right)\\
& & 
\hspace{4cm} \times  \left(\sum_{a'\in A}\overline{\eta_{a',a'}}\epsilon_{a,a'}M_j^*(\epsilon_{x,x}\otimes 1_{\cl N})M_i\epsilon_{a',a}\right)^*\\
&&= \left(\sum_{a'\in A}\overline{\eta_{a',a'}}\lambda_x\epsilon_{a,a}\right)\left(\sum_{a'\in A}\overline{\eta_{a',a'}}\lambda_x\epsilon_{a,a}\right)^*=0.
\end{eqnarray*}
Hence 
$$\sum_{j=1}^m \epsilon_a\left(M_j^{1,3}\right)^*
\left(\epsilon_{x,x}\otimes  Y_\eta^*\otimes1_{\cl N}\right)
E\left(\epsilon_{x,x}\otimes Y_\eta\otimes 1_{\cl N}\right)M_j^{1,3}\epsilon_a^*=0;$$
this implies $E^{1/2}(\epsilon_{x,x}\otimes Y_\eta\otimes 1_{\cl N}) M_j^{1,3}\epsilon_a^*=0$ and therefore
\begin{eqnarray*}
& & 
\hspace{-0.7cm}
0 = \\
& &
\hspace{-0.7cm}
(\Tr\otimes\tau)\left(\sum_{j=1}^m 
E^{1/2}(\epsilon_{x,x}\otimes Y_\eta\otimes 1_{\cl N}) M_j^{1,3}\epsilon_a^*\epsilon_a (M_j^{1,3})^*(\epsilon_{x,x}\otimes Y_\eta\otimes 1_{\cl N})E^{1/2}\right)\\
&& 
\hspace{-0.7cm}
= (\Tr\otimes\tau)\left(E^{1/2}\left(\epsilon_{x,x}\otimes Y_\eta\otimes 1_{\cl N}\right)
E\left(\epsilon_{x,x}\otimes Y_\eta^*\otimes 1_{\cl N}\right)E^{1/2}\right)\\
&& 
\hspace{-0.7cm}
= \left\langle\Gamma\left(\epsilon_{x,x}\otimes\epsilon_{x,x}\right)\eta,\eta\right\rangle
\end{eqnarray*}
showing that $\Gamma$ is concurrent.

\medskip

(2) \ 
(ii)$\Rightarrow$(i) The arguments are similar to those in part (1): we first obtain a 
*-representation $\pi: \cl C_{X,A}\to \cl R^{\frak{u}}$ 
by letting $(\pi(e_{x,x',a,a'}))_{x,x'} = \Psi(\epsilon_{a,a'})$, 
and define
$$\Gamma(\epsilon_{x,x'}\otimes\epsilon_{y,y'})=({\tr }_{\cl R^{\frak{u}}}(\pi(e_{x,x',a,a'}e_{y',y,b',b})))_{a,a',b,b'}.$$
We have that $\tau := \tr_{\cl R^{\frak{u}}}\circ\pi$ is an amenable trace on $\cl C_{X,A}$ (\cite[Proposition 6.3.5 (1),(2)]{bo}). Hence the assignment 
$s_\tau(x\otimes y^{\rm op}) := \tau(xy)$ determines a state on $\cl C_{X,A}\otimes_{\min}\cl C_{X,A}^{\rm op}$ and if $\partial: \cl C_{X,A}\to \cl C_{X,A}^{\rm op}$ is the *-isomorphism given by $\partial(e_{x,x',a,a'})=e_{x',x,a',a}$ 
(\cite[Theorem 7.7]{tt-QNS}) then $s := s_\tau\circ(\id\otimes\partial)$ is a state on $\cl C_{X,A}\otimes_{\min}\cl C_{X,A}$ such that $s(e_{x,x',a,a'}\otimes e_{y,y',b,b'})=\tau(e_{x,x',a,a'}e_{y',y,b',b})$. Applying  \cite[Theorem 6.5]{tt-QNS}, we obtain that  $\Gamma$ is an approximately quantum QNS correlation. As $\Gamma=\Gamma\circ\Delta_{XX}$, by \cite[Theorem 7.3]{tt-QNS}, $\Gamma|_{\cl D_{XX}}$ is an approximately quantum CQNS correlation. The above arguments also give that $\Gamma$ is concurrent and satisfies $\langle \Gamma(P_{\cl U_G}),P_{\cl V_\perp}\rangle = 0$. 

(i)$\Rightarrow$(ii) By Corollary \ref{c_cqt}(i'), a perfect approximately quantum 
CQNS strategy $\Gamma$ is determined by an amenable trace $\tau$ on $\cl B_{X,A}$. Hence there exists a 
*-homomor\-phism $\rho:\cl B_{X,A}\to \cl R^{\frak{u}}$ such that $\tau=\tr_{\cl R^{\frak{u}}}\circ\rho$
(see \cite[Proposition 3.2]{kirchberg}). The proof is completed similarly to (1)(i)$\Rightarrow$(1)(ii).   

\medskip


(3) \ This part of the statement is similar to (1) and (2), and uses the representation of 
quantum strategies established in Theorem \ref{th_reststr}. 
\end{proof}

In the next two propositions, we clarify some useful properties of quantum colourings.
The first one is an automatic homomorphism result.

\begin{proposition}\label{kraus-ops-on-D_X}
Let $X$ and $A$ be finite sets, $\mathcal{N}$ be a von Neumann algebra and 
$\Psi : M_A \to \mathcal{D}_X \otimes \mathcal{N}$ be a unital completely positive map with 
Kraus representation $\Psi(T) = \sum_{i=1}^m M_iTM_i^*$.
The following are equivalent:
\begin{itemize}
\item[(i)] $\Psi$ is a *-homomorphism;

\item[(ii)] condition (\ref{eq_intoI}) holds. 
\end{itemize}
\end{proposition}

\begin{proof}
(i)$\Rightarrow$(ii)
For each $a,b \in A$, we write $\Psi(\epsilon_{a,b}) = \sum_{x \in X} \epsilon_{x,x} \otimes r_{x,a,b}$. 
For $a \neq b$ in $A$ and $1 \leq i,j \leq m$, set
\begin{equation}
X_{a,b,i,j} = \epsilon_{a,a} M_i^*(\epsilon_{x,x} \otimes 1_{\mathcal{N}})M_j \epsilon_{b,b} \in M_A. \label{x-abij}
\end{equation} 
We have
\begin{align}
& \sum_{j=1}^m X_{a,b,i,j}X_{a,b,i,j}^*
= \sum_{j=1}^m \epsilon_{a,a}M_i^*(\epsilon_{x,x} \otimes 1_{\mathcal{N}})M_j\epsilon_{b,b}M_j^*(\epsilon_{x,x} \otimes 1_{\mathcal{N}})M_i \epsilon_{a,a} \nonumber\\
& = \epsilon_{a,a} M_i^*(\epsilon_{x,x} \otimes 1_{\mathcal{N}})\left(\sum_{y \in X} \epsilon_{y,y} \otimes r_{y,b,b}\right)(\epsilon_{x,x} \otimes 1_{\mathcal{N}})M_i\epsilon_{a,a} \nonumber\\
& = \epsilon_{a,a} M_i^* (\epsilon_{x,x} \otimes r_{x,b,b}) M_i \epsilon_{a,a} \nonumber\\
& = \left(\epsilon_{a,a} M_i^* \left(\epsilon_{x,x} \otimes r_{x,b,b}^{1/2}\right)\right)
\left(\epsilon_{a,a} M_i^* \left(e_{x,x} \otimes r_{x,b,b}^{1/2}\right)\right)^* \label{sum-x-abij},
\end{align}
where we have used the fact that $e_{x,b,b}$ is positive. 
Let 
\begin{equation}\label{y-abi}
Y_{a,b,i} = \epsilon_{a,a} M_i^*\left(\epsilon_{x,x} \otimes r_{x,b,b}^{1/2}\right), 
\ \ \ x\in X, a,b\in A, i = 1,\dots,m. 
\end{equation}
By (\ref{sum-x-abij}), $\sum_{j=1}^m  X_{a,b,i,j}X_{a,b,i,j}^*=Y_{a,b,i}Y_{a,b,i}^*$. 
Furthermore,
\begin{align}
\sum_{i=1}^m Y_{a,b,i}^*Y_{a,b,i}
& = \sum_{i=1}^m \left(\epsilon_{x,x} \otimes r_{x,b,b}^{1/2}\right) M_i\epsilon_{a,a}M_i^*
\left(\epsilon_{x,x} \otimes r_{x,b,b}^{1/2}\right) \nonumber\\
& = \left(\epsilon_{x,x} \otimes r_{x,b,b}^{1/2}\right)\left(\sum_{y \in X} \epsilon_{y,y} 
\otimes r_{y,a,a} \right)\left(\epsilon_{x,x} \otimes r_{x,b,b}^{1/2}\right) \nonumber\\
& = \epsilon_{x,x} \otimes \left(r_{x,b,b}^{1/2}r_{x,a,a}^{1/2}\right)\left(r_{x,b,b}^{1/2} r_{x,a,a}^{1/2}\right)^*. 
\label{sum-y-abi}
\end{align}
Since $\Psi$ is a homomorphism, $r_{x,b,b}^{1/2} r_{x,a,a}^{1/2} = r_{x,b,b} r_{x,a,a} = 0$ if $a\neq b$. 
By (\ref{sum-y-abi}), it follows that $Y_{a,b,i}=0$ for all $i$ and all $a \neq b$. 
Using (\ref{sum-x-abij}), we have $X_{a,b,i,j} = 0$ for all $i,j$ and $a \neq b$. 
By (\ref{x-abij}), this forces $M_i^*(\epsilon_{x,x} \otimes 1_{\mathcal{N}})M_j \in \mathcal{D}_A$.

Next, we show that $M_i^*(\epsilon_{x,x} \otimes 1_{\mathcal{N}})M_j$ lies in 
$\mathbb{C}I_A$. We set 
\[\lambda_{a,x,i,j} = \text{Tr}\left(\epsilon_{a,a}M_i^*(\epsilon_{x,x} \otimes 1_{\mathcal{N}})M_j\epsilon_{a,a}\right);\]
then $M_i^*(\epsilon_{x,x} \otimes 1_{\mathcal{N}})M_j = \sum_{a \in A} \lambda_{a,x,i,j} \epsilon_{a,a}$. 
To establish (iii), it suffices to show that $\lambda_{a,x,i,j}=\lambda_{b,x,i,j}$ for all $a,b \in A$.

Set 
\begin{equation}
C_{a,b,x,i,j} = \epsilon_{a,a}M_i^*(\epsilon_{x,x} \otimes 1_{\mathcal{N}})M_j\epsilon_{a,a}
- \epsilon_{a,b}M_i^*(\epsilon_{x,x} \otimes 1_{\mathcal{N}})M_j\epsilon_{b,a} \label{c-abxij}
\end{equation}
and observe that
\[ C_{a,b,x,i,j} = (\lambda_{a,x,i,j}-\lambda_{b,x,i,j})\epsilon_{a,a}.
\]
We note that $\epsilon_{a,a}M_i^*(\epsilon_{x,x} \otimes 1_{\mathcal{N}})M_j\epsilon_{b,a}
= \epsilon_{a,b}M_i^*(\epsilon_{x,x} \otimes 1_{\mathcal{N}})M_j\epsilon_{a,a} = 0$, since 
$M_i^*(\epsilon_{x,x} \otimes 1_{\mathcal{N}})M_j \in \mathcal{D}_A$. Therefore, 
\begin{equation} C_{a,b,x,i,j} = (\epsilon_{a,a}-\epsilon_{a,b})M_i^*(\epsilon_{x,x} \otimes 1_{\mathcal{N}})M_j(\epsilon_{a,a}+\epsilon_{b,a}). \label{updated-c-abxij}
\end{equation}
Since 
$$(\epsilon_{a,a}+\epsilon_{b,a})(\epsilon_{a,a}+\epsilon_{a,b})=\epsilon_{a,a}+\epsilon_{a,b}+\epsilon_{b,a}+\epsilon_{b,b},$$ 
by summing over $j$ and setting 
$d_{x,a,b} = r_{x,a,a} + r_{x,a,b} + r_{x,b,a} + r_{x,b,b} \geq 0$, we obtain
\begin{equation} 
\sum_{j=1}^m C_{a,b,x,i,j}C_{a,b,x,i,j}^*
= (\epsilon_{a,a}-\epsilon_{a,b})M_i^*(\epsilon_{x,x} \otimes d_{x,a,b}))M_i(\epsilon_{a,a}-\epsilon_{b,a}). \label{sum-c-abxij} 
\end{equation}
Let $g_{x,a,b} \in \mathcal{N}$ satisfy $g_{x,a,b}g_{x,a,b}^*= d_{x,a,b}$ and define 
\begin{equation}
D_{a,b,x,i} = (\epsilon_{a,a}-\epsilon_{a,b})M_i^*(\epsilon_{x,x} \otimes g_{x,a,b}) \label{d-abxi}.
\end{equation}
By (\ref{sum-c-abxij}), 
$$\sum_{j=1}^m C_{a,b,x,i,j}C_{a,b,x,i,j}^* = D_{a,b,x,i}D_{a,b,x,i}^*.$$ 
Set $f_{x,a,b}=r_{x,a,a}-r_{x,a,b}-r_{x,b,a}+r_{x,b,b}$ and note that $f_{x,a,b}\geq 0$ and
\begin{align}\label{sum-d-abxi}
&\sum_{i=1}^m D_{a,b,x,i}^*D_{a,b,x,i}\\
&=(\epsilon_{x,x} \otimes g_{x,a,b}^*)M_i(\epsilon_{a,a}-\epsilon_{b,a})(\epsilon_{a,a}-\epsilon_{a,b})M_i^*
(\epsilon_{x,x} \otimes g_{x,a,b})\nonumber\\
&=\epsilon_{x,x} \otimes g_{x,a,b}^*(r_{x,a,a}-r_{x,a,b}-r_{x,b,a}+r_{x,b,b})g_{x,a,b}^* \nonumber\\
&=\epsilon_{x,x} \otimes (g_{x,a,b}^* f_{x,a,b} g_{x,a,b}). \nonumber 
\end{align} 
Since $\Psi$ is a *-homomorphism, the element 
$g_{x,a,b} = r_{x,a,a} + r_{x,a,b}$ satisfies the relation
$g_{x,a,b}^*g_{x,a,b}=d_{x,a,b}$. A calculation then shows that $g_{x,a,b}^*f_{x,a,b}g_{x,a,b}=0$. 
By (\ref{sum-d-abxi}), $D_{a,b,x,i}=0$, and by (\ref{sum-c-abxij}), $C_{a,b,x,i,j}=0$. This forces $\lambda_{x,a,i,j}=\lambda_{x,b,i,j}$ for all $a \neq b$. Hence, $M_i^*(\epsilon_{x,x} \otimes 1_{\mathcal{N}})M_j \in \mathbb{C} I_A$.

(ii)$\Rightarrow$(i)
The assumption implies that 
$M_i^*(\epsilon_{x,x} \otimes 1_{\mathcal{N}})M_j \in \mathcal{D}_A$, so equations (\ref{x-abij})--(\ref{sum-y-abi}) show that $r_{x,a,a}^{1/2}r_{x,b,b}^{1/2} = 0$, and hence $r_{x,a,a}r_{x,b,b}=0$, whenever $a \neq b$. 
Since each $e_{r,a,a} \geq 0$ and $\sum_{a \in A} r_{x,a,a}=1$, we have that 
$r_{x,a,a}^2 = r_{x,a,a}$ for each $x \in X$ and $a \in A$. 
As $\Psi(\epsilon_{a,a}) = \sum_{x \in X} \epsilon_{x,x} \otimes r_{x,a,a}$, the diagonal matrix unit $\epsilon_{a,a}$ belongs to the multiplicative domain of $\Psi$ for each $a$. In particular, 
$$\Psi(\epsilon_{a,a})\Psi(\epsilon_{b,c})=\delta_{a,b} \Psi(\epsilon_{a,c})=\delta_{a,b} \sum_{x \in X} \epsilon_{x,x} \otimes r_{x,a,c}$$ for all $b,c \in A$. 

Now, choose $g_{x,a,b}$ with $g_{x,a,b}^*g_{a,x,b}=d_{x,a,b}$, and $h_{x,a,b}$ with $h_{x,a,b}^*h_{x,a,b}=f_{x,a,b}$. 
Our assumption on implies that $C_{a,b,x,i,j}=0$ for all $a \neq b$, $x \in X$ and all $i$ and $j$. 
By (\ref{c-abxij})--(\ref{sum-d-abxi}), $g_{x,a,b}^*f_{x,a,b}g_{x,a,b}=0$, yielding $g_{x,a,b}h_{x,a,b}^*=0$. Multiplying on the left by $g_{x,a,b}^*$ and on the right by $h_{x,a,b}$, we get $d_{x,a,b}f_{x,a,b}=0$. Using the fact that $\epsilon_{a,a}$ and $\epsilon_{b,b}$ are in the multiplicative domain of $\Psi$, a calculation shows that
\begin{equation} 
0=d_{x,a,b}f_{x,a,b}=r_{x,a,a}+r_{x,b,b}-r_{x,a,b}r_{x,b,a}-r_{x,b,a}r_{x,b,a}. \label{d-xab-f-xab}
\end{equation}
Multiplying equation (\ref{d-xab-f-xab}) on both sides by $r_{x,a,a}$, we get $0=r_{x,a,a}-r_{x,a,b}r_{x,b,a}$. Therefore, $r_{x,a,b}r_{x,b,a}=r_{x,a,a}$. Similarly, $r_{x,b,a}r_{x,a,b}=r_{x,b,b}$, so that $\epsilon_{a,b}$ belongs to the multiplicative domain of $\Psi$. Since $a,b \in A$ were arbitrary with $a \neq b$, $\Psi$ must be a homomorphism, completing the proof.
\end{proof}

\noindent {\bf Remark. } We note that an alternative proof of the implications (iii)$\Rightarrow$(ii) in Theorem \ref{th_nphGA}
can be given, using Proposition \ref{kraus-ops-on-D_X}. We have decided to present the given argument instead
as it shows that, it order to conclude that the game $\nph_G^A$ has a perfect strategy (of the corresponding class)
one does not need to necessarily resort to the fact that $\Psi$ has to be a homomorphism. 

\medskip

The next proposition shows the combinatorial meaning of (\ref{eq_intoperp}).

\begin{proposition}
\label{kraus-ops-on-S_G}
Let $X$ and $A$ be finite sets, $G$ be a graph with vertex set $X$, and $\mathcal{N}$ be a von Neumann algebra.
Let $\pi : M_A \to \mathcal{D}_X \otimes \mathcal{N}$ be a unital *-homomorphism with Kraus representation 
$\pi(T) = \sum_{i=1}^m M_iTM_i^*$ and write 
$\pi(\epsilon_{a,b}) = \sum_{x \in X} \epsilon_{x,x} \otimes r_{x,a,b}$, where $r_{x,a,b}\in \cl N$, $x\in X$, $a,b\in A$. 
The following are equivalent:
\begin{itemize}
\item[(i)] condition (\ref{eq_intoperp}) holds;
\item[(ii)]
if $v \sim w$ in $G$, then $\sum_{a,b\in A} r_{v,a,b}r_{w,b,a} = 0$.
\end{itemize}
\end{proposition}

\begin{proof}
Let $v \sim w$ in $G$, and define
$$R_{c,i,j} = \sum_{b\in A} \epsilon_{c,b}M_i^*(\epsilon_{v,w} \otimes 1_{\mathcal{N}})M_j\epsilon_{b,c}, \ \ 
c\in A, i,j\in [m].$$
Set 
$$T_{c,i}=\sum_{a \in A} \epsilon_{c,a} F_i(\epsilon_{v,v} \otimes e_{w,a,c}).$$
We have
\begin{align}
&
\hspace{-0.6cm}\sum_{c\in A}\sum_{i,j=1}^m R_{c,i,j}R_{c,i,j}^*\nonumber\\
&\hspace{-0.6cm}=
\sum_{a,b,c\in A}\sum_{i,j=1}^m
\epsilon_{c,a}M_i^*(\epsilon_{v,w} \otimes 1_{\mathcal{N}})M_j \epsilon_{a,c}\epsilon_{c,b} M_j^*(\epsilon_{w,v} \otimes 1_{\mathcal{N}})M_i \epsilon_{b,c} \nonumber\\
& \hspace{-0.6cm}= 
\sum_{a,b,c\in A}\sum_{i,j=1}^m
\epsilon_{c,a} M_i^*(\epsilon_{v,w} \otimes 1_{\mathcal{N}})M_j \epsilon_{a,b}
M_j^*(\epsilon_{w,v} \otimes 1_{\mathcal{N}})M_i\epsilon_{b,c} \nonumber\\
& \hspace{-0.6cm} = \sum_{a,b,c\in A}\sum_{i=1}^m 
\epsilon_{c,a} M_i^*(\epsilon_{v,w} \otimes 1_{\mathcal{N}}) \pi(\epsilon_{a,b})(\epsilon_{w,v} \otimes 1_{\mathcal{N}})M_i\epsilon_{b,c} \nonumber\\
& \hspace{-0.6cm} = \sum_{a,b,c\in A}\sum_{i=1}^m
\epsilon_{c,a}M_i^* (\epsilon_{v,v} \otimes r_{w,a,b})M_i\epsilon_{b,c} \nonumber\\
& \hspace{-0.6cm} = \sum_{a,b,c\in A}\sum_{i=1}^m 
\epsilon_{c,a}M_i^*(\epsilon_{v,v} \otimes r_{w,a,c})(\epsilon_{v,v} \otimes r_{w,c,b})M_i\epsilon_{b,c} \nonumber\\
& \hspace{-0.6cm} = \sum_{c\in A}\sum_{i=1}^m
T_{c,i}T_{c,i}^*, \label{sum-r-cij}
\end{align}
On the other hand, 
\begin{align}
\sum_{c\in A}\sum_{i=1}^m
T_{c,i}^*T_{c,i}
& = \sum_{a,b,c\in A}\sum_{i=1}^m
(\epsilon_{v,v} \otimes r_{w,c,a})M_i \epsilon_{a,c}\epsilon_{c,b}M_i^*(\epsilon_{v,v} \otimes r_{w,b,c}) \nonumber\\
& = \sum_{a,b,c\in A}\sum_{i=1}^m
(\epsilon_{v,v} \otimes r_{w,c,a})M_i\epsilon_{a,b}M_i^*(\epsilon_{v,v} \otimes r_{w,b,c}) \nonumber\\
& = \sum_{a,b,c\in A} 
(\epsilon_{v,v} \otimes r_{w,c,a})(\epsilon_{v,v} \otimes r_{v,a,b})(\epsilon_{v,v} \otimes r_{w,b,c}) \nonumber\\
&=\epsilon_{v,v} \otimes \sum_{a,b,c\in A} r_{w,c,a}r_{v,a,b}r_{w,b,c}. \label{sum-t-ci}
\end{align}
Since $\pi$ is a *-homomorphism, 
\begin{equation}
\sum_{a,b,c\in A} r_{w,c,a}r_{v,a,b}r_{w,b,c}=\left(\sum_{a,c\in A} r_{w,c,a}r_{v,a,c}\right)\left(\sum_{a,c\in A} r_{w,c,a}r_{v,a,c}\right)^*. \label{updated-sum-t-ci}
\end{equation}
Considering equations (\ref{sum-r-cij})-(\ref{updated-sum-t-ci}), it follows that condition (i) is equivalent to having 
$R_{c,i,j}=0$ for all $c\in A$ and all $i,j\in [m]$. 
The latter condition is in turn equivalent to the condition 
$\Tr(M_i^*(\epsilon_{v,w} \otimes 1_{\mathcal{N}})M_j) = 0$ for all $i,j\in [m]$. 
The proof is complete. 
\end{proof}


\subsection{Algebraic versions of the orthogonal rank}

Recall that the orthogonal rank $\xi(G)$ of $G$ is the smallest 
$k\in \bb{N}$ for which there exists an \emph{orthogonal representation} of $G$ in $\bb{C}^k$, 
that is, a collection $(\xi_x)_{x\in X}$ of unit vectors in $\bb{C}^k$ such that 
$$x\sim y \ \Longrightarrow \ \langle \xi_x,\xi_y\rangle = 0.$$
In this subsection, we discuss algebraic and C*-algebraic versions of the parameter $\xi(G)$. 
To place this into context, we define 
the \emph{relaxed classical-to-quantum colouring game} as the game
$\psi_G^A : \cl P^{\rm cl}_{XX}\to \cl P_{AA}$ determined by the requirements
$$\psi_G^A(\epsilon_{x,x}\otimes \epsilon_{y,y}) = 
\begin{cases} 
J_A^{\perp} & \mbox{if } x\sim y,\\
I_{AA} &\mbox{otherwise}.
\end{cases}$$
Let ${\rm x}\in \{{\rm loc},{\rm q}, {\rm qa}, {\rm qc}\}$.
We consider the following two parameters:
$$\xi_{\rm x}(G) = \min\{|A| : \mbox{there exists a perfect tracial }{\rm x}\mbox{-strategy for } \psi_G^A\},$$ 
which we call the \emph{relaxed orthogonal ${\rm x}$-rank} of $G$, and 
$$\xi_{\rm x}'(G) = \min\{|A| : \mbox{there exists a perfect }{\rm x}\mbox{-strategy for } \nph_G^A\},$$ 
which we call the \emph{orthogonal ${\rm x}$-rank} of $G$
(we set $\xi_{\rm x}(G) = \infty$ if there is no perfect strategy for $\nph_G^A$ for any $A$).
These parameters were introduced in \cite[Subsection 10.1]{tt-QNS} as 
quantum versions of the orthogonal rank. 
We note the following:
\begin{itemize}
    \item[(i)] Since $\nph_G^A$ is more restrictive than $\psi_G^A$, we have that
$\xi_{\rm x}(G)\leq \xi'_{\rm x}(G)$;

\item[(ii)] By \cite[Proposition 10.3]{tt-QNS}, we have $\xi_{\rm loc}(G) = \xi(G)$.
On the other hand, if $|A| > 1$ then $\xi_{\rm loc}'(G) = \infty$;

\item[(iii)] 
By \cite[Proposition 10.5]{tt-QNS}, $\xi'_{\rm q}(K_{d^2}) = \xi'_{\rm qc}(K_{d^2}) = d$, and hence $\xi'_{\rm q}(G) \leq [\sqrt{|X|}]+1$.  By Corollary \ref{xiq}, $\xi_q(K_{d^2})=d$.
\end{itemize}

Taking into account Remark \ref{r_Ecq}, we see that the ideal $\frak{I}(\nph_G^A)$ of $\frak{B}_{X,A}$ is given by 
$$\frak{I}(\nph_G^A) = \left\langle \left\{\sum_{a,b\in A} e_{x,a,b} e_{y,b,a} : x\sim y \right\}\right\rangle,$$
that is, $\frak{B}(\nph_G^A)$ is  the universal *-algebra generated by matrix unit systems
$(e_{x,a,a'})_{a,a'\in A}$, $x\in X$, subject to the relations 
$\sum_{a,b\in A} e_{x,a,b} e_{y,b,a} = 0$ whenever $x\sim y$.
Similarly, $\cl B(\nph_G^A)$ is the universal C*-algebra generated by such matrix unit systems, subject to these relations.

Corollary \ref{c_trcstog} implies the following  characterisations:

\begin{corollary}\label{c_chor}
Let $G$ be a graph with vertex set $X$. Then
\begin{itemize}
\item[(i)] 
The quantum commuting colourings of $G$ correspond to traces of $\cl B(\nph_G^A)$. In particular, 
$$\xi_{\rm qc}'(G) = \min\{|A| : \cl B(\nph_G^A) \mbox{ possesses a tracial state}\},$$ 
and 


\item[(ii)] 
The quantum colourings of $G$ correspond to finite dimensional traces of $\cl B(\nph_G^A)$. In particular, 
$$\xi_{\rm q}'(G) = \min\{|A| : \cl B(\nph_G^A) \mbox{ possesses a finite dim. *-representation}\}.$$
\end{itemize}
\end{corollary}

\begin{proof}
Suppose that $\tau$ is a tracial state on 
$\cl C_{X,A}$ that 
annihilates the generators $A_{x,y}=\sum_{a,b\in A}e_{x,a,b}e_{y,b,a}$, $x\sim y$, of $\frak{I}(\nph_G^A)$.
Note that 
$$A_{x,y}^*A_{x,y}=\sum_{a,b,c,d\in A}e_{y,a,b}e_{x,b,a}e_{x,c,d}e_{y,d,c}=\sum_{b,c,d\in A}e_{y,c,b}e_{x,b,d}e_{y,d,c};$$ 
it follows that 
$$\tau(A_{x,y}^*A_{x,y})=|A|\tau\left(\sum_{b,d\in A}e_{x,b,d}e_{y,d,b}\right)=0.$$
Combining this with the Cauchy-Schwartz inequality we obtain the statements. 
\end{proof}


\begin{definition}\label{d_orth}
\begin{itemize}
\item[(i)] The \emph{algebraic orthogonal rank} $\xi_{\rm alg}(G)$ is the smallest cardinality of a set $A$ for which 
$\frak{B}(\nph_G^A)\neq \{0\}$; if such $A$ does not exist, set $\xi_{\rm alg}(G) = \infty$;
\item[(ii)] 
The \emph{C*-algebraic orthogonal rank} $\xi_{C^*}(G)$ is the smallest cardinality of a set $A$ for which 
$\cl B(\nph_G^A)\neq \{0\}$; if such $A$ does not exist, set $\xi_{C^*}(G) = \infty$.
\end{itemize}
\end{definition}

\begin{proposition}\label{p_bound}
Let $G$ be a graph with vertex set $X$. Then $\xi_{C^*}(G) \geq \sqrt{\frac{|X|}{\theta(G)}}$.
Moreover, $\xi_{C^*}(K_{d^2}) = d$.
\end{proposition}

\begin{proof}
If $\xi_{C^*}(G) = \infty$ then the inequality is trivial; assume hence that $\cl B(\nph_G^A)\neq \{0\}$.
Since $\cl B(\nph_G^A)$ is separable, it possesses a faithful state $s$. 
Let $\pi$ be the corresponding GNS representation and $\xi$ the corresponding cyclic vector. 
Set $E_{x,a,b} = \pi(e_{x,a,b})$ and $\xi_{x,a,b} = E_{x,a,b}\xi$, $x\in X$, $a,b\in A$.
The proof of the inequality is now concluded in the same way as the proof of 
\cite[Proposition 10.5]{tt-QNS}. 

For the equality, realise
$A = \mathbb Z_d = \{0,1,\ldots, d-1\}$ and let $X = A\times A$. 
Let $\zeta$ be a primitive $|A|$-th root of unity and, 
for $x = (a',b')$ and $y = (a'',b'')\in X$, set 
$$E_{x,z,z'} = \zeta^{(z'-z)b'}e_{z-a'}e_{z'-a'}^*\in M_A, \ \ x = (a',b')\in X, 
z,z'\in A.$$
For $x = (a',b')$ and $y = (a'',b'')$ with $x\neq y$, we have 
\begin{eqnarray*}
\sum_{z,z'\in A} E_{x,z,z'}E_{y,z',z} 
& = & 
\sum_{z,z'\in A} \zeta^{(z'-z)b'} \zeta^{(z-z')b''} (e_{z-a'}e_{z'-a'}^*) (e_{z'-a''}e_{z-a''}^*)\\
& = & 
\sum_{z,z'\in A} \delta_{z'-a',z'-a''} \delta_{z-a',z-a''} \zeta^{(z'-z)(b'-b'')} I
= 0.
\end{eqnarray*}
In addition, 
\begin{eqnarray*}
E_{x,z,z'}E_{x,z',z''} 
& = & 
\zeta^{(z'-z)b'} \zeta^{(z''-z')b'} (e_{z-a'}e_{z'-a'}^*) (e_{z'-a'}e_{z''-a'}^*)\\
& = & 
\zeta^{(z''-z)b'} e_{z-a'}e_{z''-a'}^*;
\end{eqnarray*}
thus, $\cl B(\nph_{K_{d^2}}^A)$ is non-trivial. 
\end{proof}

As the next proposition shows, 
the algebraic orthogonal rank can be strictly smaller than the C*-algebraic one. 

\begin{proposition}\label{alg_orth}
$\xi_{\rm alg}(K_{d^2})=2$ for all $d \geq 2$.
\end{proposition}

\begin{proof}
We first show that $\xi_{\rm alg}(K_{d^2}) \leq 2$. The case of $d=2$ follows from Proposition \ref{p_bound}, so we assume that $d \geq 3$.
By Proposition \ref{p_bound}, 
the algebra of the (classical) $4$-colouring game for $K_{d^2}$ is non-zero. Hence, there are self-adjoint idempotents $p_{v,w}$ in a non-zero, unital *-algebra $\mathcal{A}$, for $1 \leq v \leq d^2$ and $1 \leq w \leq 4$, such that $\sum_{w=1}^4 p_{v,w}=1$ for all $v$, $p_{v,w}p_{v,z}=0$ if $w \neq z$, and 
\begin{equation}\label{eq_puw=0}
p_{u,w}p_{v,w}=0, \ \ \ u \neq v. 
\end{equation}
By Proposition \ref{p_bound}, the algebra $\mathcal{B}(\varphi_{K_4}^A)$ is non-zero when $|A|=2$. Hence, there are elements $e_{x,a,b}$ in a unital $*$-algebra $\mathcal{B}$, for $1 \leq x \leq 4$ and $1 \leq a,b \leq 2$, such that $e_{x,a,b}e_{x,c,d}=\delta_{b,c}e_{x,a,d}$, $e_{x,a,b}^*=e_{x,b,a}$, $\sum_{a=1}^2 e_{x,a,a}=1$ and 
\begin{equation}\label{eq_ab=1}
\sum_{a,b=1}^2 e_{x,a,b}e_{y,b,a}=0, \ \ \ x \neq y.
\end{equation}
For $1 \leq v \leq d^2$ and $1 \leq a,b \leq 2$, define
\[ f_{v,a,b}=\sum_{w=1}^4 p_{v,w} \otimes e_{w,a,b} \in \mathcal{A} \otimes \mathcal{B}.\]
We will show that the elements $f_{v,a,b}$ satisfy the requirements of the generators for the classical-to-quantum colouring game for $K_{d^2}$ with $|A|=2$. Observe that
\begin{align*}
f_{v,a,b}f_{v,c,d}&=\sum_{w,z=1}^4 p_{v,w}p_{v,z} \otimes e_{w,a,b}e_{z,c,d}
=\sum_{w=1}^4 p_{v,w} \otimes e_{w,a,b}e_{w,c,d} \\
&=\delta_{b,c} \sum_{w=1}^4 p_{v,w} \otimes e_{w,a,d}=\delta_{b,c}f_{v,a,d}.
\end{align*}
Since $p_{v,w}^*=p_{v,w}$ and $e_{w,a,b}^*=e_{w,b,a}$, we have  $f_{v,a,b}^*=f_{v,b,a}$. In addition, 
\[ \sum_{a=1}^2 f_{v,a,a}
= 
\sum_{a=1}^2 \sum_{w=1}^4 p_{v,w} \otimes e_{w,a,a}
=\sum_{w=1}^4 p_{v,w} \otimes 1=1 \otimes 1.\]
Lastly, using (\ref{eq_puw=0}) and (\ref{eq_ab=1}), assuming that $u\neq v$, we have 
\begin{align*}
\sum_{a,b=1}^2 f_{u,a,b}f_{v,b,a}
&= \sum_{a,b=1}^2 \sum_{w,z=1}^4 p_{u,w}p_{v,z} \otimes e_{w,a,b}e_{z,b,a} \\
&=\sum_{w=1}^4 p_{u,w}p_{v,w} \otimes 1=0.
\end{align*}
Thus, there is a unital *-homomorphism from $\mathcal{B}(\varphi_{K_{d^2}}^A)$ to $\mathcal{A} \otimes \mathcal{B}$ with $|A|=2$, so $\xi_{\rm alg}(K_{d^2}) \leq 2$.

It remains to show that $\xi_{\rm alg}(K_{d^2}) \geq 2$. To this end, we show that 
$\mathfrak{B}(\varphi_{K_{d^2}}^A)=\{0\}$ if $|A|=1$. When $|A|=1$, 
the relations defining $\mathfrak{B}(\varphi_{K_{d^2}}^A)$ reduce to having generators $q_{v,a,a}$, $a \in A$, $1 \leq v \leq d^2$, such that $\sum_{a \in A} q_{v,a,a}=1$, $q_{v,a,a}^*=q_{v,a,a}$ and $q_{v,a,a}q_{w,a,a}=0$ for $v \neq w$. Since $|A|=1$, the first relation implies that $q_{v,a,a}=1$ for all $v,a$. Then since $n \geq 2$, we may choose $1 \leq v,w \leq d^2$ with $v \neq w$. Since we must have $q_{v,a,a}q_{w,a,a}=0$, it follows that $1=1^2=0$ in 
$\mathfrak{B}(\varphi_{K_{d^2}}^A)$. Hence, $\xi_{\rm alg}(K_{d^2}) \geq 2$.
\end{proof}



\begin{thebibliography}{99}


\bibitem{delaroche}
{\sc C.  Anantharaman-Delaroche},
{\it On ergodic theorems for free group actions on noncommutative spaces}, 
{\rm Probab. Theory Rel. Fields 135 (2006), 520-546}.


\bibitem{bks}
\textsc{A. Bochniak, P. Kasprzak and P. So\l tan},
{\it Quantum correlations on quantum spaces},
{\rm preprint (2021), arXiv:2105.07820}.


\bibitem{btw}
{\sc G. Boreland, I. G. Todorov and A. Winter},
{\it Sandwich theorems and capacity bounds for non-commutative graphs},
{\rm J. Combin. Theory Ser. A 177 (2021), 105302, 39 pp.}


\bibitem{bcehpsw}
\textsc{M. Brannan, A. Chirvasitu, K. Eifler, S. Harris, V. I. Paulsen, X. Su and M. Wasilewski}, 
{\it Bigalois extensions and the graph isomorphism game},
{\rm Comm. Math. Phys. 375 (2020), no. 3, 1777-1809}.


\bibitem{bran-gan-har}
\textsc{M. Brannan, P. Ganesan and S. J. Harris},
{\it The quantum-to-classical graph homomorphism game},
{\rm preprint (2020), arXiv:2009.07229}.



\bibitem{bo}
{\sc N. P. Brown and N. Ozawa},
{\it C*-algebras and finite-dimensional approximations},
{\rm American Mathematical Society, 2008}.

\bibitem{cjppg} 
{\sc  T. Cooney, M. Junge, C. Palazuelos and D. P\'{e}rez-Garc\'{i}a}, 
{\it Rank-one quantum games}, 
{\rm Comput. Complexity 24 (2015), no. 1, 133-196}.
 


\bibitem{dlcdn}
{\sc G. De las Cuevas, T. Drescher and T. Netzer},
{\it Quantum magic squares: dilations and their limitations},
{\rm  J. Math. Phys. 61 (2020), no. 11, 111704, 15 pp.}


\bibitem{dsw}
{\sc R. Duan, S. Severini and A. Winter}, 
{\it Zero-error communication via quantum channels, non-commutative graphs and a quantum Lov\'{a}sz $\theta$ function},
{\rm IEEE Trans. Inf. Theory 59 (2013), no. 2, 1164-1174}.

\bibitem{dw}
{\sc R. Duan and A. Winter},
{\it No-signalling assisted zero-error capacity of quantum channels and an
information theoretic interpretation of the Lov\'{a}sz number},
{\rm IEEE Trans. Inf. Theory 62 (2016), no. 2, 891-914}.

\bibitem{dpp}
{\sc K. Dykema, V. I. Paulsen and J. Prakash},
{\it Non-closure of the set of quantum correlations via graphs},
{\rm Comm. Math. Phys. 365 (2019), no. 3, 1125-1142}.


\bibitem{hmps}
{\sc J. W. Helton, K. P. Meyer, V. I. Paulsen and M. Satriano},
{\it Algebras, synchronous games, and chromatic numbers of graphs},
{\rm New York J. Math. 25 (2019), 328-361}. 

\bibitem{jnvwy}
{\sc Z. Ji, A. Natarajan, T. Vidick, J, Wright and H. Yuen},
{\it MIP* = RE},
{\rm preprint (2020), arXiv:2001.04383}.

\bibitem{jnpp}
{\sc M. Junge, M. Navascues, C. Palazuelos, D. Perez-Garcia, V. Scholz and R. F. Werner}, 
{\it Connes' emnedding problem and Tsirelson's problem}, 
{\rm J. Math. Phys. 52, 012102 (2011), 12 pages}.

\bibitem{kirchberg}
{\sc E. Kirchberg},
{\it Discrete groups with Kazhdans property T and factorization property are residually finite},
{\rm Math. Ann. 299 (1994), 35-63}.


\bibitem{lo}  
{\sc L. Lov\'asz}, {\it On the Shannon capacity of a graph}, 
{\rm IEEE Trans. Inf. Theory 25 (1979), no. 1, 1-7}.


\bibitem{lmprsstw}
{\sc M. Lupini, L. Man\v{c}inska, V. I. Paulsen, D. E. Roberson, G. Scarpa, S. Severini, I. G. Todorov and A. Winter},
{\it Perfect strategies for non-local games},
{\rm Math. Phys. Anal. Geom. 23 (2020), no. 1, Paper No. 7, 31 pp.}

\bibitem{man-rob}
{\sc L. Man\v{c}inska  and D. E. Roberson},
{\it Graph homomorphisms for quantum players},
{\rm 9th Conference on the Theory of Quantum Computation,
              Communication and Cryptography, 
              LIPIcs. Leibniz Int. Proc. Inform. 27 (2014), 212-216}.


\bibitem{mr-2019}
{\sc M. Musat and M. R\o rdam},
{\it Factorizable maps and traces on the universal free product of matrix algebras},
{\rm preprint (2019), arXiv:1903.10182}.



\bibitem{mr}
{\sc M. Musat and M. R\o rdam}, 
{\it Non-closure of quantum correlation matrices and factorizable channels that require infinite dimensional ancilla. With an appendix by Narutaka Ozawa}, 
{\rm Comm. Math. Phys. 375 (2020), no. 3, 1761-1776}.


\bibitem{mrv1}
{\sc B. Musto, D. Reutter and D. Verdon}, 
{\it A compositional approach to quantum functions},
{\rm J. Math. Phys. 59 (2018), 081706, 42pp.}

\bibitem{ozawa}
{\sc N. Ozawa}, 
{\it About the Connes embedding conjecture: algebraic approaches}, 
{\rm Jpn. J. Math. 8 (2013), no. 1, 147-183}.


\bibitem{paulsen-lect}
{\sc V. I. Paulsen},
{\it Entanglement and non-locality}, 
{\rm Lecture Notes, University of Waterloo, 2016}.

\bibitem{psstw}
{\sc V. I. Paulsen, S. Severini, D. Stahlke, I. G. Todorov and A. Winter}, 
{\it Estimating quantum chromatic numbers}, 
{\rm J. Funct. Anal. 270 (2016), no. 6, 2188-2222}.

\bibitem{pr}
{\sc V. I. Paulsen and M. Rahaman},
{\it Bisynchronous games and factorizable maps},
{\rm preprint (2019), arXiv:1908.03842}.


\bibitem{ptt}
{\sc V. I. Paulsen, I. G. Todorov and M. Tomforde}, 
{\it Operator system structures on ordered spaces},
{\rm Proc. London Math. Soc.  102 (2011), 25-49}.

\bibitem{rv}
{\sc O. Regev and T. Vidick},
{\it Quantum XOR games}, 
{\rm ACM Trans. Comput. Theory 7 (2015), no. 4, Art. 15, 43 pp.}


\bibitem{soltan}
{\sc P. So\l tan}
{\it Quantum semigroups from synchronous games},
{\rm J. Math. Phys. 60 (2019), no. 4, 042203, 8pp.}

\bibitem{stahlke}
{\sc D. Stahlke},
{\it Quantum zero-error source-channel coding and non-commutative graph theory},
{\rm IEEE Trans. Inform. Theory 62 (2016), no. 1, 554-577}.


\bibitem{tt-QNS}
{\sc I. G. Todorov and L. Turowska},
{\it  Quantum no-signalling correlations and non-local games},
{\rm preprint (2020), arXiv:2009.07016}.


\bibitem{trevisan}
{\sc L. Trevisan}, 
{\it On Khot's unique games conjecture},
{\rm Bull. Amer. Math. Soc. (N.S.) 49 (2012), no. 1, 91-111}.

\bibitem{watrous}
{\sc J. Watrous},
{\it The theory of quantum information},
{\rm Cambridge University Press, 2018}.

\bibitem{weaver}
{\sc N. Weaver},
{\it Quantum Graphs as Quantum Relations},
{\rm Jour. Geom. Anal. (2021), https://doi.org/10.1007/s12220-020-00578-w}.

\end{thebibliography}
\end{document}